\documentclass[a4paper]{article}

\usepackage{amsfonts}
\usepackage{amsthm}
\usepackage{amssymb}
\usepackage{amsmath}
\usepackage{upref}
\usepackage{mathrsfs}
\usepackage{color}
\usepackage[citecolor=blue,colorlinks=true]{hyperref}
\usepackage{enumitem}
\usepackage{graphicx}
\usepackage[a4paper]{geometry}

\theoremstyle{plain} 
\newtheorem{theorem}{Theorem}[section]
\newtheorem{proposition}[theorem]{Proposition}
\newtheorem{lemma}[theorem]{Lemma}
\newtheorem{corollary}[theorem]{Corollary}
\newtheorem{definition}[theorem]{Definition}

\newtheorem{remark}[theorem]{Remark}

\numberwithin{equation}{section}

\def\N{\mathbb{N}}
\def\Z{\mathbb{Z}}

\def\R{\mathbb{R}}

\def\ds{\displaystyle} 
\def\div{\mathrm{div}}
\def\curl{\mathrm{curl}}
\def\eqref#1{(\ref{#1})}

\def\ocirc#1{\ifmmode\setbox0=\hbox{$#1$}\dimen0=\ht0
    \advance\dimen0 by1pt\rlap{\hbox to\wd0{\hss\raise\dimen0
    \hbox{\hskip.2em$\scriptscriptstyle\circ$}\hss}}#1\else
    {\accent"17 #1}\fi}

\def\eps{\varepsilon}

\def\<{\langle}
\def\>{\rangle}
\def\F{\mathcal{F}}
\def\GG{\mathbf{G}}
\def\P{\mathbb{P}}
\def\E{\mathbb{E}}
\def\T{\mathbb{T}}

\DeclareMathOperator{\esssup}{ess\,sup}

\def\1{{1\hspace{-1.2mm}{\rm I}}}
\def\CQFD{\unskip\kern 6pt\penalty 500%
\raise -1pt\hbox{\vrule\vbox to 8pt{\hrule width 6pt\vfill\hrule}\vrule}}
\def\U{\mathbf{U}}
\def\FFF{\mathbf{F}}
\newcommand{\tec}{{\overset{\cdot}{}}}

\begin{document}

\title{Stochastic isentropic Euler equations}
\author{F. Berthelin\thanks{Laboratoire Dieudonn\'e,
UMR 7351 CNRS, Universit\'e de Nice Sophia Antipolis, Parc Valrose, 06108 Nice Cedex 02, France. 
Email: Florent.Berthelin@unice.fr ; and COFFEE (INRIA Sophia Antipolis)} and J. Vovelle\thanks{Universit\'e de Lyon ; CNRS ; Universit\'e Lyon 1, Institut Camille Jordan,  43 boulevard du 11 novembre 1918, F-69622 Villeurbanne Cedex, France. Email: vovelle@math.univ-lyon1.fr, partially supported by ANR STOSYMAP, ANR STAB and ERC NuSiKiMo}}
\maketitle

\begin{abstract} We study the stochastically forced system of isentropic Euler equations of gas dynamics with a $\gamma$-law for the pressure. We show the existence of martingale weak entropy solutions; we also discuss the exis\-ten\-ce and characterization of invariant measures in the concluding section.
\end{abstract}

\medskip


\tableofcontents

\section{Introduction}

In this paper, we study the stochastically forced system of isentropic Euler equations of gas dynamics with a $\gamma$-law for the pressure. \smallskip

Let $(\Omega,\F,\P,(\F_t),(\beta_k(t)))$ be a stochastic basis, let $\T$ be the one-dimensional torus, let $T>0$ and set $Q_T:=\T\times(0,T)$. We study the system
\begin{subequations}\label{stoEuler}
\begin{align}
&d\rho+\partial_x (\rho u) dt=0, &\mbox{ in }Q_T,\label{masse}\\
&d(\rho u)+\partial_x (\rho u^2+p(\rho)) dt=\Phi(\rho,u) dW(t),&\mbox{ in }Q_T,\label{impulsion}\\
&\rho =\rho_0, \quad \rho u=\rho_0 u_0,&\mbox{ in }\T\times\{0\},\label{IC}
\end{align}
\end{subequations}
where $p$ follows the $\gamma$-law
\begin{equation}
p(\rho)=\kappa\rho^\gamma,\quad \kappa=\frac{\theta^2}{\gamma},\quad \theta=\frac{\gamma-1}{2},
\label{gammalaw}\end{equation}
for $\gamma>1$, $W$ is a cylindrical Wiener process and $\Phi(0,u)=0$. Therefore the noise affects the momentum equation only and vanishes in vacuum regions. Our aim is to prove the existence of solutions to \eqref{stoEuler} for general initial data (including vacuum), \textit{cf.} Theorem~\ref{th:martingalesol} below.
\medskip

There are to our knowledge no existing results on stochastically forced systems of first-order conservation laws, with the exception of the papers by Kim, \cite{Kim11}, and Audusse, Boyaval, Goutal, Jodeau, Ung, \cite{ABGJU15}. In \cite{Kim11} the problematic is the possibility of global existence of \textit{regular} solutions to symmetric hyperbolic systems under suitable assumptions on the structure of the stochastic forcing term. In \cite{ABGJU15} is derived a shallow water system with a stochastic Exner equation as a model for the dynamics of sedimentary river beds.  On second-order stochastic systems, and specifically on the stochastic compressible Navier-Stokes equation\footnote{which, to be exact, is first-order in the density and second-order in the velocity}, different results have been  obtained recently, see the papers by Breit, Feireisl, Hofmanov{\'a}, Maslowski, Novotny, Smith, \cite{FeireislMaslowskiNovotny13,BreitHofmanova14,BreitFeireislHofmanova15,Smith15} (see also the older work by Tornare and Fujita, \cite{TornareFujita97}).\medskip

The \textit{incompressible} Euler equations with stochastic forcing terms have been studied in particular by Bessaih, Flandoli,\cite{Bessaih99,BessaihFlandoli99,Bessaih00,Bessaih08}, Capi{\'n}ski, Cutland, \cite{CapinskiCutland99}, Brze{\'z}niak, Peszat, \cite{BrzezniakPeszat01}, Cruzeiro, Flandoli, Malliavin, \cite{CruzeiroFlandoliMalliavin07}, Brze{\'z}niak, Flandoli, Maurelli, \cite{BrzezniakFlandoliMaurelli14}, Glatt-Holtz and Vicol, \cite{GlattHoltzVicol14}, Cruzeiro and Torrecilla, \cite{CruzeiroTorrecilla15}. We refer in particular to \cite{GlattHoltzVicol14} for results in space dimension $3$. \medskip

In the deterministic case, and in in space dimension $1$, the existence of weak entropy solutions to the isentropic Euler system has been proved by Lions, Perthame, Souganidis in \cite{LionsPerthameSouganidis96}. Let us mention also the anterior papers by Di Perna \cite{Diperna83a}, Ding, Chen, Luo \cite{DingChenLuo85}, Chen \cite{Chen86}, Lions, Perthame, Tadmor \cite{LionsPerthameTadmor94ki}. The uniqueness of weak entropy solutions is still an open question.
\medskip

For {\it scalar} non-linear hyperbolic equations with a stochastic forcing term, the theory has recently known a lot of developments. Well-posedness has been proved in different contexts and under different hypotheses and also with different techniques: by Lax-Oleinik formula (E, Khanin, Mazel, Sinai \cite{EKMS00}), Kruzhkov doubling of variables for entropy solutions (Kim \cite{Kim03}, Feng, Nualart \cite{FengNualart08}, Vallet, Wittbold \cite{ValletWittbold09}, Chen, Ding, Karlsen \cite{ChenDingKarlsen12}, Bauzet, Vallet, Wittbold \cite{BauzetValletWittbold12}), kinetic formulation (Debussche, Vovelle \cite{DebusscheVovelle10,DebusscheVovelle10revised}). Resolution in $L^1$ has been given in \cite{DebusscheVovelle14}. Let us also mention the works of Hofmanov\'a in this fields (extension to second-order scalar degenerate equations, convergence of the BGK approximation \cite{Hofmanova13b,DebusscheHofmanovaVovelle15,HofmanovaBGK15}) 
 and  the recent works by Hofmanov{\'a}, Gess, Lions, Perthame, Souganidis \cite{LionsPerthameSouganidis13a,LionsPerthameSouganidis13,LionsPerthameSouganidis14,
 GessSouganidis15,GessSouganidis14inv,Hofmanova15} on scalar conservation laws with quasilinear stochastic terms.
\medskip

We will show existence of martingale solutions to \eqref{stoEuler}, see Theorem~\ref{th:martingalesol} below. The procedure is standard: we prove the convergence of (subsequence of) solutions to the parabolic approximation to \eqref{stoEuler}. For this purpose we have to adapt the concentration compactness technique (\textit{cf.} \cite{Diperna83a,LionsPerthameSouganidis96}) of the deterministic case to the stochastic case. Such an extension has already been done for scalar conservation laws by Feng and Nualart \cite{FengNualart08} and what we do is quite similar. The mode of convergence for which there is compactness, if we restrict ourselves to the sample variable $\omega$, is the convergence in law. That is why we obtain martingale solutions. There is a usual trick, the Gy\"ongy-Krylov characterization of convergence in probability, that allows to recover pathwise solutions once pathwise uniqueness of solutions is known (\textit{cf.} \cite{GyongyKrylov96}). However for the stochastic problem \eqref{stoEuler} (as it is already the case for the deterministic one), no such results of uniqueness are known.
\medskip

A large part of our analysis is devoted to the proof of existence of solutions to the parabolic approximation. What is challenging and more difficult than in the deterministic framework  for the stochastic parabolic problem is the issue of positivity of the density. We solve this problem by using a regularizing effect of parabolic equations with drifts and a bound given by the entropy, quite in the spirit of Mellet, Vasseur, \cite{MelletVasseur09}, \textit{cf.} Theorem~\ref{th:uniformpositive}. Then, the proof of convergence of the parabolic approximation~\eqref{stoEulereps} to Problem~\eqref{stoEuler} is adapted from the proof in the deterministic case to circumvent two additional difficulties: 
\begin{enumerate}
\item there is a lack of compactness with respect to $\omega$; one has to pass to the limit in some stochastic integrals,
\item there are no ``uniform in $\eps$" $L^\infty$ bounds on solutions (here $\eps$ is the regularization parameter in the parabolic problem~\eqref{stoEulereps}).
\end{enumerate}
Problem 1. is solved by use of convergence in law and martingale formulations, Problem 2. is solved thanks to higher moment estimates
(see \eqref{estimmomenteps2} and \eqref{corestimgradientepsrho}-\eqref{corestimgradientepsu}). 
We will give more details about the main problematic of the paper in Section~\ref{sec:prob}, after our framework has been introduced more precisely. 
Note that Problem 2. also occurs in the resolution of the isentropic Euler system for flows in non-trivial geometry, as treated by Le Floch, Westdickenberg, \cite{LeFlochWestdickenberg07}.

\section{Notations and main result}

\subsection{Stochastic forcing}\label{sec:stoForce}

Our hypotheses on the stochastic forcing term $\Phi(\rho,u) W(t)$ are the following ones. We assume that $W=\sum_{k\geq 1}\beta_k e_k$ where the $\beta_k$ are independent Brownian processes and $(e_k)_{k\geq 1}$ is a complete orthonormal system in a Hilbert space $\mathfrak{U}$. For each $\rho\geq 0, u\in\R$, $\Phi(\rho,u)\colon \mathfrak{U}\to L^2(\T)$ is defined by 
\begin{equation}\label{sigmakstar}
\Phi(\rho,u)e_k=\sigma_k(\cdot,\rho,u)=\rho\sigma_k^*(\cdot,\rho,u),
\end{equation} 
where $\sigma_k^*(\cdot,\rho,u)$ is a $1$-periodic continuous function on $\R$. More precisely, we assume
$\sigma_k^*\in C(\T_x\times\R_+\times\R)$ and the bound
\begin{equation}\label{A0}
\GG(x,\rho,u):=\bigg(\sum_{k\geq 1}|\sigma_k(x,\rho,u)|^2\bigg)^{1/2}\leq {A_0}\rho\left[1+u^{2}+\rho^{2\theta }\right]^{1/2},
\end{equation}
for all $x\in\T$, $\rho\geq 0$, $u\in\R$, where ${A_0}$ is some non-negative constant. Depending on the statement, we will sometimes also make the following localization hypothesis: for $\varkappa>0$, denote by $z=u-\rho^\theta$, $w=u+\rho^\theta$ the Riemann invariants for \eqref{stoEuler} and by $\Lambda_\varkappa$ the domain
\begin{equation}\label{invariantregion}
\Lambda_\varkappa=\left\{(\rho,u)\in\R_+\times\R; -\varkappa\leq z\leq w\leq \varkappa\right\}.
\end{equation}
We will establish some of our results (more precisely: the resolution of the approximate parabolic Problem~\eqref{stoEulereps}) under the hypothesis that there exists $\varkappa>0$ such that 
\begin{equation}\label{Trunc}
\mathrm{supp}(\GG)\subset \T_x\times\Lambda_\varkappa.
\end{equation}

We define the auxiliary space $\mathfrak{U}_0\subset\mathfrak{U}$ by
\begin{equation}\label{defUUU0}
\mathfrak{U}_0=\bigg\{v=\sum_{k\geq1}\alpha_k e_k;\;\sum_{k\geq1}\frac{\alpha_k^2}{k^2}<\infty\bigg\},
\end{equation}
and the norm
$$
\|v\|^2_{\mathfrak{U}_0}=\sum_{k\geq1}\frac{\alpha_k^2}{k^2},\qquad v=\sum_{k\geq1}\alpha_k e_k.
$$
The embedding $\mathfrak{U}\hookrightarrow\mathfrak{U}_0$ is then an Hilbert-Schmidt operator. Moreover, trajectories of $W$ are $\P$-a.s. in $C([0,T];\mathfrak{U}_0)$ (see Da Prato, Zabczyk \cite{DaPratoZabczyk92}). We use the path space $C([0,T];\mathfrak{U}_0)$ to recover the cylindrical Wiener process $W$ in certain limiting arguments, \textit{cf.} Section~\ref{subsec:compact} for example.

\subsection{Notations}\label{sec:notations}

We denote by 
\begin{equation}\label{defUUU}
\U=\begin{pmatrix}\rho\\ q\end{pmatrix}, \quad\FFF(\U)=\begin{pmatrix} q\\ \frac{q^2}{\rho}+p(\rho)\end{pmatrix},\quad q=\rho u,
\end{equation}
the $2$-dimensional unknown and flux of the conservative part of the problem. We also set
$$
\psi_k(\U)=\begin{pmatrix}0 \\ \sigma_k(\U)\end{pmatrix},\quad \mathbf{\Psi}(\U)=\begin{pmatrix}0 \\ \Phi(\U)\end{pmatrix}.
$$
With the notations above, \eqref{stoEuler} can be more concisely rewritten as the following stochastic first-order system
\begin{equation}\label{stoEuler1}
d\U+\partial_x\FFF(\U)dt=\mathbf{\Psi}(\U)dW(t).
\end{equation}

If $E$ is a space of real-valued functions on $\T$, we will denote $\U(t)\in E$ instead of $\U(t)\in E\times E$ when this occurs. Similarly, we will denote $\U\in E$ instead of $\U\in E\times E$ if $E$ is a space of real-valued functions on $\T\times[0,T]$ (see the statement of Definition~\ref{def:entropysol} as an example).
\medskip

We denote by $\mathcal{P}_T$ the predictable $\sigma$-algebra on $\Omega\times[0,T]$ generated by $(\mathcal{F}_t)$. 
\medskip

We will also use the following notation in various estimates below:
$$
A=\mathcal{O}(1)B,
$$
where $A,B\in\R_+$, with the meaning $A\leq CB$ for a constant $C\geq 0$. In general, the dependence of $C$ over the data and parameters at stake will be given in detail, see for instance Theorem~\ref{th:existspatheps} below. We use the notation
$$
A\lesssim B
$$
with the same meaning $A\leq CB$, but when the constant $C\geq 0$ depends only on $\gamma$ and nothing else, $C$ being bounded for $\gamma$ in a compact subset of $[1,+\infty)$. In this last case, $C$ depends sometimes even not on $\gamma$ and is simply a numerical constant (see Appendix~\ref{app:regparabolic} for instance).
\medskip

\subsection{Entropy Solution}

In relation with the kinetic formulation for~\eqref{stoEuler} in \cite{LionsPerthameTadmor94ki}, there is a family of entropy functionals
\begin{equation}
\eta(\U)=\int_{\R} g(\xi)\chi(\rho,\xi-u)d\xi, \quad \textrm{ with } q=\rho u,
\label{entropychi}\end{equation}
for \eqref{stoEuler}, where 
\begin{equation*}
\chi(\U)=c_\lambda(\rho^{2\theta}-u^2)^\lambda_+,\quad\lambda=\frac{3-\gamma}{2(\gamma-1)},
\quad c_\lambda=\left(\int_{-1}^1 (1-z^2)_+^\lambda \,dz\right)^{-1},
\end{equation*}
$s_+^\lambda:=s^\lambda\mathbf{1}_{s>0}$. Indeed, if $g\in C^2(\R)$ is a convex function, then $\eta$ is of class $C^2$ on the set
$$
\mathcal{U}:=\left\{\U=\begin{pmatrix}\rho\\ q\end{pmatrix}\in\R^2;\rho>0\right\}
$$
and $\eta$ is a convex function of the argument $\U$. Formally, by the It\={o} Formula, solutions to \eqref{stoEuler} satisfy
\begin{equation}
d\E\eta(\U)+\partial_x \E H(\U) dt=\frac12 \E\partial^2_{qq} \eta(\U)\GG^2(\U)dt,
\label{entropyeq}\end{equation}
where the entropy flux $H$ is given by
\begin{equation}
H(\U)=\int_\R g(\xi)[\theta\xi+(1-\theta)u]\chi(\rho,\xi-u)d\xi, \quad \textrm{ with } q=\rho u.
\label{entropychiflux}\end{equation}
Note that, by a change of variable, we also have
\begin{equation} \label{eqeta}
\eta(\U)=\rho c_\lambda \int_{-1}^1 g\left(u+z\rho^{\theta}\right) (1-z^2)^\lambda_+ dz 
\end{equation}
and
\begin{equation} \label{eqH}
H(\U)=\rho c_\lambda \int_{-1}^1 g\left(u+z\rho^{\theta}\right) \left(u+z\theta\rho^{\theta}\right)(1-z^2)^\lambda_+ dz.
\end{equation}
In particular, for $g(\xi)=1$ we obtain the density $\eta_0(\U)=\rho$.
To $g(\xi)=\xi$ corresponds the impulsion $\eta(\U)=q$ and
to $g(\xi)=\frac12 \xi^2$ corresponds the energy
\begin{equation}\label{entropyenergy}
\eta_E(\U)=\frac12\rho u^2+\frac{\kappa}{\gamma-1}\rho^\gamma.
\end{equation}
Note the form of the energy, in particular the fact that the hypothesis \eqref{A0} on the noise gives a bound
\begin{equation}\label{noisebyenergy}
\GG^2(x,\U)=\sum_{k\geq 1}|\Phi(\rho,u)e_k(x)|^2\leq \rho A_0^\sharp(\eta_0(\U)+\eta_E(\U)),
\end{equation}
for a constant $A_0^\sharp$ depending on ${A_0}$ and $\gamma$ (recall that $\eta_0(\U):=\rho$). If \eqref{entropyeq} is satisfied with an inequality $\leq$, then formally \eqref{noisebyenergy} and the Gronwall Lemma give a bound on $\E\int_\T(\eta_0+\eta_E)(\U)(t) dx$ in terms of $\E\int_T(\eta_0+\eta_E)(\U)(0) dx$. Indeed, we have $\partial^2_{qq}\eta_{E}(\U)=\frac{1}{\rho}$ and, therefore,
\begin{equation*}
\E\partial^2_{qq} \eta_E(\U)\GG^2(\U)\leq A_0^\sharp\E(\eta_0(\U)+\eta_E(\U)).
\end{equation*}
\medskip

We will prove rigorously uniform bounds for approximate parabolic solutions
in Section~\ref{sec:entropybounds}. The above formal computations are however sufficient for the moment to introduce the following definition.

\begin{definition}[Entropy solution] Let $\rho_0, u_0\in L^2(\T)$ with $\rho_0\geq 0$ a.e. and let $\U_0=\begin{pmatrix}\rho_0\\ \rho_0 u_0\end{pmatrix}$ satisfy
$$
\int_\T \rho_0(1+u^{2}_0+\rho_0^{2\theta}) dx<+\infty.
$$
A process $(\U(t))$ with values in $W^{-2,2}(\T)$ is said to be a pathwise weak entropy solution to \eqref{stoEuler} with initial datum $\U_0$ if 
\begin{enumerate}
\item almost surely, $\U\in C([0,T],W^{-2,2}(\T))$ and $(\U(t))$ is predictable,
\item almost surely, $\U$ is an integrable function on $Q_T$,
\item the bound 
\begin{equation}\label{boundEntropyDef}
\E\, \underset{0\leq t\leq T}{\esssup}\,\int_{\T}\eta(\U(x,t))dx<+\infty,
\end{equation}
is satisfied for $\eta=\eta_E$, the energy defined in \eqref{entropyenergy},
\item $\Phi(\U)$ satisfies
\begin{equation}\label{predictPhi}
\Phi(\U)\in L^2\big(\Omega\times[0,T],\mathcal{P}_T,d\P\times dt;L_2(\mathfrak{U};L^2(\T))\big),
\end{equation}
where $L_2(\mathfrak{U};L^2(\T))$ is the space of Hilbert-Schmidt operators from $\mathfrak{U}$ into $L^2(\T)$,
\item for any $(\eta,H)$ given by \eqref{entropychi}-\eqref{entropychiflux}, where $g\in C^2(\R)$ is convex and subquadratic\footnote{in the sense that $g$ satisfies \eqref{gsubquad}}, for all $t\in(0,T]$, for all non-negative $\varphi\in C^1(\T)$, and non-negative $\alpha\in C^1_c([0,t))$, the following entropy inequality is almost surely satisfied:
\begin{align}
&\int_0^t \big\langle \eta(\U)(s),\varphi\big\rangle\alpha'(s)+\big\langle H(\U)(s),\partial_x \varphi\big\rangle\alpha(s)\, ds\nonumber\\
&+\int_0^t\big\langle \GG^2(x,\U)\partial^2_{qq}\eta(\U),\varphi\big\rangle\alpha(s)\,ds+\big\langle \eta(\U_0),\varphi\big\rangle\alpha(0)\nonumber\\
&+\sum_{k\geq 1}\int_0^t \big\langle\sigma_k(x,\U)\partial_q\eta(\U),\varphi\big\rangle\alpha(s)\,d\beta_k(s)\geq 0.\label{Entropy}
\end{align}
\end{enumerate}
\label{def:entropysol}\end{definition}

\begin{remark} A pathwise weak entropy solution $\U$ is a priori a process $(\U(t))$ with values in $W^{-2,2}(\T)$, a space of distributions. In item \textit{2.} we require that, almost surely, $\U$ is an integrable function of $(t,x)$: $\rho(x,t)$ and $q(x,t)$ are defined a.e. To give a sense to $\eta_E(\U(x,t))$, we need to know $u(x,t)$, or to be able to interpret the quotient $\frac{q(x,t)^2}{\rho(x,t)}$.  We will prove the existence of a martingale weak entropy solution $\U$ to \eqref{stoEuler} (see Theorem~\eqref{th:martingalesol}) satisfying $u=0$ in the vacuum region $\rho=0$ (see \eqref{0Vacuum}). Note besides, to make an additional comment on Definition~\ref{def:entropysol}, that, with the choice $(\eta,H)(\U)=\pm(\rho,q)$, we infer from \eqref{Entropy} the weak formulation of Equation~\eqref{stoEuler}.
\end{remark}

\begin{remark} By \eqref{predictPhi}, the stochastic integral $t\mapsto\int_0^t\Phi(\U)(s) dW(s)$ is a well defined process taking values in $L^2(\T)$ (see \cite{DaPratoZabczyk92} for the details of the construction). There is a little redundancy here in the definition of entropy solutions since, apart from the predictability, the integrability property \eqref{predictPhi} will follow from \eqref{A0} and the bounds \eqref{boundEntropyDef}, \textit{cf.} \eqref{noisebyenergy}.
\end{remark}

In Definition~\ref{def:entropysol}, the notion of solution considered is weak in space-time, strong with respect to $\omega$. The following notion of solution is weak in $(x,t,\omega)$.

\begin{definition}[Martingale solution] Let $\rho_0, u_0 \in L^2(\T)$ with $\rho_0\geq 0$ a.e. and let $\U_0=\begin{pmatrix}\rho_0\\ \rho_0 u_0\end{pmatrix}$ satisfy
$$
\int_T \rho_0(1+u^{2}_0+\rho_0^{2\theta}) dx<+\infty.
$$
A martingale weak entropy solution to \eqref{stoEuler} with initial datum $\U_0$ is a multiplet
$$
(\tilde\Omega,\tilde\F,\tilde\P,(\tilde\F_t),\tilde W,\tilde\U),
$$
where $(\tilde\Omega,\tilde\F,\tilde\P)$ is a probability space, with filtration $(\tilde\F_t)$ satisfying the usual conditions, $\tilde W$ a $(\tilde\F_t)$-cylindrical Wiener process, and $(\tilde\U(t))$ defines, according to Definition~\ref{def:entropysol}, a pathwise weak entropy solution to~\eqref{stoEuler} with initial datum $\U_0$.
\label{def:martingalesol}\end{definition}

In summary, if after the substitution
\begin{equation}\label{substitution}
\big(\Omega,\mathcal{F},(\mathcal{F}_t),\P,W\big)\leftarrow\big(\tilde{\Omega},\tilde{\mathcal{F}},(\tilde{\mathcal{F}}_t),\tilde{\P},\tilde{W}\big),
\end{equation}
$\tilde\U$ is a pathwise weak entropy solution to \eqref{stoEuler}, then we say that $\tilde\U$ (or, to be more rigorous, $(\tilde\Omega,\tilde\F,\tilde\P,(\tilde\F_t),\tilde W,\tilde\U)$) is a martingale weak entropy solution to \eqref{stoEuler}. The substitution \eqref{substitution} leaves invariant the \textit{law} of the resulting process $(\U(t))$. The fact is that we are in most cases interested only in the law of the process. An example is the discussion on the large time behaviour and invariant measures given in Section~\ref{sec:conclusion}.

\begin{theorem}[Main result] Let $p\in\N$ satisfy $p\geq 4+\frac{1}{2\theta}$. Assume that the structure and growth hypothesis \eqref{A0} on the noise are satisfied. Let $\rho_0, u_0\in L^2(\T)$ with $\rho_0\geq 0$ a.e. and let $\U_0=\begin{pmatrix}\rho_0\\ \rho_0 u_0\end{pmatrix}$ satisfy
$$
\int_\T \rho_0(1+u^{4p}_0+\rho_0^{4\theta p}) dx<+\infty.
$$
Then there exists a martingale solution to \eqref{stoEuler} with initial datum $\U_0$.
\label{th:martingalesol}\end{theorem}

\subsection{Organization of the paper and main problematic}\label{sec:prob}

The paper is organized as follows.
In Section~\ref{sec:parabolicapproximation}, we prove the existence of strong solutions to the parabolic approximation of Problem~\eqref{stoEuler}, 
see Problem~\eqref{stoEulereps}. The parabolic approximation to Problem~\eqref{stoEuler} is a stochastic parabolic PDE with singularity at the 
state-value $\rho=0$. To get existence of a solution to \eqref{stoEulereps}, we use a priori estimates: some are naturally furnished by the entropy balance 
equations, see Corollary~\ref{cor:boundmoments}, Corollary~\ref{cor:boundgradient2tau}. These estimates are however of no use in the vacuum region $\{\rho=0\}$ 
(observe that, indeed, a factor $\rho$ is present in each of the estimates stated in Corollary~\ref{cor:boundmoments}, Corollary~\ref{cor:boundgradient2tau}). 
For the isentropic Euler system, an estimate still of use in the vacuum region is an $L^\infty$ estimate given by the invariance of some regions $\Lambda_\varkappa$ 
defined with the help of the Riemann invariants (see the definition of $\Lambda_\varkappa$ in \eqref{invariantregion}). In our stochastic setting, we can use such 
invariant regions provided the noise is compactly supported (but here the  $L^\infty$ estimates will be lost when $\eps \to 0$). 
This is what we assume, see hypothesis~\eqref{Trunceps}. We need crucially this estimate ``still of use in the vacuum" to prove the last a priori estimate necessary for the existence of a solution to the parabolic approximation~\eqref{stoEulereps}, which is the positivity of the density, see Section~\ref{sec:PositiveDensity}. The positivity results is obtained thanks to the regularizing effects of the heat equation. This is the subject of Appendix~\ref{app:boundfrombelow}. \medskip

All these a priori estimates are proved rigorously on an approximation of the solution to the parabolic approximation obtained by time splitting in Section~\ref{sec:ParabolicProblemSol}. Once the existence of solutions to the parabolic approximation of Problem~\eqref{stoEuler} has been proved, we want to take the limit on the regularizing parameter to obtain a martingale solution to \eqref{stoEuler}. As in the deterministic case \cite{Diperna83a,Diperna83b,
LionsPerthameSouganidis96}, we use the concept of measure-valued solution (Young measure) to achieve this.
In Section~\ref{sec:YoungMeasures} we develop the tools on Young measure (in our stochastic framework) which are required. This is taken in part (but quite different) from Section~4.3 in \cite{FengNualart08}. We also use the probabilistic version of Murat's Lemma from \cite[Appendix A]{FengNualart08}, to identify the limiting Young measure. This is the content of Section~\ref{sec:reductionYoung}, which requires two other fundamental tools: the a\-na\-ly\-sis of the consequences of the div-curl lemma in \cite[Section I.5]{LionsPerthameSouganidis96} and an identification result for densely defined martingales from \cite[Appendix A]{Hofmanova13b}. We obtain then the existence of a martingale solution to \eqref{stoEuler}. In Section~\ref{sec:conclusion} we discuss the existence of invariant measures to \eqref{stoEuler}. Besides, as explained above, we need at some point some bounds from below on solutions to ($1$-dimensional here) parabolic equations, which are developed in Appendix~\ref{app:boundfrombelow}. We also need 
some regularity results, with few variations, on the ($1$-dimensional) heat semi-group, and those are given in Appendix~\ref{app:regparabolic}.

\subsection*{Acknowledgements}

We thank warmly Martina Hofmanov\'a, for her help with Section~\ref{sec:parabolicapproximation}, and Franco Flandoli, who suggested us the use of the splitting method in Section~\ref{sec:parabolicapproximation}. We also thank an anonymous referee, whose earnest work helped us to improve our paper.

\section{Parabolic Approximation}\label{sec:parabolicapproximation}

For $\eps>0$, we consider the following second-order approximation to \eqref{stoEuler}

\begin{subequations}\label{stoEulereps}
\begin{align}
d{\U_\eps}+\partial_x\FFF({\U_\eps})dt&=\eps\partial^2_{xx}{\U_\eps} dt+\mathbf{\Psi}^\eps({\U_\eps})dW(t), \label{eq:stoEulereps}\\
\nonumber\\
{\U_\eps}_{|t=0}&={\U_\eps}_0.\label{IC:stoEulereps}
\end{align}
\end{subequations}

Recall that $\U$ and $\FFF(\U)$ are defined by
$$
\U=\begin{pmatrix}\rho\\ q\end{pmatrix}, \quad\FFF(\U)=\begin{pmatrix} q\\ \frac{q^2}{\rho}+p(\rho)\end{pmatrix}.
$$
Problem \eqref{stoEulereps} is a regularized version of Problem \eqref{stoEuler}: this is a parabolic regularization of \eqref{stoEuler} and we will also assume more regularity than in \eqref{stoEuler} on the coefficients of the noise (see \eqref{Lipsigmaeps}). More precisely, as in \eqref{stoEuler} we assume that there is no noise in the evolution equation for ${\rho}_\eps$: the first component of $\mathbf{\Psi}^\eps({\U_\eps})$ is $0$. For each given $\U$, the second component is the map $\Phi^\eps(\U)\colon \mathfrak{U}\to L^2(\T)$ given by 
$$
\left[\Phi^{\eps}(\rho,u)e_k\right](x)=\sigma^{\eps}_k(x,\rho,u),
$$
where $\sigma^\eps_k$ is a continuous function of its arguments. We assume (compare to \eqref{A0}) 
\begin{equation}\label{A0eps}
\GG^\eps(x,\rho,u):=\bigg(\sum_{k\geq 1}|\sigma_k^\eps(x,\rho,u)|^2\bigg)^{1/2}\leq {A_0}\rho\left[1+u^{2}+\rho^{2\theta}\right]^{1/2},
\end{equation}
for all $x\in\T$, $\U\in\R_+\times\R$. We will also assume that $\GG^\eps$ is supported in an invariant region: there exists $\varkappa_\eps>0$ such that 
\begin{equation}\label{Trunceps}
\mathrm{supp}(\GG^\eps)\subset \T_x\times\Lambda_{\varkappa_\eps},
\end{equation}
where the region $\Lambda_\varkappa$ is defined by \eqref{invariantregion}. Note that this gives \eqref{A0eps}, but with a constant ${A_0}$ depending on $\varkappa_\eps$: we have indeed
\begin{equation}\label{BoundTrunceps}
|\mathbf{G}^\eps(x,\U)|\leq M(\varkappa_\eps),
\end{equation}
for all $x\in\T$, $\U\in\R_+\times\R$. Note however that, in \eqref{A0eps}, ${A_0}$ is assumed independent on $\eps$. Eventually, we will assume that the following Lipschitz condition is satisfied:
\begin{equation}\label{Lipsigmaeps}
\sum_{k\geq 1}\left|\sigma_k^\eps(x,\U_1)-\sigma_k^\eps(x,\U_2)\right|^2\leq C(\eps,R)|\U_1-\U_2|^2,
\end{equation}
for all $x\in\T$, $\U_1,\U_1\in \mathcal{U}_R$, where $C(\eps,R)$ is a constant depending on $\eps$ and $R$. Here, for $R>1$, $\mathcal{U}_R$ denotes the set of $\U\in\R_+\times\R$ such that
\begin{equation}\label{defDR}
R^{-1}\leq\rho\leq R,\quad |q|\leq R.
\end{equation}
We also denote by $D_R$ be the set of functions $\U\in L^2(\T)$ such that $\U(x)\in \mathcal{U}_R$ for a.e. $x\in\T$. Note that $D_R$ is a closed subset of $L^2(\T)$.

\subsection{Pathwise solution to the parabolic problem}\label{sec:pathwiseParabolicProblem}

\begin{definition}[Bounded solution to the parabolic approximation] Let $\U_0\in L^\infty(\T)$ satisfy $\rho_0\geq c_0$ a.e. in $\T$, where $c_0>0$. Let $T>0$. Assume \eqref{A0eps}. A process $(\U(t))_{t\in[0,T]}$ with values in $(L^2(\T))^2$ is said to be a \textrm{bounded solution} to \eqref{stoEulereps} if it is a predictable process such that
\begin{enumerate}
\item almost surely, $\U \in C([0,T];L^2(\T))$,
\item $\U\in D_R$ with high probability, \textit{i.e.} for all $\alpha>0$, there exists $R>0$ such that the probability of the event "for all $t\in[0,T]$, $\U(t)\in D_R$" is greater than $1-\alpha$, 
\item almost surely, for all $t\in[0,T]$, for all test function $\varphi\in C^2(\T;\R^2)$, the following equation is satisfied:
\begin{align}
\big\langle \U(t),\varphi\big\rangle=\big\langle \U_0,\varphi\big\rangle+\int_0^t&\big\langle \FFF(\U),\partial_x\varphi\big\rangle+\eps\big\langle\U,\partial^2_{xx}\varphi\big\rangle\,d s\nonumber\\
&+\int_0^t\big\langle\mathbf{\Psi}^\eps(\U)\,d W(s),\varphi\big\rangle.\label{EqBoundedSolution}
\end{align}
\end{enumerate}
\label{def:pathsoleps}\end{definition}

Let us make some comments about item 2. in Definition~\ref{def:pathsoleps}. By continuity of $\U(t)$ with values in $L^2(\T)$ and by continuity of the filtration $(\mathcal{F}_t)$, we can introduce the stopping time $T_R$ defined by
\begin{equation}\label{deftauR}
{T}_R=\inf\left\{t\in[0,T];\U(t)\notin D_R\right\}
\end{equation}
(with the convention that $T_R=T$ if $\U(t)\in D_R$ for all $t\in[0,T]$), and then item 2. in Definition~\ref{def:pathsoleps} is equivalent to
\begin{equation}\label{TRT}
\displaystyle\lim_{R\to+\infty}\P(T_R<T)=0.
\end{equation}

We will prove the existence of pathwise solutions to the parabolic stochastic problem~\eqref{stoEulereps} satisfying uniform (or weighted) estimates with respect to $\eps$. If $\eta$ is an entropy function given by \eqref{entropychi} with a convex function $g$ of class $C^2$, we denote by 
$$
\Gamma_\eta(\U)=\int_{\T}\eta(\U(x))dx,
$$
the total entropy of a function $\U\colon\T\to\R^2$.

\begin{theorem}[Existence of pathwise solution to \eqref{stoEulereps}] Let  ${\U_\eps}_0\in W^{2,2}(\T)$ satisfy ${\rho_\eps}_0\geq c_0$ a.e. in $\T$, for a positive constant $c_0$. For $m\in\N$, let $\eta_m$ denote the entropy associated to $\xi\mapsto \xi^{2m}$ by \eqref{entropychi}. Assume that hypotheses \eqref{A0eps}, \eqref{Trunceps}, \eqref{Lipsigmaeps} are satisfied and that ${\U_\eps}_0\in\Lambda_{\varkappa_\eps}$. Then the problem \eqref{stoEulereps} admits a unique bounded solution ${\U_\eps}$, which has the following property: 
\begin{enumerate}
\item it satisfies some moment estimates: for all $m\in\N$,
\begin{equation}\label{estimmomenteps2}
\E\sup_{t\in[0,T]}\int_{\T^1}\left(|{u}_\eps|^{2m}+|{\rho}_\eps|^{m(\gamma-1)}\right) {\rho}_\eps dx=\mathcal{O}(1),
\end{equation}
where $\mathcal{O}(1)$ depends on $T$, $\gamma$, on the constant ${A_0}$ in \eqref{A0eps}, on $m$ and on $\E\Gamma_{\eta}({\U_\eps}_0)$ for $\eta\in\{\eta_0,\eta_{2m}\}$,
\item it satisfies the following gradient estimates: for all $m\in\N$,
\begin{equation}\label{corestimgradientepsrho}
\eps\E\iint_{Q_T} \left(|{u}_\eps|^{2m}+{\rho}_\eps^{2m\theta}\right){\rho}_\eps^{\gamma-2}|\partial_x {\rho}_\eps|^2 dx dt=\mathcal{O}(1),
\end{equation}
and
\begin{equation}\label{corestimgradientepsu}
\eps\E\iint_{Q_T} \left(|{u}_\eps|^{2m}+{\rho}_\eps^{2m\theta}\right){\rho}_\eps|\partial_x {u}_\eps|^2 dx dt=\mathcal{O}(1),
\end{equation}
 where $\mathcal{O}(1)$ depends on $T$, $\gamma$, on the constant ${A_0}$ in \eqref{A0eps} and on the initial quantities $\E\Gamma_{\eta}(\U_0)$ for $\eta\in\{\eta_0,\eta_{2m+2}\}$,
\item the region $\Lambda_{\varkappa_\eps}$ is an invariant region: a.s., for all $t\in[0,T]$, ${\U_\eps}(t)\in\Lambda_{\varkappa_\eps}$.
\end{enumerate}
Besides, ${\U_\eps}$ has the regularity $L^2_\omega C^\alpha_t W^{1,2}_x$ ($\alpha<1/4$) and $L^2_{\omega}C^0_tW^{2,2}_x$, see \eqref{HoldertH1xBounded}-\eqref{LinftytH2xBounded}, and $\U_\eps$ satisfies the following entropy balance equation: for all entropy-entropy flux pair $(\eta,H)$ where $\eta$ is of the form \eqref{entropychi} with a convex function $g$ of class $C^2$, almost surely, for all $t\in[0,T]$, for all test function $\varphi\in C^2(\T)$,
\begin{align}
\big\langle \eta({\U_\eps}(t)),\varphi\big\rangle+&\eps\int_0^t\big\langle \eta''({\U_\eps})\cdot({\partial_x\U_\eps},{\partial_x\U_\eps}),\varphi\big\rangle ds\nonumber\\
=&\big\langle \eta({\U_\eps}_0),\varphi\big\rangle+\int_0^t\left[ \big\langle H({\U_\eps}),\partial_x\varphi\big\rangle+\eps\big\langle\eta({\U_\eps}),\partial^2_x\varphi\big\rangle\right]d s\nonumber\\
&+\int_0^t \big\langle\eta'({\U_\eps})\mathbf{\Psi}^{\eps}({\U_\eps})\,d W(s),\varphi\big\rangle\nonumber\\
&+ \frac{1}{2}\int_0^t\big\langle\GG^{\eps}({\U_\eps})^2\partial^2_{qq} {\eta}({\U_\eps}),\varphi\big\rangle ds.\label{Itoentropyeps}
\end{align}
\label{th:existspatheps}\end{theorem}


To prove the existence of such pathwise solutions, we will prove first the existence of martingale solution 
and then use the Gy{\"o}ngy-Krylov argument~\cite{GyongyKrylov96} to conclude (section \ref{subsec:prooftheps}). 
This means that we have to prove a result of pathwise uniqueness, which is given by the following theorem.

\begin{theorem}[Uniqueness of bounded solution to \eqref{stoEulereps}] Let ${\U_\eps}_0\in L^\infty(\T)$ satisfy ${\rho_\eps}_0\geq c_0$ a.e. in $\T$, for a positive constant $c_0$. Let $T>0$. Assume that hypotheses \eqref{Trunceps}, \eqref{Lipsigmaeps} are satisfied. Then, the problem \eqref{stoEulereps} admits at most one bounded solution ${\U_\eps}$.
\label{th:uniqpatheps}\end{theorem}

\begin{proof} Let $S_\eps(t)=S(\eps^{-1}t)$, where $S(t)$ is the heat semi-group on $\T$. From the weak formulation \eqref{EqBoundedSolution} follows the mild formulation: almost surely, for all $t\in[0,T]$,
\begin{equation}\label{MildBoundedSolution}
\U(t)=S_\eps(t)\U_0-\int_0^t \partial_x S_\eps(t-s)\FFF(\U(s))ds+\int_0^t S_\eps(t-s)\mathbf{\Psi}^\eps(\U(s))\,d W(s),
\end{equation}
(see, \textit{e.g.}, \cite{Ball77} in the deterministic case and \cite[Proposition~3.7]{GyongyRovira00} for a stochastic version of that result). Note that each member of \eqref{MildBoundedSolution} is almost surely in $C([0,T];L^2(\T))$: this is the case of $\U$ by Definition~\ref{def:pathsoleps}; the term $S_\eps(t)\U_0$ is deterministic and continuous in $t$ with values in $L^2(\T)$ by continuity of the semi-group $(S_\eps(t))$. To prove the continuity of the two remaining terms in \eqref{MildBoundedSolution}, let us set 
\begin{align*}
\mathcal{T}_\mathrm{det}\U(t)&=\int_0^t \partial_x S_\eps(t-s)\FFF(\U(s))ds,\\
\mathcal{T}_\mathrm{sto}\U(t)&=\int_0^t S_\eps(t-s)\mathbf{\Psi}^\eps(\U(s))\,d W(s).
\end{align*}
Let $L(R)$ denote the Lipschitz constant of $\FFF$ on $D_R$. 
Let $\omega\in\Omega$ be such that $\U(x,t) \in D_R$ for a.e. $(x,t)\in Q_T$. 
Since $\U$ is a bounded solution, such a bound is satisfied for almost all $\omega$, 
provided $R=R(\omega)$ is large enough (\textit{cf.} \eqref{TRT}). By \eqref{partialKtp} with $j=1$, $k=1$, $p=2$, we have, with $S(t)u=K_t \ast u$ (where $K_t$ is defined in \eqref{perHeatKernel}),
\begin{align*}
&\|\partial_x S_\eps(t_2-s)\FFF(\U(s))-\partial_x S_\eps(t_1-s)\FFF(\U(s))\|_{L^2(\T)}\\
\\
\lesssim\ & \|\partial_x S_\eps(t_2-s)\FFF(\U(s))-\partial_x S_\eps(t_1-s)\FFF(\U(s))\|_{L^\infty(\T)}\\
\\
\leq\ &\|\partial_x K_{\eps(t_2-s)}-\partial_x K_{\eps(t_1-s)}\|_{L^2(\T)}\|\FFF(\U(s))\|_{L^\infty(\T)}\\
\\
\lesssim\ &\eps^{-7/4}\int_{t_1-s}^{t_2-s} t^{-7/4} dt\ \|\FFF\|_{L^\infty(D_R)}\\
\\
\lesssim\ &\eps^{-7/4}\left[(t_2-s)^{-3/4}-(t_1-s)^{-3/4}\right]\|\FFF\|_{L^\infty(D_R)}.
\end{align*}
Similarly, taking $j=1$, $k=0$, $p=2$ in \eqref{partialKtp}, we obtain
$$
\left\|\int_{t_1}^{t_2}\partial_x S_\eps(t-s)\FFF(\U(s))ds\right\|_{L^2(\T)}\lesssim\eps^{-1/2}(\sqrt{t_2}-\sqrt{t_1})\|\FFF\|_{L^\infty(D_R)}.
$$
It follows that
\begin{equation}\label{HolderTdet}
\left\|\mathcal{T}_\mathrm{det}\U(t_2)-\mathcal{T}_\mathrm{det}\U(t_1)\right\|_{L^2(\T)}
\lesssim \eps^{-7/4}\|\FFF\|_{L^\infty(D_R)}\delta_{\mathrm{det}}(t_1,t_2),
\end{equation}
where
\begin{equation}\label{defdeltadet}
\delta_{\mathrm{det}}(t_1,t_2)=\sqrt{t_2}-\sqrt{t_1}+\int_0^{t_1}\left[(t_2-s)^{-3/4}-(t_1-s)^{-3/4}\right]ds.
\end{equation}
We use the same kind of estimates to show the continuity of the stochastic term. Instead of fixed times $t_1,t_2$, let us consider some stopping times $T_1\leq T_2$ satisfying $T_i\leq T$ a.s. for $i=1,2$. Recall (see Corollary~5.10 p.52 in \cite{DoleansDade77} for example) that 
$$
\int_0^{{T}_i} S_\eps({T}_i-s)\mathbf{\Psi}^\eps(\U(s))\,d W(s)=\int_0^T \mathbf{1}_{s\in[0,{T}_i]} S_\eps({T}_i-s)\mathbf{\Psi}^\eps(\U(s))\,d W(s).
$$
By It\={o}'s Isometry and the bound \eqref{BoundTrunceps}, we have therefore
\begin{align}
&\E\left\|\mathcal{T}_\mathrm{sto}\U({T}_2)-\mathcal{T}_\mathrm{sto}\U({T}_1)\right\|_{L^2(\T)}^2\nonumber\\
=\ &\E\int_{{T}_1}^{{T}_2}\|S_\eps({T}_2-s)\mathbf{G}^\eps(\U(s))\|_{L^2(\T)}^2 ds+\E\int_{0}^{{T}_1}\|\left[S_\eps({T}_2-s)-S_\eps({T}_1-s)\right]\mathbf{G}^\eps(\U(s))\|_{L^2(\T)}^2 ds\nonumber\\
\lesssim\ & \E({T}_2-{T}_1)M(\varkappa_\eps)^2+\E\int_{0}^{{T}_1}\left|\eps^{-5/4}\left[({T}_2-s)^{-1/4}-({T}_1-s)^{-1/4}\right]\right|^2 ds M(\varkappa_\eps)^2\nonumber\\
\lesssim\ & \eps^{-5/2}M(\varkappa_\eps)^2 \E \delta_{\mathrm{sto}}({T}_1,{T}_2)^2,\label{TstoStop}
\end{align}
where
\begin{equation}\label{defdeltasto}
\delta_{\mathrm{sto}}(t_1,t_2)^2=(t_2-t_1)+\int_{0}^{t_1}\left[(t_2-s)^{-1/4}-(t_1-s)^{-1/4}\right]^2ds.
\end{equation}
Note that the estimate on $\mathcal{T}_\mathrm{det}\U$ can also be adapted to the case where $t_i={T}_i(\omega)$ for ${T}_1\leq{T}_2$ some stopping times as above. In particular, we have
\begin{equation}\label{TdetStop}
\E\left\|\mathcal{T}_\mathrm{det}\U({T}_2\wedge{T}_R)-\mathcal{T}_\mathrm{det}\U({T}_1\wedge{T}_R)\right\|_{L^2(\T)}^2
\lesssim \eps^{-7/2}\|\FFF\|_{L^\infty(D_R)}^2\E\delta_{\mathrm{det}}({T}_1\wedge{T}_R,{T}_2\wedge{T}_R)^2,
\end{equation}
where $T_R$ is defined by \eqref{deftauR}.\medskip

Let $\sigma$ be a stopping time such that $\sigma\leq T$ almost surely. If $\sigma$ takes a finite number on values $\sigma_1,\ldots,\sigma_n$, then by \eqref{MildBoundedSolution}, almost surely on $\{\sigma=\sigma_k\}$, for all $t\in[0,\sigma_k]$, \eqref{MildBoundedSolution} is satisfied. Equivalently, we have: almost surely, for all $t\in[0,T]$,
\begin{align}
\U(t\wedge\sigma)=\ & S_\eps(t\wedge\sigma)\U_0-\int_0^{t\wedge\sigma} \partial_x S_\eps(t\wedge\sigma-s)\FFF(\U(s))ds\nonumber\\
&+\int_0^{t\wedge\sigma} S_\eps(t\wedge\sigma-s)\mathbf{\Psi}^\eps(\U(s))\,d W(s).\label{MildBoundedSolutionTau}
\end{align}
Let $\sigma^n$ be a sequence of simple stopping times converging to $\sigma$ in $L^1(\Omega)$ and such that $\sigma^n\geq\sigma$ for all $n$, \textit{e.g.} $\sigma^n=2^{-n}[2^n\sigma+1]$, where $[t]$ is the integer part of $t$. If $\alpha>0$, we have, by \eqref{TdetStop} and the Markov inequality, for $R>0$,
\begin{equation*}
\P\left[\left\|\mathcal{T}_\mathrm{det}\U(\sigma^n)-\mathcal{T}_\mathrm{det}\U(\sigma)\right\|_{L^2(\T)}>\alpha\right]
\lesssim \P(T_R<T)+\alpha^{-1}\eps^{-7/4}\|\FFF\|_{L^\infty(D_R)}\E\delta_{\mathrm{det}}(\sigma,\sigma^n).
\end{equation*}
Since $\P(T_R<T)\to 0$ when $R\to+\infty$, it follows that $\mathcal{T}_\mathrm{det}\U(\sigma^n)\to\mathcal{T}_\mathrm{det}\U(\sigma)$ in $L^2(\T)$ in probability. Using \eqref{TstoStop}, we can also pass to the limit in the stochastic term to show that \eqref{MildBoundedSolutionTau} holds true when $\sigma$ is a general stopping time.\medskip

Now we consider two bounded solutions $\U_1$, $\U_2$ to \eqref{stoEulereps}. Let $R>1$ be such that ${\U_\eps}_0\in D_R$, let
\begin{equation*}
{T}_R^{1,2}=\inf\left\{t\in[0,T];\U^1(t)\mbox{ or }\U^2(t)\notin D_R\right\}.
\end{equation*}
By \eqref{regSr2}, we have: almost surely, for $0\leq s\leq t\wedge{T}_R^{1,2}$,
\begin{align*}
&\|\partial_x S_\eps(t\wedge{T}_R^{1,2}-s)\left[\FFF(\U_1(s))-\FFF(\U_2(s))\right]\|_{L^2(\T)}\\
\\
\leq\ & \eps^{-1/2}(t\wedge{T}_R^{1,2}-s)^{-1/2}L(R)\sup_{s\in[0,t\wedge{T}^{1,2}_R]}\|\U_1(s)-\U_2(s)\|_{L^2(\T)}.
\end{align*}
This gives
\begin{multline}\label{diffTdet}
\E\left\|\mathcal{T}_\mathrm{det}\U_1(t\wedge{T}^{1,2}_R)-\mathcal{T}_\mathrm{det}\U_2(t\wedge{T}^{1,2}_R)\right\|_{L^2(\T)}^2\\
\leq 4\eps^{-1}L(R)^2 \ t\ \E\sup_{s\in[0,t]}\|\U_1(s\wedge{T}^{1,2}_R)-\U_2(s\wedge{T}^{1,2}_R)\|_{L^2(\T)}^2.
\end{multline}
By It\={o}'s Isometry and the bound \eqref{Lipsigmaeps}, we have
\begin{multline}\label{diffTsto}
\E\left\|\mathcal{T}_\mathrm{sto}\U_1(t\wedge{T}^{1,2}_R)-\mathcal{T}_\mathrm{sto}\U_2(t\wedge{T}^{1,2}_R)\right\|_{L^2(\T)}^2\\
\leq C(\eps,R)\ t\ \E\sup_{s\in[0,t]}\|\U_1(s\wedge{T}^{1,2}_R)-\U_2(s\wedge{T}^{1,2}_R)\|_{L^2(\T)}^2.
\end{multline}
It follows from \eqref{MildBoundedSolutionTau}, \eqref{diffTdet}, \eqref{diffTsto} that 
\begin{multline*}
\E\sup_{s\in[0,t]}\|\U_1(s\wedge{T}^{1,2}_R)-\U_2(s\wedge{T}^{1,2}_R)\|_{L^2(\T)}^2\\
\leq \tilde C(\eps,R)\ t\ \E\sup_{s\in[0,t]}\|\U_1(s\wedge{T}^{1,2}_R)-\U_2(s\wedge{T}^{1,2}_R)\|_{L^2(\T)}^2,
\end{multline*}
where $\tilde C(\eps,R)=4\eps^{-1}L(R)^2+C(\eps,R)$. For $t<t_1:=1/\tilde C(\eps,R)$, we obtain: almost surely, $\U_1=\U_2$ 
on the interval $[0,t_1\wedge{T}^{1,2}_R]$. We then repeat the argument on the intervals 
$[kt_1,(k+1)t_1]$, $k=1,\ldots$ This is licit since the semi-group property shows that \eqref{MildBoundedSolutionTau} holds true when starting from time $t_1$:
\begin{align*}
\U(t\wedge\sigma+t_1\wedge\sigma)=\ & S_\eps(t\wedge\sigma)\U(t_1\wedge\sigma)-\int_{0}^{t\wedge\sigma} \partial_x S_\eps(t\wedge\sigma-s)\FFF(\U(s+t_1\wedge\sigma))ds\\
&+\int_{0}^{t\wedge\sigma} S_\eps((t\wedge\sigma)\mathbf{\Psi}^\eps(\U(s+t_1\wedge\sigma))\,d W(s).
\end{align*}
This gives $\U_1=\U_2$ a.s. on $[0,{T}^{1,2}_R]$. Since ${T}_R^{1,2}\to T$ almost surely as $R\to+\infty$, we conclude to $\U_1=\U_2$ a.s.
\end{proof}

\begin{remark} Assume $\mathbf{\Psi^\eps}=0$. In this deterministic case the stopping time $T_R$ is deterministic: $\P(T_R=T)>0$ implies $T_R=T$. For $R$ large enough, and by the bound \eqref{diffTdet}, we obtain the following estimate:
\begin{equation*}
\sup_{t\in[0,T]}\|\U_1(t)-\U_2(t)\|_{L^2(\T)}\leq C(T,R,\eps)\|\U_1(0)-\U_2(0)\|_{L^2(\T)},
\end{equation*}
where $\U_1$ and $\U_2$ are two bounded solutions to Problem~\eqref{stoEulereps} and $C(T,R,\eps)$ is a constant depending on $T$, $R$ and $\eps$.

\label{rk:uniqpathepsdet}\end{remark}

In the following proposition, we use the fractional Sobolev space $W^{s,2}(\T)$, defined in Appendix~\ref{app:regparabolic}. 

\begin{proposition}[Regularity of bounded solutions to \eqref{stoEulereps}] Let ${\U_\eps}_0\in W^{1,2}(\T)$ satisfy ${\rho_\eps}_0\geq c_0$ a.e. in $\T$, for a positive constant $c_0$. Let $T>0$. Assume that hypothesis \eqref{Trunceps} is satisfied. Let ${\U_\eps}$ be a bounded solution to Problem~\eqref{stoEulereps}. Then, for all $\alpha\in[0,1/4)$, ${\U_\eps}(\cdot\wedge{T}_R)$ has a mo\-di\-fi\-ca\-tion whose trajectories are almost surely in $C^{\alpha}([0,T];L^2(\T))$ and such that
\begin{equation}\label{HolderURalpha}
\E\|{\U_\eps}(\cdot\wedge{T}_R)\|_{C^{\alpha}([0,T];L^2(\T))}^2\leq C(R,\eps,T,\alpha,{\U_\eps}_0),
\end{equation}
where ${T}_R$ is the exit time from $D_R$ (see \eqref{deftauR}) and $C(R,\eps,T,\alpha)$ is a constant depending on $R$, $T$, $\eps$, $\alpha$ and $\|{\U_\eps}_0\|_{W^{1,2}(\T)}$. Furthermore, for every $s\in[0,1)$, ${\U_\eps}$ satisfies the estimate
\begin{equation}\label{H1UepsR}
\sup_{t\in[0,T]}\E\|{\U_\eps}(t\wedge{T}_R)\|_{W^{s,2}(\T)}^2\leq C(R,\eps,T,s,{\U_\eps}_0)
\end{equation}
where $C(R,\eps,T,s,{\U_\eps}_0)$ is a constant depending on $R$, $T$, $\eps$, $s$ and $\|{\U_\eps}_0\|_{W^{1,2}(\T)}$.\medskip

If additionally ${\U_\eps}_0\in W^{2,2}(\T)$ and the Lipschitz condition \eqref{Lipsigmaeps} is satisfied, then 
\begin{equation}\label{HoldertH1xBounded}
\E\|{\U_\eps}(t\wedge T_R)\|_{C^\alpha([0,T];W^{1,2}(\T))}^2\leq C(R,\eps,T,\alpha,{\U_\eps}_0),
\end{equation}
for all $\alpha\in [0,1/4)$, and
\begin{equation}\label{LinftytH2xBounded}
\sup_{t\in[0,T]}\E\|{\U_\eps}(t\wedge T_R)\|_{W^{2,2}(\T)}^2\leq C(R,\eps,T,{\U_\eps}_0),
\end{equation}
where $C(R,\eps,T,{\U_\eps}_0)$ is a constant depending on $R$, $T$, $\eps$, on the constant $C(\eps,R)$ in \eqref{Lipsigmaeps}, and on $\|{\U_\eps}_0\|_{W^{2,2}(\T)}$. 
\label{prop:regboundedeps}\end{proposition}

\begin{proof} \textbf{Step 1.} Note first that ${\U_\eps}_0\in W^{1,2}(\T)$ gives (see \eqref{toHolderU0})
$$
(x,t)\mapsto S_\eps(t){\U_\eps}_0(x)\in C^{1/2}([0,T);L^2(\T)),
$$
with
\begin{equation}\label{HolderU0}
\|S_\eps(t_2){\U_\eps}_0-S_\eps(t_1){\U_\eps}_0\|_{L^2(\T)}\lesssim \eps^{-1/2}|t_2-t_1|^{1/2}\|{\U_\eps}_0\|_{W^{1,2}(\T)}.
\end{equation}
Next, to prove the H\"older regularity of ${\U_\eps}$ in $t$, we use the estimates \eqref{TstoStop} and \eqref{TdetStop} established in the proof of Theorem~\ref{th:uniqpatheps}. By \eqref{defdeltadet} and \eqref{defdeltasto}, we have
$$
\delta_\mathrm{det}(t_2,t_1)\leq (t_2-t_1)^{1/2}+\int_0^{+\infty}\left[(1+s)^{-3/4}-s^{-3/4}\right]ds\ (t_2-t_1)^{1/4},
$$
and
$$
\delta_{\mathrm{sto}}(t_1,t_2)^2\leq(t_2-t_1)+\int_{0}^{+\infty} \left[(1+s)^{-1/4}-s^{-1/4}\right]^2ds\ (t_2-t_1)^{1/2}.
$$
It follows that
\begin{equation}\label{HolderUR}
\E\|\U(t_2\wedge{T}_R)-\U(t_1\wedge{T}_R)\|_{L^2(\T)}^2\leq C(R,\eps,T,{\U_\eps}_0)\max\left(t_2-t_1, (t_2-t_1)^{1/2}\right),
\end{equation}
for all $0\leq t_1\leq t_2\leq T$, where $C(R,\eps,T,{\U_\eps}_0)$ is a constant depending on $R$, $T$, $\eps$ and $\|{\U_\eps}_0\|_{W^{1,2}(\T)}$. We can improve the bound \eqref{HolderUR} as follows: first, we deduce from \eqref{HolderTdet} that, for all $k\geq 1$,
\begin{multline}\label{TdetStopk}
\E\left\|\mathcal{T}_\mathrm{det}\U(t_2\wedge{T}_R)-\mathcal{T}_\mathrm{det}\U(t_1\wedge{T}_R)\right\|_{L^2(\T)}^{2k}\\
\lesssim \eps^{-7k/2}\|\FFF\|_{L^\infty(D_R)}^{2k}\E\delta_{\mathrm{det}}(t_1\wedge{T}_R,t_2\wedge{T}_R)^{2k}\\
\leq C(R,\eps,T,k)\max\left((t_2-t_1)^{k/2},(t_2-t_1)^k\right),
\end{multline}
where $C(R,\eps,T)$ is a constant depending on $R$, $T$, $\eps$, $k$. By the Burkholder-Davis-Gundy inequality, we also have the following analogue to \eqref{TstoStop}:
\begin{align}
&\E\left\|\mathcal{T}_\mathrm{sto}\U({T}_2)-\mathcal{T}_\mathrm{sto}\U({T}_1)\right\|_{L^2(\T)}^{2k}\nonumber\\
\lesssim\ &\E\left[\int_{{T}_1}^{{T}_2}\|S_\eps({T}_2-s)\mathbf{G}^\eps(\U(s))\|_{L^2(\T)}^2 ds\right]^k \nonumber\\
&+\E\left[\int_{0}^{{T}_1}\|\left[S_\eps({T}_2-s)-S_\eps({T}_1-s)\right]\mathbf{G}^\eps(\U(s))\|_{L^2(\T)}^2 ds \right]^{k} \nonumber\\
\lesssim\ & \E({T}_2-{T}_1)^k M(\varkappa_\eps)^{2k}
+\E\left[\int_{0}^{{T}_1}\left|\eps^{-5/4}\left[({T}_2-s)^{-1/4}-({T}_1-s)^{-1/4}\right]\right|^2 ds\right]^k M(\varkappa_\eps)^{2k}\nonumber\\
\leq\ & C(R,\eps,T,k)\max\left((T_2-T_1)^{k/2},(T_2-T_1)^k\right),\label{TstoStopk}
\end{align}
where $C(R,\eps,T,k)$ is a constant depending on $R$, $T$, $\eps$, $k$. By \eqref{HolderU0}, \eqref{TdetStopk} and \eqref{TstoStopk}, we obtain
\begin{equation}\label{HolderURk}
\E\|\U(t_2\wedge{T}_R)-\U(t_1\wedge{T}_R)\|_{L^2(\T)}^{2k}\leq C(R,\eps,T,k)\max\left((t_2-t_1)^{k/2},(t_2-t_1)^k\right),
\end{equation}
for all $0\leq t_1\leq t_2\leq T$, where $C(R,\eps,T,k,{\U_\eps}_0)$ is a constant depending on $R$, $T$, $\eps$, $k$ and $\|{\U_\eps}_0\|_{W^{1,2}(\T)}$. 
By the Kolmogorov's criterion, the existence of a modification with trajectories almost surely $C^\alpha$ and \eqref{HolderURalpha} follow from \eqref{HolderURk}.
\medskip

\textbf{Step 2.} The proof of the regularity in $x$ of ${\U_\eps}$ is also standard: by the contraction property, we have
\begin{equation}\label{H2U0}
\|S_\eps(\cdot){\U_\eps}_0\|_{C([0,T];W^{1,2}(\T))}\leq \|{\U_\eps}_0\|_{W^{1,2}(\T)}.
\end{equation}
Let $s\in(0,1)$. To prove \eqref{H1UepsR}, we use the identity \eqref{BesovBessel}. By \eqref{regSJ}, we have
\begin{align}
\|J^s\mathcal{T}_\mathrm{det}{\U_\eps}(t\wedge{T}_R)\|_{L^2(\T)}&\leq C(R,\eps,T,s),\label{xTdet1}\\
\E\|J^s\mathcal{T}_\mathrm{sto}{\U_\eps}(t\wedge{T}_R)\|_{L^2(\T)}^2&\leq C(R,\eps,T,s),\label{xTsto1}
\end{align}
where $C(R,\eps,T,s)$ is a constant depending on $R$, $T$, $\eps$, $s$. Indeed, the left-hand side in \eqref{xTdet1} is bounded by
\begin{equation}\label{toxTdet1}
\int_0^t(t-r)^{-\frac{1+s}{2}}dr\ C(R,\eps),
\end{equation}
and the left-hand side in \eqref{xTsto1} is bounded by
\begin{equation}\label{toxTsto1}
\int_0^t(t-r)^{-s}dr\ C(R,\eps),
\end{equation}
where $C(R,\eps)$ depends on $R$ and $\eps$. With \eqref{H2U0}, \eqref{xTdet1} and \eqref{xTsto1} give \eqref{H1UepsR}.
\medskip

\textbf{Step 3.} Let us assume now that ${\U_\eps}_0\in W^{2,2}(\T)$ and that the Lipschitz condition \eqref{Lipsigmaeps} is satisfied. By \eqref{Foperates} and \eqref{H1UepsR}, we have 
$$
\sup_{t\in[0,T]}\E\|\FFF({\U_\eps})(t\wedge{T}_R)\|_{W^{s,2}(\T)}^2,\leq C(R,\eps,T,s,{\U_\eps}_0),
$$
and
$$
\sup_{t\in[0,T]}\E\|\GG^\eps({\U_\eps})(t\wedge{T}_R)\|_{W^{s,2}(\T)}^2\leq C(R,\eps,T,s,{\U_\eps})
$$
where $C(R,\eps,T,s,{\U_\eps}_0)$ is a constant depending on $R$, $T$, $\eps$, $s$, $\|{\U_\eps}_0\|_{W^{1,2}(\T)}$ and also on $\FFF$ and on the constant $C(\eps,R)$ in \eqref{Lipsigmaeps}. Using the decompositions
$$
J^{2s}\partial_x S_\eps(t-r)\FFF({\U_\eps})=J^s\partial_x S_\eps(t-r)J^s\FFF({\U_\eps}),
$$
and
$$
J^{2s}\partial_x S_\eps(t-r) \sigma_k(\U_\eps)=J^s\partial_x S_\eps(t-r) J^s\sigma_k(\U_\eps),
$$
we deduce as in \eqref{xTdet1}-\eqref{xTsto1} that, for all $s\in[\frac12,1)$, and for some constants $C(R,\eps,T,s,{\U_\eps}_0)$ possibly varying from lines to lines,
$$
\sup_{t\in[0,T]}\E\|J^{2s-1}J\mathcal{T}_\mathrm{det}{\U_\eps}(t\wedge{T}_R)\|_{L^2(\T)}\leq C(R,\eps,T,s,{\U_\eps}_0)
$$
and
$$
\sup_{t\in[0,T]}\E\|J^{2s-1}J\mathcal{T}_\mathrm{sto}{\U_\eps}(t\wedge{T}_R)\|_{L^2(\T)}\leq C(R,\eps,T,s,{\U_\eps}_0).
$$
This shows that
$$
\sup_{t\in[0,T]}\E\|J{\U_\eps}(t\wedge{T}_R)\|_{W^{2s-1,2}(\T)}\leq C(R,\eps,T,s,{\U_\eps}_0).
$$ 
In particular, almost surely,
\begin{equation}\label{SobolevUepsbounded}
\partial_x{\U_\eps}(\cdot\wedge{T}_R)\in C([0,T];W^{2s-1,2}(\T)),
\end{equation}
and $\partial_x{\U_\eps}$ is solution to the fixed-point equation
\begin{equation}\label{FixedpartialU}
\partial_x{\U_\eps}=S_\eps(\cdot)\partial_x{\U_\eps}_0+\mathcal{T}_\mathrm{det}(D\FFF({\U_\eps})\cdot\partial_x{\U_\eps})
+\mathcal{T}_\mathrm{sto}(D\FFF({\U_\eps})\cdot\partial_x{\U_\eps}),
\end{equation}
on $[0,T_R]$. By \eqref{FixedpartialU}, we can estimate $J\partial_x \U_\eps$. Indeed, \eqref{SobolevUepsbounded} gives 
$$
\partial_x{\U_\eps}(\cdot\wedge{T}_R)\in C([0,T]\times \T),
$$
almost surely. Using \eqref{SobolevAlgebra}, 
we obtain
\begin{equation}\label{CtH2xBounded}
\sup_{t\in[0,T]}\E\|{\U_\eps}(t)\|_{W^{2,2}(\T)}^2\leq C(R,\eps,T,{\U_\eps}_0),
\end{equation}
and therefore \eqref{LinftytH2xBounded}. By \eqref{FixedpartialU} and \eqref{SobolevAlgebra} we obtain \eqref{HoldertH1xBounded} by the same proof as \eqref{HolderURalpha}.
\end{proof}

\begin{remark} By using \eqref{H1UepsR} (\textit{resp.} \eqref{CtH2xBounded}) it is possible to improve \eqref{HolderURalpha} (\textit{resp.} \eqref{HoldertH1xBounded}) to the range $\alpha\in[0,3/8)$. We will not need it anyhow.
\label{rk:betterHolder}\end{remark}
\subsection{Solution to the parabolic problem}\label{sec:ParabolicProblemSol}

\subsubsection{Time splitting}\label{sec:TimeSplitting}

To prove the existence of a solution to \eqref{stoEulereps}, we perform a splitting in time. Let $\tau>0$. Set $t_k=k\tau$, $k\in\N$. We solve alternatively the deterministic, parabolic part of \eqref{stoEulereps} on time intervals $[t_{2k},t_{2k+1})$ and the stochastic part of \eqref{stoEulereps} on time intervals $[t_{2k+1},t_{2k+2})$, \textit{i.e.}
\begin{itemize}
\item for $t_{2k} \leq t < t_{2k+1}$,   
\begin{subequations}\label{splitEulerDet}
\begin{align}
\partial_t \U^\tau+2\partial_x\FFF(\U^\tau)&=2\eps\partial^2_{xx}\U^\tau&\mbox{ in } Q_{t_{2k},t_{2k+1}}, \label{splitEulerDetEq}\\
\U^\tau(t_{2k})&=\U^\tau(t_{2k}-)&\mbox{ in }\T,\label{splitEulerDetCI}
\end{align}
\end{subequations}
\item for $t_{2k+1} \leq t < t_{2k+2}$,   
\begin{subequations}\label{splitEulerSto}
\begin{align}
d\U^\tau&=\sqrt{2} \mathbf{\Psi}^{\eps,\tau}(\U^\tau)dW(t)&\mbox{ in }Q_{t_{2k+1},t_{2k+2}},\label{splitEulerStoEq}\\
\U^\tau(t_{2k+1})&=\U^\tau(t_{2k+1}-)&\mbox{ in }\T.\label{splitEulerStoCI}
\end{align}
\end{subequations}
\end{itemize}

Note that we took care to speed up the deterministic equation \eqref{splitEulerDetEq} by a factor $2$ and the stochastic equation \eqref{splitEulerStoEq} by a factor $\sqrt{2}$. This rescaling procedure should yield a solution $\U^\tau$ consistent with the solution of \eqref{stoEulereps} when $\tau\to 0$. In \eqref{splitEulerSto} we have also truncated (in the number of ``modes") the coefficient $\mathbf{\Psi}^{\eps}$ into a coefficient $\mathbf{\Psi}^{\eps,\tau}$: we assume that, for a finite integer $K^\tau\geq 1$, for each $\rho\geq 0, u\in\R$, we have
\begin{equation}\label{sigmakstartau}
\left[\Phi^{\eps,\tau}(\rho,u)e_k\right](x)=\sigma_k^{\eps,\tau}(x,\rho,u):=\zeta_{\alpha_\tau}*\sigma^\eps_k(x,\rho,u)\mathbf{1}_{k\leq K^\tau}.
\end{equation} 
Then $\mathbf{\Psi}^{\eps,\tau}$ is defined as the vector with first component $0$ and second component $\Phi^{\eps,\tau}(\rho,u)$. Here $\alpha_\tau$ is a sequence tending to $0$ with $\tau$ and $\zeta_\alpha$ is the regularizing kernel defined by
$$
\zeta_\alpha(x,\rho,u)=\frac{1}{\alpha^{3}}\bar\zeta\left(\frac{x}{\alpha}\right)\bar\zeta\left(\frac{\rho}{\alpha}\right)\bar\zeta\left(\frac{u}{\alpha}\right),
$$
where $\bar\zeta$ is the non-negative smooth density of a probability measure. To define the convolution product with respect to $\rho$ in \eqref{sigmakstartau} we have set $\sigma^\eps_k(x,\rho,u)=0$ for $\rho\leq 0$: this is consistent with the bound \eqref{A0eps} which gives $\sigma^\eps_k(x,\rho,u)=0$ for $\rho=0$. We assume furthermore that $\bar\zeta$ is compactly supported in $\R_+$. Then, by \eqref{A0eps}, we have, for $\alpha_\tau$ small enough,
\begin{equation}\label{A0epstau}
\GG^{\eps,\tau}(x,\rho,u):=\bigg(\sum_{k\geq 1}|\sigma_k^{\eps,\tau}(x,\rho,u)|^2\bigg)^{1/2}\leq 2 A_0\rho\left[1+u^{2}+\rho^{2\theta}\right]^{1/2},
\end{equation}
for all $x\in\T$, $\U\in\R_+\times\R$. By \eqref{Trunceps}, we have, for $\alpha_\tau$ small enough with respect to $\varkappa_\eps$, 
\begin{equation}\label{Truncepstau}
\mathrm{supp}(\GG^{\eps,\tau})\subset\T_x\times\Lambda_{2\varkappa_\eps}.
\end{equation}
If follows also from \eqref{BoundTrunceps} and \eqref{Lipsigmaeps} that
\begin{equation}\label{BoundTruncepstau}
|\mathbf{G}^{\eps,\tau}(x,\U)|\leq M(\varkappa_\eps),
\end{equation}
and
\begin{equation}\label{Lipsigmaepstau}
\sum_{k\geq 1}\left|\sigma_k^{\eps,\tau}(x,\U_1)-\sigma_k^{\eps,\tau}(x,\U_2)\right|^2\leq C(\eps,R)|\U_1-\U_2|^2,
\end{equation}
or all $x\in\T$, $\U_1,\U_2\in\R_+\times\R$. 
\medskip

For further use, we note here that \eqref{A0epstau} gives
\begin{equation}\label{noisebyenergyepstau}
|\GG^{\eps,\tau}(x,\U)|^2\leq \rho A_0^\sharp(\eta_0(\U)+\eta_E(\U)),
\end{equation}
where $A_0^\sharp$ depends on ${A_0}$ and $\gamma$ only (compare to \eqref{noisebyenergy}). 
\medskip

Let us define the following indicator functions 
\begin{equation}\label{defhtau}
\mathbf{1}_\mathrm{det}=\sum_{k\geq 0}\mathbf{1}_{[t_{2k},t_{2k+1})},\quad \mathbf{1}_\mathrm{sto}=1-\mathbf{1}_\mathrm{det},
\end{equation}
which will be used to localize various estimates below.

\begin{definition}[Pathwise solution to the splitting approximation] Let $\U_0\in L^\infty(\T)$  satisfy $\rho_0\geq c_0$ a.e. in $\T$, for a positive constant $c_0$. Let $T>0$. A process $(\U(t))_{t\in[0,T]}$ with values in $L^2(\T)$ is said to be a pathwise solution to \eqref{splitEulerDet}-\eqref{splitEulerSto} if it is a predictable process such that
\begin{enumerate}
\item almost surely, $\U \in C([0,T];L^2(\T))$,
\item there exists some random variables $c_\mathrm{min}$, $C_\mathrm{max}$ with values in $(0,+\infty)$ such that, almost surely, 
\begin{equation}\label{asregsolepstau}
c_\mathrm{min}\leq\rho\leq C_\mathrm{max},\quad |q|\leq C_\mathrm{max}\mbox{ a.e. in }Q_T,
\end{equation} 
\item almost surely, for all $t\in[0,T]$, for all test function $\varphi\in C^2(\T;\R^2)$, the following equation is satisfied:
\begin{align}
\big\langle \U(t),\varphi\big\rangle=\big\langle \U_0,\varphi\big\rangle&+2\int_0^t\mathbf{1}_\mathrm{det}(s)\left[ \big\langle\FFF(\U),\partial_x\varphi\big\rangle+\eps\big\langle\U,\partial^2_{xx}\varphi\big\rangle\right]d s\nonumber\\
&+\sqrt{2}\int_0^t\mathbf{1}_\mathrm{sto}(s)\big\langle\mathbf{\Psi}^\eps(\U)\,d W(s),\varphi\big\rangle.\label{weaksoltau}
\end{align}
\end{enumerate}
\label{def:pathsoltau}\end{definition}

Note that item 2. in Definition~\ref{def:pathsoltau} implies a property similar to item 2. in Definition~\ref{def:pathsoleps}: $\U\in D_R$ with high probability. Indeed, if $\alpha>0$ is given and $R$ is chosen such that
$$
\P(c_\mathrm{min}<R^{-1})\leq\alpha,\quad \P(C_\mathrm{max}>R)\leq\alpha,
$$
then the event "for all $t\in[0,T]$, $\U(t)\in D_R$" has probability greater than $1-2\alpha$.

\begin{proposition}[Pathwise solution to the splitting approximation] Let $T>0$, let $\U_0\in W^{2,2}(\T)$ satisfy $\rho_0\geq c_0$ a.e. in $\T$ for a given constant $c_0>0$. Assume that \eqref{A0} is satisfied. Then there exists a unique pathwise solution $\U^\tau$ to \eqref{splitEulerDet}-\eqref{splitEulerSto}. Let $g\in C^2(\R)$ be a convex function. Given an entropy-entropy flux pair $(\eta,H)$ defined by \eqref{entropychi}-\eqref{entropychiflux}, $\U^\tau$ satisfies the following entropy balance equation: almost surely, for all $t\in[0,T]$, for all test function $\varphi\in C^2(\T)$,
\begin{align}
\big\langle \eta(\U^\tau(t)),\varphi\big\rangle
=&\big\langle \eta(\U_0),\varphi\big\rangle+2\int_0^t\mathbf{1}_\mathrm{det}(s)\left[ \big\langle H(\U^\tau),\partial_x\varphi\big\rangle+\eps\big\langle\eta(\U^\tau),\partial^2_{xx}\varphi\big\rangle\right]d s\nonumber\\
&-2\eps\int_0^t\mathbf{1}_\mathrm{det}(s)\big\langle \eta''(\U^\tau)\cdot(\partial_x\U^\tau,\partial_x\U^\tau),\varphi\big\rangle ds\nonumber\\
&+\sqrt{2}\int_0^t\mathbf{1}_\mathrm{sto}(s)\big\langle\eta'(\U^\tau)\mathbf{\Psi}^{\eps,\tau}(\U^\tau)\,d W(s),\varphi\big\rangle\nonumber\\
&+ \int_0^t\mathbf{1}_\mathrm{sto}(s)\big\langle\GG^{\eps,\tau}(\U^\tau)^2\partial^2_{qq} {\eta}(\U^\tau),\varphi\big\rangle ds.\label{Itoentropytau}
\end{align}
\label{prop:pathsoltau}\end{proposition}
\begin{proof} The deterministic problem \eqref{splitEulerDet} is solved in \cite{LionsPerthameSouganidis96}: for Lipschitz con\-ti\-nu\-ous initial data $(\rho_0,q_0)$ with an initial density $\rho_0$ uniformly positive, say $\rho_0\geq c_0>0$ on $\T$, the Problem~\eqref{splitEulerDet} admits a unique solution $\U$ in the class of functions 
$$
\U\in L^\infty(0,\tau,W^{1,\infty}(\T))\cap C([0,t_1];L^2(\T));\quad \rho\geq c_1\mbox{ on }\T\times[0,t_1].
$$
Here $c_1>0$ is a constant depending continuously on $t_1$ and on $c_0$, $\|\rho_0\|_{L^\infty(\T)}$, $\|q_0\|_{L^\infty(\T)}$ (see Theorem~\ref{th:uniformpositive} and Remark~\ref{rk:positivityfordet} in this paper for more details about this positivity result). By \eqref{LinftytH2xBounded}, we have $\U(t_1)\in W^{2,2}(\T)$.\medskip

In a second step, we solve the stochastic problem \eqref{splitEulerSto} on the interval $[t_1,t_2)$ . It is an ordinary stochastic differential equation. The coefficients of the noise in \eqref{sigmakstartau} are functions with bounded derivatives at all orders. Since  $x\mapsto\rho^\tau(x,t_1)$ is in $W^{2,2}(\T)$, we may rewrite the second equation of \eqref{splitEulerSto} as
\begin{equation}\label{splitEulerSto2}
dq=\sum_{k=1}^{K^\tau}g_k(x,q)d\beta_k(t),
\end{equation}
where $g_k$ satisfies
\begin{equation}\label{reggk}
\partial_x^m\partial_q^l g_k\in L^\infty(\R;L^2(\T)),
\end{equation}
for all $l\geq 0$, $m\in\{0,1,2\}$. The existence of a solution to \eqref{splitEulerSto2} on $(t_1,t_2)$ with initial datum $q(x,t_1)$ at $t=t_1$ is ensured by a classical fixed point theorem, in the space of adapted functions
$$
q\in C([t_1,t_2];L^2(\Omega\times\T)).
$$
By differentiating once with respect to $x$ in \eqref{splitEulerSto2}, we obtain
$$
d(\partial_x q)=\sum_{k=1}^{K^\tau}\big(\partial_x g_k(x,q)+\partial_q g_k(x,q) (\partial_x q)\big)d\beta_k(t).
$$
By the It\={o} Formula and the Gronwall Lemma, it follows that 
\begin{equation}\label{firstdiffqtau0}
\sup_{t\in[t_1,t_2]}\E\|\partial_x q\|_{L^p(\T)}^p\leq C\E\|\partial_x q(t_1)\|_{L^p(\T)}^p,\quad p\geq 2,
\end{equation}
where the constant $C$ depends on the function $g_k$'s, on $p$ and on $\tau$. By differentiating again in \eqref{splitEulerSto2}, we have
\begin{align}
d(\partial_{xx}^2 q)=\sum_{k=1}^{K^\tau}\big(&\partial^2_{xx} g_k(x,q)+2\partial^2_{xq} g_k(x,q) (\partial_x q)
+\partial^2_{qq} g_k(x,q) |\partial_x q|^2+\partial_q g_k(x,q) (\partial^2_{xx} q)\big)d\beta_k(t).\label{seconddiffqtau-1}
\end{align}
Using \eqref{firstdiffqtau0} with $p=2$ and $p=4$, the It\={o} Formula and the Gronwall Lemma, we obtain
\begin{equation}\label{seconddiffqtau}
\sup_{t\in[t_1,t_2]}\E\|\partial_{xx}^2 q\|_{L^2(\T)}^2\leq C\big(\E\|\partial_{xx}^2 q(t_1)\|_{L^2(\T)}^2+\E\|\partial_x q(t_1)\|_{L^2(\T)}^2+\E\|\partial_x q(t_1)\|_{L^4(\T)}^4\big),
\end{equation}
where the constant $C$ depends on the function $g_k$'s and on $\tau$. By the Doob's Martingale Inequality, we have therefore
\begin{align*}
&\E\sup_{t\in[t_1,t_2]}\Big\|\int_{t_1}^t \partial_q g_k(x,q(s)) \partial^2_{xx} q(s) d\beta_k(s)\Big\|_{L^2(\T)}^2\\
\leq &4 \E\Big\|\int_{t_1}^{t_2} \partial_q g_k(x,q(s)) \partial^2_{xx} q(s) d\beta_k(s)\Big\|_{L^2(\T)}^2\\
\leq &C\big(\E\|\partial_{xx}^2 q(t_1)\|_{L^2(\T)}^2+\E\|\partial_x q(t_1)\|_{L^2(\T)}^2+\E\|\partial_x q(t_1)\|_{L^4(\T)}^4\big).
\end{align*}
Returning to \eqref{seconddiffqtau-1}, we deduce that
\begin{equation}\label{seconddiffqtau+1}
\E\sup_{t\in[t_1,t_2]}\|\partial_{xx}^2 q\|_{L^2(\T)}^2\leq C\big(\E\|\partial_{xx}^2 q(t_1)\|_{L^2(\T)}^2+\E\|\partial_x q(t_1)\|_{L^2(\T)}^2+\E\|\partial_x q(t_1)\|_{L^4(\T)}^4\big).
\end{equation}
By a similar argument, using Doob's Martingale Inequality, we can improve \eqref{firstdiffqtau0} into
\begin{equation}\label{firstdiffqtau+1}
\E\sup_{t\in[t_1,t_2]}\|\partial_x q\|_{L^p(\T)}^p\leq C\E\|\partial_x q(t_1)\|_{L^p(\T)}^p,\quad p\geq 2.
\end{equation}
Note that differentiation in \eqref{splitEulerSto2} has to be justified. The argument is standard: to obtain a solution to \eqref{splitEulerSto2} which satisfies \eqref{firstdiffqtau+1} and \eqref{seconddiffqtau+1}, we simply prove existence by using a fixed-point method in a smaller space, in\-cor\-po\-ra\-ting the bounds \eqref{firstdiffqtau+1} and \eqref{seconddiffqtau+1}. By \eqref{seconddiffqtau+1}, we conclude that $\U(t_2)\in W^{2,2}(\T)$. Of course $\rho(t_2)=\rho(t_1)\geq c_1$ a.e. on $\T$. The initial datum $\U(t_2)$ is therefore admissible with regard to the re\-so\-lu\-tion of the deterministic problem \eqref{splitEulerDet} on $Q_{t_2,t_3}$. By iterating the procedure, we build $\U^\tau$ on the whole interval $[0,T]$. On intervals $[t_{2k+1},t_{2k+2}]$ (stochastic evolution), the density $\rho$ is unchanged. On intervals $[t_{2k},t_{2k+1}]$ the positivity of $\rho$ at $t=t_{2k}$ is preserved by 
Theorem~\ref{th:uniformpositive} and Remark~\ref{rk:positivityfordet}. Therefore there exists a random variable $c_\mathrm{min}$ (the possibility that it depends on $\tau$ is not excluded at this stage of the proof) such that, almost surely $\rho^\tau\geq c_\mathrm{min}$ a.e. on $Q_T$. \medskip

Regarding the measurability of $\U^\tau$, we observe that the function $\U^\tau(t_2)$ is $\mathcal{F}_{t_2}$-measurable. Since $\U^\tau(t_2)\mapsto (\U^\tau(t))_{t\in[t_2,t_3]}$ is Lipschitz continuous from $L^2(\T)^2$ into $C([t_2,t_3];L^2(\T)^2)$ by Remark~\ref{rk:uniqpathepsdet}, the random variable $\U^\tau(t)$ is $\mathcal{F}_{t_2}$-measurable for every $t\in[t_2,t_3]$. In particular, $\U^\tau(t)$ is adapted on $[t_2,t_3]$. Repeating the argument, we obtain that $\U^\tau(t)$ is adapted. Since $\U^\tau$ is almost surely in $C([0,T];L^2(\T))$, it has a modification which is predictable. Note that the proof of the regularity results of Proposition~\eqref{prop:regUtau}, based on a mild formulation of Problem \eqref{splitEulerDet}-\eqref{splitEulerSto}, shows that $\U^\tau$ can be expressed as the composition of continuous, regularizing operators, with $\U^\tau$. In particular, $(\U^\tau(t))$ is also a predictable process when viewed as a process with values in $W^{1,2}(\T)$ or $L^m(\T)$, $m>2$.\medskip

This achieves the proof of the existence of a pathwise solution $\U^\tau$ to \eqref{splitEulerDet}-\eqref{splitEulerSto}. The uniqueness is a consequence of the uniqueness properties for the deterministic and the stochastic problems. Similarly, the entropy balance equation \eqref{Itoentropytau} is obtained by using the following entropy balance law on $[t_{2k},t_{2k+1}]$:
\begin{align}
\big\langle \eta(\U^\tau(t)),\varphi\big\rangle
=&\big\langle \eta(\U^\tau(t_{2k})),\varphi\big\rangle+2\int_{t_{2k}}^t\mathbf{1}_\mathrm{det}(s)\left[ \big\langle H(\U^\tau),\partial_x\varphi\big\rangle+\eps\big\langle\eta(\U^\tau),\partial^2_{xx}\varphi\big\rangle\right]d s\nonumber\\
&-2\eps\int_{t_{2k}}^t\mathbf{1}_\mathrm{det}(s)\big\langle \eta''(\U^\tau)\cdot(\partial_x \U^\tau,\partial_x \U^\tau),\varphi\big\rangle ds, \label{Itoentropytaudet}
\end{align}
for all $t\in [t_{2k},t_{2k+1}]$, and by combining \eqref{Itoentropytaudet} with the identity
\begin{align}
\big\langle \eta(\U^\tau(t)),\varphi\big\rangle
=&\big\langle \eta(\U^\tau(t_{2k+1})),\varphi\big\rangle
+\sqrt{2}\int_{t_{2k+1}}^t\mathbf{1}_\mathrm{sto}(s)\big\langle\eta'(\U^\tau)\mathbf{\Psi}^{\eps,\tau}(\U^\tau)\,d W(s),\varphi\big\rangle\nonumber\\
&+\int_{t_{2k+1}}^t\mathbf{1}_\mathrm{sto}(s)\big\langle\GG^{\eps,\tau}(\U^\tau)^2\partial^2_{qq} {\eta}(\U^\tau),\varphi\big\rangle ds,\label{Itoentropytausto}
\end{align}
for all $t\in [t_{2k+1},t_{2k+2}]$. We deduce \eqref{Itoentropytausto} from the stochastic equation \eqref{splitEulerSto} (where $x$ is a parameter) and the It\={o} Formula, which we sum against $\varphi$. This concludes the proof of the proposition.
\end{proof} 
 

\subsubsection{Entropy bounds}\label{sec:entropybounds}

If $\eta\in C(\R^2)$ is an entropy and $\U\colon\T\to\R^2$, we denote by
$$
\Gamma_\eta(\U):=\int_{\T}\eta(\U(x))dx
$$ 
the averaged entropy of $\U$. 
\medskip

\begin{proposition}[Entropy bounds] Let $m\in \N$. Let $\eta_m$ be the entropy given by \eqref{entropychi} with $g(\xi)=\xi^{2m}$. 
Let $\U_0\in W^{2,2}(\T)$ be such that $\rho_0\geq c_0$ a.e. in $\T$ for a given constant $c_0>0$. 
Assume that the growth condition \eqref{A0eps} is satisfied
and that
$$
\E\int_\T \left(\eta_0(\U_{\eps 0})+\eta_{2m}(\U_{\eps 0}) \right)dx <+\infty.
$$
Then the solution $\U^\tau$ to \eqref{splitEulerDet}-\eqref{splitEulerSto} satisfies the estimate
\begin{equation}\label{estimmomentepstau}
\E\sup_{t\in[0,T]}\Gamma_{\eta_m}(\U^\tau(t))+2\eps\E\iint_{Q_T} \mathbf{1}_\mathrm{det}\eta_m''(\U^\tau)\cdot(\partial_x \U^\tau,\partial_x \U^\tau)dx dt=\mathcal{O}(1),
\end{equation}
where the quantity denoted by $\mathcal{O}(1)$ depends only on $T$, $\gamma$, on the constant ${A_0}$ in \eqref{A0eps}, on $m$ and on the initial quantities $\E\Gamma_{\eta}(\U_0)$ for $\eta\in\{\eta_0,\eta_{2m}\}$.
\label{prop:boundetam}\end{proposition}

\begin{proof} To prove Proposition~\ref{prop:boundetam} we will use the following result.
\begin{lemma}\label{lemmaentropies} Let $m,n\in\N$. Then
\begin{equation}\label{estimetabelow}
\rho(u^{2m}+\rho^{2m\theta})=\mathcal{O}(1){\eta_m}(\U),\quad {\eta_m}(\U)=\mathcal{O}(1)\left[\rho(u^{2m}+\rho^{2m\theta})\right],
\end{equation}
where $\mathcal{O}(1)$ depends on $m$; 
\begin{equation}\label{productnm}
\eta_m(\U)\cdot\eta_n(\U)=\mathcal{O}(1)\left[\rho\eta_{m+n}(\U)\right],
\end{equation}
where $\mathcal{O}(1)$ depends on $m$ and $n$;
\begin{equation}\label{rhoetam}
\rho\eta_m(\U)=\mathcal{O}(1)\left[\eta_m(\U)+\eta_p(\U)\right],
\end{equation}
for any $p\geq m+\frac{1}{2\theta}$, where $\mathcal{O}(1)$ depends on $m$ and $p$,
and
\begin{equation}\label{nversus0m}
\eta_n(\U)=\mathcal{O}(1)\left[\eta_{0}(\U)+\eta_{m}(\U)\right],
\end{equation}
where $\mathcal{O}(1)$ depends on $m$ and $n$ if $0\leq n\leq m$. Besides, Hypothesis~\eqref{A0epstau} gives the following bounds:
\begin{equation}\label{dqetaG}
\GG^{\eps,\tau}(\U)^2|\partial_q{\eta_m}(\U)|^2=\mathcal{O}(1)\left[\eta_{0}(\U)+\eta_{2m}(\U)\right],
\end{equation}
and
\begin{equation}\label{dqqetaG}
\GG^{\eps,\tau}(\U)^2\partial^2_{qq} {\eta_m}(\U)=\mathcal{O}(1)\left[\eta_0(\U)+\eta_m(\U)\right].
\end{equation}
\end{lemma}

\begin{proof} The second estimate in \eqref{estimetabelow}, the estimates \eqref{productnm} and \eqref{nversus0m}, are all obtained by repeated applications of the Young Inequality. The first estimate in \eqref{estimetabelow} is proved by developing the term $g(u+z\rho^\theta)$ in \eqref{eqeta}: 
\begin{equation}\label{developetam}
\eta_m(\U)=\rho c_\lambda\sum_{j=0}^{2m}\binom{2m}{j}\int_{-1}^1 u^j z^{2m-j}\rho^{\theta(2m-j)}(1-z^2)^\lambda_+ dz.
\end{equation}
The terms with odd index $j$ in the sum in the right-hand side of \eqref{developetam} all vanish. Therefore only non-negative terms remain:
\begin{align*}
\eta_m(\U)&\geq \rho c_\lambda\sum_{j\in\{0,2m\}}\binom{2m}{j}\int_{-1}^1 u^j z^{2m-j}\rho^{\theta(2m-j)}(1-z^2)^\lambda_+ dz\\
&=\rho\left(d_\lambda(m)\rho^{2\theta m}+u^{2m}\right), 
\end{align*}
where the coefficient $d_\lambda(m)$ is given by
$$
d_\lambda(m)=c_\lambda\int_{-1}^1 z^{2m}(1-z^2)^\lambda_+ dz.
$$
Let us now give the details of the proof of \eqref{rhoetam}: using \eqref{estimetabelow}, it is sufficient to get an estimate on $\rho^2(u^{2m}+\rho^{2m\theta})$. If $\rho\leq 1$, then $\eta_m(\U)$ will provide an upper bound by \eqref{estimetabelow} again. If $\rho\geq 1$, then $\rho^{2m\theta+1}\leq\rho^{2p\theta}$ and 
$$
\rho u^{2m}\leq \frac{\rho^\alpha}{\alpha}+\frac{u^{2m\beta}}{\beta},\quad\frac{1}{\alpha}+\frac{1}{\beta}=1.
$$
Taking $\beta=\frac{p}{m}$ gives $\alpha=\frac{p}{p-m}\leq 2p\theta$, hence
$$
\rho u^{2m}=\mathcal{O}(1)\left[u^{2p}+\rho^{2p\theta}\right]
$$
since $\rho\geq 1$. We conclude to \eqref{rhoetam}. To obtain \eqref{dqetaG} and \eqref{dqqetaG}, we observe that \eqref{A0epstau} is equivalent to
\begin{equation}\label{HypGeta}
\GG^{\eps,\tau}(\U)^2=\mathcal{O}(1)\left[\rho\left(\eta_0(\U)+\eta_1(\U)\right)\right].
\end{equation}
Also, by \eqref{entropychi} and \eqref{estimetabelow}, we have
$$
|\partial_q{\eta_m}(\U)|^2=\mathcal{O}(1)\left[\frac{1}{\rho^2}\eta_{2m-1}(\U)\right], \,
\partial^2_{qq}\eta_m(\U)=\mathcal{O}(1)\left[\frac{1}{\rho^2}\eta_{m-1}(\U)\right].
$$
Using \eqref{productnm}, \eqref{nversus0m}, we deduce \eqref{dqetaG} and \eqref{dqqetaG}.
\end{proof}

We go on now with the proof of Proposition~\ref{prop:boundetam}: we apply the entropy balance equation \eqref{Itoentropytau} to $\U^\tau$ with $\varphi\equiv 1$ and take expectation in both sides. This gives 
\begin{equation*}
\E \Gamma_{\eta_m}(\U^\tau(t))+2\eps\E\iint_{Q_t} \mathbf{1}_\mathrm{det} \eta''(\U^\tau)\cdot(\partial_x \U^\tau,\partial_x \U^\tau)dx ds
=\E\Gamma_{\eta_m}(\U^\tau_0)+\E R_{\eta_m}(t),
\end{equation*}
where 
$$
R_{\eta_m}(t):=\iint_{Q_t} \mathbf{1}_\mathrm{sto}\GG^{\eps,\tau}(\U^\tau)^2\partial^2_{qq} {\eta_m}(\U^\tau) dx ds
$$
is the It\={o} correction term. If $m=0$, then $\partial^2_{qq} \eta=0$ and we obtain (note the difference with \eqref{estimmomentepstau} in the first term)
\begin{equation}\label{estimmomentepstau0}
\sup_{t\in[0,T]}\E\Gamma_{\eta_0}(\U^\tau(t))+2\eps\E\iint_{Q_T} \mathbf{1}_\mathrm{det}\eta_0''(\U^\tau)\cdot(\partial_x \U^\tau,\partial_x \U^\tau)dx dt=\mathcal{O}(1).
\end{equation}
If $m\geq 1$, then \eqref{dqqetaG} gives 
\begin{equation}\label{preGronwall}
\E R_{\eta_m}(t)
=\mathcal{O}(1)\left[\int_0^t\E (\Gamma_{\eta_m}(\U^\tau(s))+\Gamma_{\eta_0}(\U^\tau(s)))ds\right].
\end{equation}
We use Gronwall's Lemma and \eqref{estimmomentepstau0} and deduce the following preliminary estimate
\begin{equation}\label{estimmomentepstau0m}
\sup_{t\in[0,T]}\E\Gamma_{\eta_m}(\U^\tau(t))+2\eps\E\iint_{Q_T} \mathbf{1}_\mathrm{det}\eta_m''(\U^\tau)\cdot(\partial_x \U^\tau,\partial_x \U^\tau)dx dt=\mathcal{O}(1).
\end{equation}
To prove \eqref{estimmomentepstau}, we have to take into account the noise term, \textit{i.e.} we apply the entropy balance equation \eqref{Itoentropytau} to $\U^\tau$ with $\varphi\equiv 1$ and do not take expectation this time: we have then
\begin{equation}\label{withdissip}
0\leq \Gamma_{\eta_m}(\U^\tau(t))= \Gamma_{\eta_m}(\U^\tau_0)+M_{\eta_m}(t)+ R_{\eta_m}(t)-\eps D_{\eta_m}(t)
\end{equation}
where
$$
M_{\eta_m}(t)=\sqrt{2}\sum_{k\geq 1} \int_0^t \mathbf{1}_{\mathrm{sto}}(s)\<\sigma_k^{\eps,\tau}(\U^\tau(s)),\partial_q{\eta_m}(\U^\tau(s))\>_{L^2(\T)} d\beta_k(s)
$$
and 
$$
D_{\eta_m}(t)=2\iint_{Q_t}\mathbf{1}_\mathrm{det}\eta_m''(\U^\tau)\cdot(\partial_x \U^\tau,\partial_x \U^\tau)dx ds.
$$
Since $D_{\eta_m}\geq 0$, \eqref{withdissip} gives
$$
0\leq \Gamma_{\eta_m}(\U^\tau(t)) \leq \Gamma_{\eta_m}(\U^\tau_0)+M_{\eta_m}(t)+ R_{\eta_m}(t).
$$
Similarly as for \eqref{preGronwall}, we have 
$$
\ds \E \sup_{t\in[0,T]} |R_{\eta_m}(t)| = \mathcal{O}(1)\left[\int_0^T\E (\Gamma_{\eta_m}(\U^\tau(s))+\Gamma_{\eta_0}(\U^\tau(s)))ds\right],
$$
and therefore, by \eqref{estimmomentepstau0m}, the last term $R_{\eta_m}$ satisfies the bound
$$
\E\sup_{t\in[0,T]}|R_{\eta_m}(t)|=\mathcal{O}(1).
$$
By the Doob's Martingale Inequality, we also have
\begin{align*}
\E\sup_{t\in[0,T]}\left|M_{\eta_m}(t)\right|&\leq C\E\left(\int_0^T\sum_{k\geq 1}\<\sigma_k^{\eps,\tau}(\U^\tau(s)),\partial_q{\eta_m}(\U^\tau(s))\>_{L^2(\T)}^2\, ds\right)^{1/2}\\
&\leq C\E\left(\iint_{Q_T}\GG^{\eps,\tau}(\U^\tau)^2|\partial_q{\eta_m}(\U^\tau)|^2\, dx ds\right)^{1/2}
\end{align*}
for a given constant $C$. By \eqref{dqetaG} and \eqref{estimmomentepstau0m} (with $2m$ instead of $m$) we obtain
$$
\E\sup_{t\in[0,T]}\left|M_\eta(t)\right|=\mathcal{O}(1).
$$
This concludes the proof of the proposition. 
\end{proof}

\begin{corollary}[Bounds on the moments] Let $m\in\N$. Let $\eta_m$ be the entropy given by \eqref{entropychi} with $g(\xi)=\xi^{2m}$. 
Let $\U_0\in W^{2,2}(\T)$ be such that $\rho_0\geq c_0$ a.e. in $\T$ for a given constant $c_0>0$. Assume that the growth condition \eqref{A0eps} is satisfied
and that
$$
\E\int_\T \left(\eta_0(\U_{\eps 0})+\eta_{2m}(\U_{\eps 0}) \right)dx <+\infty.
$$
Then, the solution $\U^\tau$ to \eqref{splitEulerDet}-\eqref{splitEulerSto} satisfies:
\begin{equation}\label{estimmomentepstau2}
\E\sup_{t\in[0,T]}\int_{\T}\left(|u^\tau|^{2m}+|\rho^\tau|^{m(\gamma-1)}\right) \rho^\tau dx=\mathcal{O}(1),
\end{equation}
where $\mathcal{O}(1)$ depends only on $T$, $\gamma$, on the constant ${A_0}$ in \eqref{A0eps}, on $m$ and on the initial quantities $\E\Gamma_{\eta}(\U_0)$ for $\eta\in\{\eta_0,\eta_{2m}\}$.
\label{cor:boundmoments}\end{corollary}

To conclude this part we complete Lemma~\ref{lemmaentropies} with the following result, which will be used later, in particular to get some estimates on the moments of entropy-entropy flux pairs.

\begin{lemma}\label{lemmaentropies2} For $m\in\N$, let $(\eta_m,H_m)$ be the entropy-entropy flux pair associated to the function $g(\xi)=\xi^{2m}$ by \eqref{entropychi}-\eqref{entropychiflux}. Let $s\geq 1$. Then
\begin{align*}
|\eta_m(\U)|^s&=\mathcal{O}(1)\left[\eta_0(\U)+\eta_p(\U)\right],\quad p\geq ms+\frac{s-1}{2\theta},\\
|H_m(\U)|^s&=\mathcal{O}(1)\left[\eta_0(\U)+\eta_p(\U)\right],\quad p\geq (m+1/2)s+\frac{s-1}{2\theta},\\
|u \eta_m(\U)|^s&=\mathcal{O}(1)\left[\eta_0(\U)+\eta_p(\U)\right],\quad p\geq (m+1/2)s+\frac{s-1}{2\theta},\\
|u H_m(\U)|^s&=\mathcal{O}(1)\left[\eta_0(\U)+\eta_p(\U)\right],\quad p\geq (m+1)s+\frac{s-1}{2\theta},
\end{align*}
where $\mathcal{O}(1)$ depends on $m$, $s$ and the exponent $p$ of each equation.
\end{lemma}

\begin{proof} By \eqref{estimetabelow}, $|\eta_m(\U)|^s=\mathcal{O}(1)\left[\rho^s|u|^{2ms}+\rho^{s+2m\theta s}\right]$. Let $\tilde s\geq ms$. By the Young Inequality, we have
\begin{equation}\label{stildeseta}
\rho^s|u|^{2ms}\leq C_{s,\tilde s}\rho\left(|u|^{2\tilde s}+\rho^{\frac{(s-1)\tilde s}{\tilde s-ms}}\right).
\end{equation}
Let $\tilde s=ms+\frac{s-1}{2\theta}$. If $p\geq\tilde s$, then 
$$
\frac{(s-1)\tilde s}{\tilde s-ms}=2\theta\tilde s\leq 2\theta p
$$
and \eqref{stildeseta} gives
$$
\rho^s|u|^{2ms}=\mathcal{O}(1)\left[\eta_0(\U)+\eta_p(\U)\right].
$$
We also have
$$
\rho^{s+2m\theta s}=\rho\rho^{2\theta\tilde s}=\mathcal{O}(1)\left[\eta_0(\U)+\eta_p(\U)\right]
$$
and the first estimate follows. The proof of the three other estimates is similar.
\end{proof}

\subsubsection{$L^\infty$ estimates}\label{sec:Linfty}

\begin{proposition}[$L^\infty$ estimates] Let $\U_0\in W^{2,2}(\T)$ be such that $\rho_0\geq c_0$ a.e. in $\T$ for a given constant $c_0>0$. 
Assume that the growth condition \eqref{A0eps} and the localization condition \eqref{Trunceps} are satisfied and that $\U_0\in\Lambda_{\varkappa_\eps}$. Then the so\-lu\-tion $\U^\tau$ to \eqref{splitEulerDet}-\eqref{splitEulerSto} satisfies: almost surely, for all $t\in[0,T]$, $\U^\tau(t)\in\Lambda_{2\varkappa_\eps}$. In particular, almost surely, $\|u^\tau\|_{L^\infty(Q_T)}\leq 2\varkappa_\eps$ and $\|\rho^\tau\|_{L^\infty(Q_T)}^\theta\leq 2\varkappa_\eps$.
\label{prop:meanLinfty}\end{proposition}

%
\begin{proof} We refer to \cite[Section 4.]{Diperna83b} and \cite{ChuehConleySmoller77} for the proof 
that $\Lambda_{\varkappa}$ is an invariant region for \eqref{splitEulerDet}. 
See also \cite{Serre96} for a presentation of invariant regions.
In \eqref{splitEulerSto}, $\rho(t)$ is constant. Dividing by $\rho$ the equation on $q=\rho u$, we deduce from \eqref{splitEulerSto} a stochastic differential equation on $u$. Using again that $\rho(t)$ is constant, this gives a stochastic differential equation on $w$ with $x$ as a parameter and similarly for $z$. By Hypothesis \eqref{Trunceps}, we have the localization property \eqref{Truncepstau} and  the region $\Lambda_{2\varkappa_\eps}$ is also an invariant domain for \eqref{splitEulerSto}. It follows that, a.s., for all $t\in[0,T]$, $\U^\tau(t)\in\Lambda_{2\varkappa_\eps}$.
\end{proof}

\subsubsection{Gradient estimates}\label{sec:gradestimates}

In Proposition~\ref{prop:boundetam} above, we have obtained an estimate on $\U^\tau_x$. In the case where $\eta=\eta_E$ is the energy (this corresponds to $g(\xi)=\frac12\xi^2$), we have
\begin{equation}\label{gradientEnergy}
\eta_E''(\U)\cdot(\partial_x \U,\partial_x \U)=\theta^2|\rho|^{\gamma-2}|\partial_x \rho|^2+\rho|\partial_x u|^2.
\end{equation}
More generally, we prove the following weighted estimates (see in particular Corollary~\ref{cor:boundgradient2tau} below).
\begin{proposition}[Gradient bounds] Let $m\in \N$. Let $\eta_m$ be the entropy given by \eqref{entropychi} with $g(\xi)=\xi^{2m}$. 
Let $\U_0\in W^{2,2}(\T)$ be such that $\rho_0\geq c_0$ a.e. in $\T$ for a given constant $c_0>0$. Assume that the growth condition \eqref{A0eps} is satisfied
and that
$$
\E\int_\T \left(\eta_0(\U_{\eps 0})+\eta_{2m}(\U_{\eps 0}) \right)dx <+\infty.
$$
Then, the solution $\U^\tau$ to \eqref{splitEulerDet}-\eqref{splitEulerSto} satisfies the estimate
\begin{align}
&\eps\E\iint_{Q_T}\mathbf{1}_{\mathrm{det}}(t)\,G^{[2]}(\rho^\tau,u^\tau) \big[\theta^2|\rho^\tau|^{\gamma-2} |\partial_x \rho^\tau|^2+\rho^\tau |\partial_x u^\tau|^2\big]dx dt\nonumber\\
\leq&\eps\E\iint_{Q_T}\mathbf{1}_{\mathrm{det}}(t)\,G^{[1]}(\rho^\tau,u^\tau)\big[2\theta |\rho^\tau|^{\frac{\gamma-2}{2}} |\partial_x \rho^\tau|\cdot |\rho^\tau|^{1/2} |\partial_x u^\tau|\big] dx dt
+\mathcal{O}(1),\label{estimgradientepstau}
\end{align}
where
\begin{align*}
G^{[2]}(\rho,u)&=c_\lambda \int_{-1}^1 g''(u+z\rho^\theta)(1-z^2)_+^\lambda dz,\\
G^{[1]}(\rho,u)&=c_\lambda \int_{-1}^1 |z| g''(u+z\rho^\theta)(1-z^2)_+^\lambda dz,
\end{align*}
and $\mathcal{O}(1)$ depends on $T$, $\gamma$, on the constant ${A_0}$ in \eqref{A0eps} and on the initial quantities $\E\Gamma_{\eta}(\U_0)$ for $\eta\in\{\eta_0,\eta_{2m}\}$.
\label{prop:boundgradient2tau}\end{proposition}

\begin{proof} We introduce the probability measure
$$
dm_\lambda(z)=c_\lambda (1-z^2)^\lambda_+ dz
$$
and the $2\times 2$ matrix
$$
S=\begin{pmatrix} 1 & 0 \\ u &1
\end{pmatrix},
$$
which satisfies
\begin{equation}\label{introW}
\partial_x \U=S \mathbf{W},\quad \mathbf{W}:= \begin{pmatrix} \partial_x\rho \\ \rho\partial_x u \end{pmatrix}.
\end{equation}
By \eqref{estimmomentepstau}, we then have
\begin{equation}\label{Sconvex}
\eps\int_0^T \E\int_{\T} \mathbf{1}_{\mathrm{det}}(t)\,\<S^* \eta''(\U^\tau)S\mathbf{W},\mathbf{W}\> dx dt=\mathcal{O}(1),
\end{equation}
where $\<\cdot,\cdot\>$ is the canonical scalar product on $\R^2$ and $S^*$ is the adjoint of $S$ for this scalar product.  We compute 
$$
\eta''(\U)=\frac1\rho\int_\R \left[A(z)g'\left(u+z\rho^{\theta}\right)+B(z)g''\left(u+z\rho^{\theta}\right)\right]dm_\lambda(z),
$$
where
$$
A(z)=\begin{pmatrix} \frac{\gamma^2-1}{4}z\rho^\theta & 0 \\ 0 &0
\end{pmatrix},\quad B(z)=\begin{pmatrix} \left[-u+\theta z\rho^\theta\right]^2 & -u+\theta z\rho^\theta \\ -u+\theta z\rho^\theta &1
\end{pmatrix}.
$$
In particular
$$
S^*AS(z)=\begin{pmatrix} \frac{\gamma^2-1}{4} z\rho^\theta & 0 \\ 0 &0
\end{pmatrix},\quad S^*BS(z)=\begin{pmatrix} \theta^2 z^2\rho^{2\theta} & \theta z\rho^\theta \\ \theta z\rho^\theta &1
\end{pmatrix},
$$
and \eqref{introW}-\eqref{Sconvex} give
\begin{equation}
\eps\E\iint_{Q_T}\mathbf{1}_{\mathrm{det}}(t)\,\left(\mathbf{I}|\partial_x \rho^\tau|^2+\mathbf{J}\partial_x\rho^\tau\cdot|\rho^\tau|^{1/2}\partial_x u^\tau+\mathbf{K}\rho^\tau|\partial_x u^\tau|^2\right) dx dt =\mathcal{O}(1),
\label{IJKeps}\end{equation}
where
\begin{equation*}
\mathbf{I}=|\rho^\tau|^{2\theta-1}\int_\R \theta^2 z^2 g''\left(u^\tau+z|\rho^\tau|^\theta\right) dm_\lambda(z)
+|\rho^\tau|^{\theta-1}\int_\R  \frac{\gamma^2-1}{4} z g'\left(u^\tau+z|\rho^\tau|^{\theta}\right) dm_\lambda(z),
\end{equation*}
and
\begin{align*}
\mathbf{J}=2|\rho^\tau|^{\theta-\frac12}\int_\R \theta  z g''\left(u^\tau+z|\rho^\tau|^{\theta}\right) dm_\lambda(z),\quad
\mathbf{K}=\int_\R g''\left(u^\tau+z|\rho^\tau|^{\theta}\right) dm_\lambda(z).
\end{align*}
We observe that $2z dm_\lambda(z)=-\frac{c_\lambda}{\lambda+1}d(1-z^2)_+^{\lambda+1}$. By integration by parts, the second term in $\mathbf{I}$ can therefore be written
$$
\frac{1}{\lambda+1}|\rho^\tau|^{2\theta-1}\int_\R \frac{\gamma^2-1}{8}(1-z^2) g''\left(u^\tau+z|\rho^\tau|^{\frac{\gamma-1}{2}}\right) dm_\lambda(z).
$$
Since $\frac{\gamma^2-1}{8}\frac{1}{\lambda+1}=\theta^2$, we have
$$
\mathbf{I}=|\rho^\tau|^{2\theta-1}\int_\R \theta^2 g''\left(u^\tau+z|\rho^\tau|^{\frac{\gamma-1}{2}}\right) dm_\lambda(z).
$$
This gives \eqref{estimgradientepstau}. 
\end{proof}
\bigskip

We apply \eqref{estimgradientepstau} with $g(\xi):=|\xi|^{2m+2}$ and $\eta=\eta_{m+1}$ given by \eqref{entropychi}. Then we compute explicitly
$$
G^{[2]}(\rho,u)-G^{[1]}(\rho,u)=(2m+2)(2m+1)\sum_{k=0}^m\binom{2m}{2k}a_k \rho^{2\theta k}u^{2(m-k)},
$$
where the coefficients
$$
a_k=c_\lambda\int_{-1}^1 |z|^{2k}(1-|z|)(1-z^2)_+^\lambda dz
$$
are positive. By letting $m$ vary, we obtain the following result.

\begin{corollary} Let $\U_0\in W^{2,2}(\T)$ be such that $\rho_0\geq c_0$ a.e. in $\T$ for a given constant $c_0>0$. 
Let $\eta_m$ be the entropy given by \eqref{entropychi} with $g(\xi)=\xi^{2m}$. Assume that the growth condition \eqref{A0eps} is satisfied
and that
$$
\E\int_\T \left(\eta_0(\U_{\eps 0})+\eta_{2m}(\U_{\eps 0}) \right)dx <+\infty.
$$
Then, the solution $\U^\tau$ to \eqref{splitEulerDet}-\eqref{splitEulerSto} satisfies the estimate
\begin{equation}\label{corestimgradientepsrhotau}
\eps\E\iint_{Q_T} \mathbf{1}_{\mathrm{det}}(t)\,\left(|u^\tau|^{2m}+|\rho^\tau|^{2m\theta}\right)|\rho^\tau|^{\gamma-2}|\partial_x \rho^\tau|^2 dx dt=\mathcal{O}(1),
\end{equation}
and
\begin{equation}\label{corestimgradientepsutau}
\eps\E\iint_{Q_T} \mathbf{1}_{\mathrm{det}}(t)\,\left(|u^\tau|^{2m}+|\rho^\tau|^{2m\theta}\right)\rho^\tau|\partial_x u^\tau|^2 dx dt=\mathcal{O}(1),
\end{equation}
for all $m\in\N$, where $\mathcal{O}(1)$ depends on $T$, $\gamma$, on the constant ${A_0}$ in \eqref{A0} and on the initial quantities $\E\Gamma_{\eta}(\U_0)$ for $\eta\in\{\eta_0,\eta_{2m+2}\}$.
\label{cor:boundgradient2tau}\end{corollary}

\subsubsection{Positivity of the density}\label{sec:PositiveDensity}

\begin{proposition}[Positivity] Let $\U^\tau$ be the solution to \eqref{splitEulerDet}-\eqref{splitEulerSto} with initial datum $\U_0=(\rho_0,q_0)$ and assume that $\rho_0$ is uniformly positive: there exists $c_0>0$ such that $\rho_0\geq c_0$ a.e. on $\T$. Let $m>3$. Then there is a random variable $c_\mathrm{min}$ with values in $(0,+\infty)$ depending on $c_0$, $T$,
\begin{equation}\label{normforpos}
\iint_{Q_T}\mathbf{1}_\mathrm{det}(t)\rho^\tau|\partial_x u^\tau|^2 dx dt\mbox{ and }\iint_{Q_T}|u^\tau|^m dx dt
\end{equation}
only (in the sense that $c_\mathrm{min}^{-1}$ is a continuous non-decreasing function of these former quantities), such that, almost surely, 
\begin{equation}\label{rhotaupos}
\rho^\tau\geq c_\mathrm{min}
\end{equation} 
a.e. in $\T\times[0,T]$.
\label{prop:uniformpositivepositive}\end{proposition}

\begin{proof} We apply Theorem~\ref{th:uniformpositive} proved in Appendix~\ref{app:boundfrombelow}.
\end{proof}


\subsubsection{Regularity of $U^\tau$}\label{sec:HDUtau}

Proposition~\ref{prop:meanLinfty} and Corollary~\ref{cor:boundgradient2tau} give a control (on the expectancy) of the two quantities in \eqref{normforpos} in Proposition~\ref{prop:uniformpositivepositive}. By the Markov inequality, we have 
$$
\P\left(\iint_{Q_T}\mathbf{1}_\mathrm{det}(t)\rho^\tau|\partial_x u^\tau|^2 dx dt\geq M\;\mbox{or}\; \|u^\tau\|_{L^m(Q_T)}\geq M\right)\leq\frac{C(\eps)}{M}
$$
where the constant $C(\eps)$ depend on $\eps$ and also on $T$, $\gamma$, on the constant ${A_0}$ in \eqref{A0}, and on $\|\U_0\|_{L^\infty(\T)}$. Let $\alpha>0$. Let $M>0$ be such that $\frac{C(\eps)}{M}\leq\alpha$. Let $R>0$ be such that 
$$
\iint_{Q_T}\mathbf{1}_\mathrm{det}(t)\rho^\tau|\partial_x u^\tau|^2 dx dt\leq M
\;\mbox{and}\; \|u^\tau\|_{L^m(Q_T)}\leq M
$$
implies $c_{\mathrm{min}}>R^{-1}$. Note that $R$ is independent on $\tau$. Without loss of generality, we can also assume $R\geq R_0$, where
$$
R_0=\max\big((2\varkappa_\eps)^{\frac{1}{\theta}},(2\varkappa_\eps)^{1+\frac{1}{\theta}}\big).
$$
From the bounds from above obtained in Proposition~\ref{prop:meanLinfty}, we deduce the following result.

\begin{proposition}[$\U^\tau$ is a bounded solution] Let $\U_0\in W^{2,2}(\T)$ be such that $\rho_0\geq c_0$ a.e. in $\T$ for a given constant $c_0>0$. Assume that the growth condition \eqref{A0eps} and the localization condition \eqref{Trunceps} are satisfied and that $\U_0\in\Lambda_{\varkappa_\eps}$. Then $\U^\tau\in D_R$ with high probability, uniformly with respect to $\tau$: for all $\alpha>0$, there exists $R>0$ independent on $\tau$, such that the event "for all $t\in[0,T]$, $\U^\tau(t)\in D_R$" has probability greater then $1-\alpha$.
\label{prop:Utaubounded}\end{proposition}

We use Proposition~\ref{prop:Utaubounded} in particular to obtain some estimates on H\"older or Sobolev norms of $\U^\tau$ independent on $\tau$.
We let $T_R$ denote the exit time
\begin{equation}\label{defTRtau}
T_R=\inf\left\{t\in[0,T];\U^\tau(t)\notin D_R\right\},
\end{equation}
where $D_R$ is defined in \eqref{defDR}. By Proposition~\ref{prop:Utaubounded}, we have
\begin{equation}\label{minPTRtauT}
\lim_{R\to+\infty}\P(T_R=T)=1
\end{equation}
uniformly in $\tau$.

\begin{proposition}[Regularity of $\U^\tau$] Let $\U_0\in (W^{2,2}(\T))^2$ be such that $\rho_0\geq c_0$ a.e. in $\T$ for a given constant $c_0>0$. Assume that the growth condition \eqref{A0eps} and the localization condition \eqref{Trunceps} are satisfied and that $\U_0\in\Lambda_{\varkappa_\eps}$. Let $\U^\tau$ be the solution to \eqref{splitEulerDet}-\eqref{splitEulerSto}. Then, for all $\alpha\in(0,1/4)$, $\U^\tau(\cdot\wedge T_R)$ has a mo\-di\-fi\-ca\-tion whose trajectories are almost surely in $C^{\alpha}([0,T];L^2(\T))$ and such that
\begin{equation}\label{HolderUtauRalpha}
\E\|\U^\tau(\cdot\wedge T_R)\|_{C^{\alpha}([0,T];L^2(\T))}^2\leq C(R,\eps,T,\alpha,\U_0),
\end{equation}
where $C(R,\eps,T,\alpha)$ is a constant depending on $R$, $T$, $\eps$, $\alpha$, $\|\U_0\|_{W^{1,2}(\T)}$ but not on $\tau$. Furthermore, for every $s\in[0,1)$, $\U^\tau$ satisfies the estimate
\begin{equation}\label{H1UepstauR}
\sup_{t\in[0,T]}\E\|\U^\tau(t\wedge T_R)\|_{W^{s,2}(\T)}^2\leq C(R,\eps,T,s,\U_0),
\end{equation}
where $C(R,\eps,T,s,\U_0)$ is a constant depending on $R$, $T$, $\eps$, $s$ and $\|\U_0\|_{W^{1,2}(\T)}$ but not on $\tau$.
If additionally ${\U_\eps}_0\in W^{2,2}(\T)$ and the Lipschitz condition \eqref{Lipsigmaeps} is satisfied, then 
\begin{equation}\label{HoldertH1xBoundedtau}
\E\|\U^\tau(t\wedge T_R)\|_{C^\alpha([0,T];W^{1,2}(\T))}^2\leq C(R,\eps,T,\alpha,{\U_\eps}_0),
\end{equation}
for all $\alpha\in [0,1/4)$, and
\begin{equation}\label{LinftytH2xBoundedtau}
\sup_{t\in[0,T]}\E\|\U^\tau(t\wedge T_R)\|_{W^{2,2}(\T)}^2\leq C(R,\eps,T,{\U_\eps}_0),
\end{equation}
for some constants $C(R,\eps,T,\alpha,{\U_\eps}_0)$ and $C(R,\eps,T,{\U_\eps}_0)$ depending on $\alpha$, $R$, $T$, $\eps$, on the constant $C(\eps,R)$ in \eqref{Lipsigmaeps}, on $\|{\U_\eps}_0\|_{W^{2,2}(\T)}$, but not on $\tau$. 
\label{prop:regUtau}\end{proposition}

\begin{proof} We only give the sketch of the proof since this is very similar to the proof of Proposition~\ref{prop:regboundedeps}. First, we establish, for $\U^\tau$, an identity analogous to \eqref{MildBoundedSolution}. For Problem~\eqref{splitEulerDet} we have the mild formulation 
\begin{equation}\label{milddet}
\U^\tau(t)=S_{2\eps}(t-t_{2n})\U^\tau(t_{2n})-2\int_{t_{2n}}^{t} \partial_x S_{2\eps}(t-s)\FFF(\U^\tau(s))ds
\end{equation}
for $t_{2n}\leq t\leq t_{2n+1}$, and, for Problem~\eqref{splitEulerSto}, we have the integral formulation
\begin{equation}\label{mildsto}
\U^\tau(t)=\U^\tau(t_{2n+1})+\sqrt{2}\int_{t_{2n+1}}^{t} \mathbf{\Psi}^{\eps,\tau}(\U^\tau(s))\,d W(s),
\end{equation}
for $t_{2n+1}\leq t\leq t_{2n+2}$. By combining \eqref{milddet} and \eqref{mildsto}, we obtain the identity
\begin{multline}\label{MildSolutiontau}
\U^\tau(t)=S_\eps(t_\sharp)\U_0-\int_0^{t_\sharp} \partial_x S_\eps(t_\sharp-s)\FFF(\U^\tau(s_\flat))ds\\
+\sqrt{2}\int_0^{t} \mathbf{1}_\mathrm{sto}(s)S_\eps(t_\sharp-s_\sharp)\mathbf{\Psi}^{\eps,\tau}(\U^\tau(s))\,d W(s),
\end{multline}
where we have set
$$
t_\sharp=\min(2t-t_{2n},t_{2n+2}),\quad t_\flat=\frac{t+t_{2n}}{2},\quad t_{2n}\leq t< t_{2n+2}.
$$
Then we proceed as in the proof of Proposition~\ref{prop:regboundedeps}. Note that $t\mapsto t_\sharp$ is $2$-Lipschitz continuous and that we have the control~\eqref{BoundTruncepstau}, therefore (compare with \eqref{HolderURk}), $\U^\tau$ satisfies 
\begin{equation}\label{HolderUtauRk}
\E\|\U^\tau(t\wedge T_R)-\U^\tau(s\wedge T_R)\|_{L^2(\T)}^{2k}\leq C(R,\eps,T,k)\max\left((t-s)^{k/2},(t-s)^k\right),
\end{equation}
for all $0\leq s\leq t\leq T$, where $C(R,\eps,T,k,{\U_\eps}_0)$ is a constant depending on $R$, $T$, $\eps$, $k$, $\|\U_0\|_{W^{1,2}(\T)}$ but not on $\tau$. This gives \eqref{HolderUtauRalpha} by the Kolmogorov's criterion.\medskip

To obtain the regularity in $x$ \eqref{H1UepstauR}, we also proceed as in the proof of Proposition~\ref{prop:regboundedeps}. 
Let $s \in [0,1)$. The estimates \eqref{xTdet1}-\eqref{xTsto1} hold true here: the dependence on time being slightly different between \eqref{MildBoundedSolution} and \eqref{MildSolutiontau}, the bounds \eqref{toxTdet1} and \eqref{toxTsto1} have to be replaced by
\begin{equation}\label{toxTdet1tau}
\int_0^{t_\sharp}(t_\sharp-r)^{-\frac{1+s}{2}}dr\ C(R,\eps),
\end{equation}
and
\begin{equation}\label{toxTsto1tau}
\int_0^t\mathbf{1}_\mathrm{sto}(r)(t_\sharp-r_\sharp)^{-s}dr\ C(R,\eps),
\end{equation}
respectively. In \eqref{toxTdet1tau}, we have
$$
\int_0^{t_\sharp}(t_\sharp-r)^{-\frac{1+s}{2}}dr\leq \frac{2}{1-s}T^{\frac{1-s}{2}},
$$
while, for $t_{2n}\leq t\leq t_{2n+2}$ (and thus $2n\tau\leq T$), the integral term in \eqref{toxTsto1tau} is
\begin{align*}
\int_0^t\mathbf{1}_\mathrm{sto}(r)(t_\sharp-r_\sharp)^{-s}dr&=\sum_{k=1}^n \tau (t_{2k})^{-s}\leq C(s) T^{1-s},
\end{align*}
where $C(s)$ depends on $s$ only. The proof of \eqref{HoldertH1xBoundedtau}-\eqref{LinftytH2xBoundedtau} is similar to the proof of the estimates  \eqref{HoldertH1xBounded}-\eqref{LinftytH2xBounded} for the solution to \eqref{stoEulereps}, \textit{cf.} the proof of Proposition~\ref{prop:regboundedeps}.
\end{proof}

\subsubsection{Compactness argument}\label{subsec:compact}

We introduce the independent processes $X^\tau_1,X^\tau_2,\ldots$ defined by
$$
X^\tau_k(t)=\sqrt{2}\int_0^t \mathbf{1}_{\mathrm{sto}}(s)d\beta_k(s)
$$
and set 
\begin{equation}\label{defWtau}
W^\tau(t)=\sum_{k\geq 1} X^\tau_k(t) e_k.
\end{equation}
The random variable $X_k(t)$ is Gaussian, with mean $0$ and variance given by
$$
\sigma_\tau^2(t)=t_{2n+1}+2(t-t_{2n+2}),\quad t\in[t_{2n+1},t_{2n+2}].
$$
Let $0\leq s_1\leq\ldots\leq s_m\leq T$ be $m$ given points in $[0,T]$. We have $|\sigma_\tau^2(t)-t|\leq\tau$ for all $t\in[0,T]$, therefore the finite dimensional distribution of $(X^\tau_1(s_j))_{1,m}$ converges in law to the distribution of $(\beta_1(s_i))_{1,m}$ when $\tau\to0$. Besides, $(X^\tau_1)$ is tight in $C([0,T])$ since $\E \|X^\tau_1\|_{C^\alpha([0,T])}$ is bounded uniformly with respect to $\tau$ for any $\alpha<1/2$. By \cite[Theorem 8.1]{BillingsleyBook}, $(X_1^\tau)$ converges in law to $\beta_1$ on $C([0,T])$. We have the same result $X^\tau_k\to\beta_k$ in law for each $k\geq 2$, since the distributions are all the same.\medskip

Let $\mathfrak{U}_0$ be defined by \eqref{defUUU0} and let
\begin{equation}\label{WpathSpace}
\mathcal{X}_W=C\big([0,T];\mathfrak{U}_0\big)
\end{equation}
denote the path space of $W^\tau$. Since the embedding $\mathfrak{U}\hookrightarrow\mathfrak{U}_0$ is Hilbert-Schmidt, the $\mathcal{X}_W$-valued process $W^\tau$ converges in law to $W$ when $\tau\to 0$ (again, we can use \cite[Theorem 8.1]{BillingsleyBook}). 
\medskip

Define the path space $\mathcal{X}=\mathcal{X}_\U\times\mathcal{X}_W$, where
$$
\mathcal{X}_\U=C\big([0,T];L^2(\T)\big).
$$
Let us denote by $\mu^\tau_\U$ the law of $\U^\tau$ on $\mathcal{X}_\U$. The joint law of $\U^\tau$ and $W^\tau$ on $\mathcal{X}$ is denoted by $\mu^\tau$.

\begin{proposition}[Tightness of $(\mu^\tau)$] Let $\U_0\in W^{2,2}(\T)$ be such that $\rho_0\geq c_0$ a.e. in $\T$ for a given constant $c_0>0$. Assume that the growth condition \eqref{A0eps} and the localization condition \eqref{Trunceps} are satisfied and that $\U_0\in\Lambda_{\varkappa_\eps}$. Let $\U^\tau$ be the solution to \eqref{splitEulerDet}-\eqref{splitEulerSto}. Then the set $\{\mu^\tau;\,\tau\in(0,1)\}$ is tight and therefore relatively weakly compact in the set of probability measures on $\mathcal{X}$.
\label{prop:tight}\end{proposition}

\begin{proof}
First, we prove tightness of $\{\mu_\U^\tau;\,\tau\in(0,1)\}$ in $\mathcal{X}_\U$. Let $\alpha\in(0,1/4)$ and $s\in (0,1)$. Then 
$$
K_{M}:=\big\{\U\in\mathcal{X}_\U ;\|\U\|_{C^\alpha([0,T];L^2(\T))}+\|\U\|_{L^2([0,T];W^{s,2}(\T))}\leq M\big\}
$$
is compact in $\mathcal{X}_\U$ \cite{Simon87}. Recall that the stopping time $T_R$ is defined by \eqref{defTRtau}. 
Note also that a consequence of the $L^\infty_t$-estimate \eqref{H1UepstauR}, is the $L^2_t$-estimate 
\begin{equation*}
\E\int_0^T \|\U^\tau(t\wedge T_R)\|_{W^{s,2}(\T)}^2 dt\leq C(R,\eps,T,s,\U_0),
\end{equation*}
which gives
\begin{equation}\label{H1UepstauRL2}
\E\|\U^\tau(t\wedge T_R)\|_{L^2(0,T;W^{s,2}(\T))}^2\leq C(R,\eps,T,s,\U_0),
\end{equation}
by the Fubini Theorem. By \eqref{HolderUtauRalpha}, \eqref{H1UepstauRL2} and the Markov inequality, we obtain the estimate  
\begin{align*}
\P(\U^\tau\notin K_M)&\leq  \P(T_R<T)+\P(\U^\tau\notin K_M\;\&\; T_R=T)\\
&\leq \P(T_R<T)+\frac{C(R,\eps,T,\alpha,s,\U_0)}{M^2}.
\end{align*}
By \eqref{minPTRtauT}, given $\eta>0$ there exists $R,M>0$ independent on $\tau$ such that
$$
\mu_\U^\tau(K_{M})\geq 1-\eta,
$$
\textit{i.e.} $(\mu_\U^\tau)$ is tight. We have proved above that the law $\mu_{W^\tau}$ is tight. The set of the joint laws $\{\mu^\tau;\,\tau\in(0,1)\}$ is therefore tight. By Prohorov's theorem, it is relatively weakly compact. 
\end{proof}
\bigskip

Let now $(\tau_n)$ be a sequence decreasing to $0$. Up to a subsequence, there is a probability measure $\mu_\eps$ on $\mathcal{X}$ such that $(\mu^{\tau_n})$ converges weakly to $\mu$. By the Skorohod Theorem \cite[p.~70]{BillingsleyBook}, we can assume almost sure convergence of the random variables by changing the probability space.

\begin{proposition}\label{prop:Skorohod}
There exists a probability space $(\tilde{\Omega}^\eps,\tilde{\mathcal{F}}^\eps,\tilde{\P}^\eps)$, a sequence of $\mathcal{X}$-valued random variables $(\tilde{\U}^{\tau_n},\tilde{W}^{\tau_n})_{n\in\N}$ and a $\mathcal{X}$-valued random variable $(\tilde{\U}_\eps,\tilde{W}_\eps)$ such that
\begin{enumerate}
 \item the laws of $(\tilde{\U}^{\tau_n},\tilde{W}^{\tau_n})$ and $(\tilde{\U}_\eps,\tilde{W}_\eps)$ under $\,\tilde{\P}^\eps$ coincide with $\mu^{\tau_n}$ and $\mu_\eps$ respectively,
 \item $(\tilde{\U}^{\tau_n},\tilde{W}^{\tau_n})$ converges $\,\tilde{\P}^\eps$-almost surely to $(\tilde{\U}_\eps,\tilde{W}_\eps)$ in the topology of $\mathcal{X}$.
\end{enumerate}
\end{proposition}

We had dropped the variable $\eps$ in most of the quantities defined by the splitting scheme, in particular $\U^\tau$, $W^\tau$, etc. We reintroduce the dependence on $\eps$ for the limits $\U_\eps$, $W_\eps$ etc. to indicate the relation to Problem~\eqref{stoEulereps}.

\subsubsection{Identification of the limit}\label{subsec:identif}

Our aim in this section is to identify the limit $(\tilde{\U}_\eps,\tilde{W}_\eps)$ given by Proposition~\ref{prop:Skorohod}. \medskip

Let $(\tilde{\mathcal{F}}^\eps_t)$ be the $\tilde{\P}^\eps$-augmented canonical filtration of the process $(\tilde{\U}_\eps,\tilde{W}_\eps)$, \textit{i.e.}
$$
\tilde{\mathcal{F}}^\eps_t=\sigma\big(\sigma\big(\varrho_t\tilde{\U}_\eps,\varrho_t\tilde{W}_\eps\big)\cup\big\{N\in\tilde{\mathcal{F}^\eps};\;\tilde{\P}^\eps(N)=0\big\}\big),\quad t\in[0,T],
$$
where $\varrho_t$ is the operator of restriction to the interval $[0,t]$ defined as follows: if $E$ is a Banach space and $t\in[0,T]$, then
\begin{equation}\label{restr}
\begin{split}
\varrho_t:C([0,T];E)&\longrightarrow C([0,t];E)\\
k&\longmapsto k|_{[0,t]}.
\end{split}
\end{equation}

\begin{proposition}[Martingale solution to \eqref{stoEulereps}]\label{prop:martsoleps}
The sextuplet $$\big(\tilde{\Omega}^\eps,\tilde{\mathcal{F}}^\eps,(\tilde{\mathcal{F}}^\eps_t),\tilde{\P}^\eps,\tilde{W}_\eps,\tilde{\U}_\eps\big)$$
is a martingale bounded solution to \eqref{stoEulereps}. 
\end{proposition}

By martingale bounded solution, we mean the following: 
$$
\big(\tilde{\Omega}^\eps,\tilde{\mathcal{F}}^\eps,(\tilde{\mathcal{F}}^\eps_t),\tilde{\P}^\eps,\tilde{W}_\eps\big)
$$
is a stochastic basis and $\tilde{\U}_\eps$ is a bounded solution, in the sense of Definition~\ref{def:pathsoleps}, to \eqref{stoEulereps} after the substitution
$$
\big(\Omega,\mathcal{F},(\mathcal{F}_t),\P,W,\U_\eps\big)\leftarrow\big(\tilde{\Omega}^\eps,\tilde{\mathcal{F}}^\eps,(\tilde{\mathcal{F}}^\eps_t),\tilde{\P}^\eps,\tilde{W}_\eps,\tilde{\U}_\eps\big).
$$
This substitution leaves invariant the law of the resulting process $(\U_\eps(t))$.\medskip

The proof of Proposition~\ref{prop:martsoleps} uses a method of construction of martingale solutions to SPDEs that avoids in part the use of representation Theorem. This technique has been developed in Ondrej\'at \cite{Ondrejat10}, Brze\'zniak, Ondrej\'at \cite{BrzezniakOndrejat11} and used in particular in Hofmanov\'a, Seidler \cite{HofmanovaSeidler12} and in \cite{Hofmanova13b,DebusscheHofmanovaVovelle15}.
\medskip

Note first that item \textit{1}. in Definition~\ref{def:pathsoleps} is satisfied by the choice of the path space $\mathcal{X}_\U$ and that item \textit{2.} follows from the convergence of $(\tilde{\U}^{\tau_n})$ to $(\tilde{\U}_\eps)$ $\,\tilde{\P}^\eps$-almost surely in the topology of $\mathcal{X}_\U$ and from \eqref{minPTRtauT}. Our objective is now to prove item \textit{3.} in Definition~\ref{def:pathsoleps}. \medskip

Recall that $\FFF$, the flux function in Equation~\eqref{stoEuler}, is defined by \eqref{defUUU}. 
Let us define for all $t\in[0,T]$ and a test function $\varphi=(\varphi_1,\varphi_2)\in C^\infty(\T;\R^2)$,
\begin{equation*}
\begin{split}
M^\tau(t)&=\big\langle \U^\tau(t),\varphi\big\rangle-\big\langle \U_{\eps 0},\varphi\big\rangle-2\int_0^t \mathbf{1}_\mathrm{det}(s)\big\langle\FFF(\U^\tau),\partial_x \varphi\big\rangle+\eps\big\langle\U^\tau,\partial^2_{xx} \varphi\big\rangle\,d s,\\
\tilde{M}^\tau(t)&=\big\langle \tilde{\U}^\tau(t),\varphi\big\rangle-\big\langle \tilde{\U}_{\eps 0},\varphi\big\rangle-2\int_0^t \mathbf{1}_\mathrm{det}(s)\big\langle\FFF(\tilde\U^\tau),\partial_x \varphi\big\rangle+\eps\big\langle\tilde\U^\tau,\partial^2_{xx} \varphi\big\rangle\,d s,\\
\tilde{M}_\eps(t)&=\big\langle \tilde{\U}_\eps(t),\varphi\big\rangle-\big\langle \tilde{\U}_{\eps 0},\varphi\big\rangle-\int_0^t\big\langle\FFF(\tilde{\U}_\eps),\partial_x \varphi\big\rangle+\big\langle\tilde{\U}_\eps,\partial^2_{xx} \varphi\big\rangle\,d s.
\end{split}
\end{equation*}
The proof of Proposition~\ref{prop:martsoleps} will be a consequence of the following two lemmas.

\begin{lemma}\label{lem:tildeW}
The process $\tilde{W}_\eps$ has a modification which is a $(\tilde{\mathcal{F}}^\eps_t)$-adapted  $\mathfrak{U}_0$-cylindrical Wiener process, and there exists a collection of mutually independent real-valued $(\tilde{\mathcal{F}}^\eps_t)$-Wiener processes $\{\tilde{\beta}_k^\eps\}_{k\geq1}$ such that 
\begin{equation}\label{Wienertildeeps}
\tilde{W}_\eps=\sum_{k\geq1}\tilde{\beta}_k^\eps e_k
\end{equation}
in $C([0,T];\mathfrak{U}_0)$
\end{lemma}

\begin{proof} it is clear that $\tilde{W}_\eps$ is a $\mathfrak{U}_0$-cylindrical Wiener process (this notion is stable by convergence in law; actually it can be characterized in terms of the law of $\tilde{W}_\eps$ uniquely if we drop the usual hypothesis of a.s. continuity of the trajectories. This latter can be recovered, after a possible modification of the process, by using the Kolmogorov's Theorem). Also $\tilde{W}_\eps$ is
$(\tilde{\mathcal{F}}^\eps_t)$-adapted by definition of the filtration $(\tilde{\mathcal{F}}^\eps_t)$. By \cite[Proposition~4.1]{DaPratoZabczyk92}, we obtain the decomposition~\eqref{Wienertildeeps}.
\end{proof}

\begin{lemma}\label{lem:tildeM}
The processes $\tilde{M}_\eps$,
\begin{equation*}
\tilde{M}^2_\eps-\sum_{k\geq1}\int_0^\tec\big\langle \sigma_k^{\eps}(\tilde{\U}_\eps),\varphi_2\big\rangle^2\,d r\quad\mbox{and}\quad\tilde{M}_\eps\tilde{\beta}^\eps_k-\int_0^\tec\big\langle \sigma_k^{\eps}(\tilde{\U}_\eps),\varphi_2\big\rangle\,d r,
\end{equation*}
are $(\tilde{\mathcal{F}}^\eps_t)$-martingales. 
\end{lemma}

\begin{proof} We fix some times $s,t\in[0,T],\,s\leq t,$ and a continuous function
$$
\gamma:C\big([0,s];L^2(\T)\big)\times C\big([0,s];\mathfrak{U}_0\big)\longrightarrow [0,1].
$$
Let us set $\tilde{X}^\tau_k=\<\tilde{W}^\tau,e_k\>_\mathfrak{U_0}$. For all $\tau\in(0,1)$, the process
$$
M^\tau=\sum_{k\geq1}\int_0^\tec \big\langle \sigma_k^{\tau}(\U^\tau),\varphi_2\big\rangle\,dX_k^\tau(r)
$$
is a square integrable $(\mathcal{F}_t)$-martingale and therefore 
$$
M^\tau_2:=(M^\tau)^2-\sum_{k\geq1}\int_0^\tec \big\langle \sigma_k^{\tau}(\U^\tau),\varphi_2\big\rangle^2\,d \<\!\<X^\tau\>\!\>(r),
$$
and
$$
M^\tau_3:=M^\tau\beta_k-\int_0^\tec \big\langle \sigma_k^{\tau}(\U^\tau),\varphi_2\big\rangle\,d \<\!\<X^\tau\>\!\>(r)
$$
are $(\mathcal{F}_t)$-martingales, where we have denoted by 
$$
\<\!\<X^\tau\>\!\>(t)=2\int_0^t \mathbf{1}_\mathrm{sto}(r)dr
$$
the quadratic variation of $X_k^\tau$ (note that $\<\!\<X^\tau\>\!\>(t)\to t$ when $\tau\to 0$). Besides, it follows from the equality of laws that
\begin{equation*}
\tilde{\E}^\eps\,\gamma\big(\varrho_s \tilde{\U}^\tau,\varrho_s\tilde{W}^\tau\big)\big[\tilde{M}^\tau(t)-\tilde{M}^\tau(s)\big]
=\E\,\gamma\big(\varrho_s \U^\tau,\varrho_s W^\tau\big)\big[M^\tau(t)-M^\tau(s)\big].
\end{equation*}
hence 
$$
\tilde{\E}^\eps\,\gamma\big(\varrho_s \tilde{\U}^{\tau_n},\varrho_s\tilde{W}^{\tau_n}\big)\big[\tilde{M}^{\tau_n}(t)-\tilde{M}^{\tau_n}(s)\big]=0,
$$
for all $n$. We can pass to the limit in this equation, due to the moment estimates \eqref{estimmomentepstau2} and the Vitali convergence theorem. We obtain
\begin{equation*}
\tilde{\E}^\eps\,\gamma\big(\varrho_s\tilde{\U}_\eps,\varrho_s\tilde{W}_\eps\big)\big[\tilde{M}_\eps(t)-\tilde{M}_\eps(s)\big]=0,
\end{equation*}
\textit{i.e.} $\tilde{M}_\eps$ is a $(\tilde{\mathcal{F}}^\eps_t)$-martingale. We proceed similarly to show that 
$$
\tilde{M}_{\eps 2}:=\tilde{M}_\eps^2-\sum_{k\geq1}\int_0^\tec\big\langle \sigma_k^{\eps}(\tilde{\U}_\eps),\varphi_2\big\rangle^2\,d r
$$
is a $(\tilde{\mathcal{F}}^\eps_t)$-martingale, by passing to the limit in the identity
$$
\tilde{\E}^\eps\,\gamma\big(\varrho_s \tilde{\U}^\tau,\varrho_s\tilde{W}^\tau\big)\big[\tilde{M}_2^\tau(t)-\tilde{M}_2^\tau(s)\big]=0,
$$
and again similarly for
$$
\tilde{M}_{\eps 3}:=\tilde{M}_\eps\tilde{\beta}^\eps_k-\int_0^\tec\big\langle \sigma_k^{\eps}(\tilde{\U}_\eps),\varphi_2\big\rangle\,d r.
$$
This concludes the proof of Lemma~\ref{lem:tildeM}.
\end{proof}

\begin{proof}[Proof of Proposition \ref{prop:martsoleps}]
Once the above lemmas are established, we can show that
\begin{equation}\label{EqA3Martina}
\tilde{\E}^\eps\left[\left(\tilde{M}_\eps(t)-\tilde{M}_\eps(s)\right)\int_s^t
\<hd\tilde W_\eps(\sigma),\varphi_2\>
-\sum_{k\geq 1}\int_s^t \<h(\sigma)e_k,\varphi_2\>\< \sigma_k^{\eps}(\tilde{\U}_\eps)(\sigma),\varphi_2\>d\sigma \Big| \tilde{\mathcal{F}}^\eps_s\right]=0,
\end{equation}
for all $(\tilde{\mathcal{F}}^\eps_t)$-predictable, $L_2(\mathfrak{U},L^2(\T))$-valued process satisfying
\begin{equation}\label{normHSh}
\int_0^T\|h(\sigma)\|^2_{L_2(\mathfrak{U},L^2(\T))} d\sigma<+\infty.
\end{equation}
Here, if $H$ is a given Hilbert space, $L_2(\mathfrak{U},H)$ is the set of Hilbert-Schmidt operators $\mathfrak{U}\to H$. In particular, in \eqref{normHSh}, we have
$$
\|h(\sigma)\|^2_{L_2(\mathfrak{U},L^2(\T))}=\sum_{k\geq 1}\|h(\sigma)e_k\|_{L^2(\T)}^2.
$$ 
Equation~\eqref{EqA3Martina} is proved in \cite[Proposition~A.1]{Hofmanova13b}. Taking $s=0$ and $h= \Phi^\eps(\tilde{\U}_\eps)$ in \eqref{EqA3Martina}, we obtain
$$
\tilde{\E}^\eps\sum_{k\geq 1}\left[\tilde{M}_\eps(t)\int_0^t\<\sigma_k^{\eps}(\tilde{\U}_\eps),\varphi_2\>ds-\int_0^t \< \sigma_k^{\eps}(\tilde{\U}_\eps),\varphi_2\>^2 ds\right]=0.
$$
This shows that
\begin{equation}\label{EqA3Martina2}
\tilde{\E}^\eps\left[\tilde{M}_\eps(t)-\sum_{k\geq 1}\int_0^t\big\langle\sigma_k^\eps(\tilde{\U}_\eps)\,d\tilde\beta_k^\eps(s),\varphi_2\big\rangle\right]^2=0.
\end{equation}
Accordingly, we have
\begin{align*}
\big\langle \tilde{\U}_\eps(t),\varphi\big\rangle=\big\langle \tilde{\U}_{\eps 0},\varphi\big\rangle+\int_0^t&\big\langle\FFF(\tilde{\U}_\eps),\partial_x \varphi\big\rangle+\big\langle \tilde{\U}_\eps,\partial^2_x\varphi\big\rangle\,d s\\
&+\sum_{k\geq 1}\int_0^t\big\langle\sigma_k^\eps(\tilde{\U}_\eps)\,d\tilde{\beta}_k^\eps,\varphi_2\big\rangle,\quad t\in[0,T],\;\;\tilde{\P}^\eps\text{-a.s.},
\end{align*}
and this gives the weak formulation \eqref{EqBoundedSolution} $\tilde{\P}^\eps$-almost surely. By Proposition~\ref{prop:Utaubounded}, we have point \textit{2.} of Definition~\ref{def:pathsoleps}. This concludes the proof of Proposition~\ref{prop:martsoleps}.
\end{proof}

\subsubsection{Proof of Theorem~\ref{th:existspatheps}}\label{subsec:prooftheps}

We apply the Gy{\"o}ngy-Krylov argument~\cite{GyongyKrylov96}, see also \cite[Section 4.5]{Hofmanova13b}, which shows that the existence of a martingale solution and uniqueness of pathwise solutions (Theorem~\ref{th:uniqpatheps}) give existence and uniqueness of pathwise solutions and convergence in probability in $\mathcal{X}_\U=C([0,T);L^2(\T))$ of the whole sequence $(\U^{\tau_n})$ to $\U_\eps$. If $\U\mapsto J(\U)\in[0,+\infty]$ is a lower semi-continuous functional on the space $\mathcal{X}$, then $\U\mapsto \E J(\U)$ is a lower semi-continuous functional on the space $L^1(\Omega;\mathcal{X})$ endowed with the topology of convergence in probability. To prove this fact we apply the inequality
$$
\E J(\U)\leq \E\left(\mathbf{1}_{\|\U-\U^n\|\leq \eps} J(U)\right)+\P\left(\|\U-\U^n\|>\eps\right).
$$
In particular the moment estimate \eqref{estimmomenteps2} follows from the moment estimate \eqref{estimmomentepstau2} for $\U^\tau$ and the gradient estimates \eqref{corestimgradientepsrho} and \eqref{corestimgradientepsu} are deduced from the corresponding estimates \eqref{corestimgradientepsrhotau} and \eqref{corestimgradientepsutau} satisfied by $\U^\tau$. Also we have the regularity \eqref{HoldertH1xBounded}-\eqref{LinftytH2xBounded} as a consequence of \eqref{HoldertH1xBoundedtau}-\eqref{LinftytH2xBoundedtau}. By \eqref{HoldertH1xBoundedtau}-\eqref{LinftytH2xBoundedtau} we also have, up to a subsequence, and in probability, convergence of $\U^{\tau_n}$ to $\U_\eps$ in $C([0,T];W^{1,2}(\T))$. This convergence is strong enough to obtain the entropy balance equation \eqref{Itoentropyeps} by taking the limit in Equation~\eqref{Itoentropytau}. This concludes the proof of Theorem~\ref{th:existspatheps}.

\subsection{Additional estimates}\label{sec:contestimate}

\subsubsection{Moments of the entropy and entropy flux}\label{sec:momentEntropy}
\begin{proposition}[Moments of the entropy and entropy flux] Let  ${\U_\eps}_0\in W^{2,2}(\T)$ sa\-tis\-fy ${\rho_\eps}_0\geq c_{\eps 0}$ a.e. in $\T$, for a positive constant $c_{\eps 0}$. Let $p\in\N$ satisfy $p\geq 4+\frac{1}{2\theta}$. Assume that hypotheses \eqref{A0eps}, \eqref{Trunceps}, \eqref{Lipsigmaeps} are satisfied, that ${\U_\eps}_0\in\Lambda_{\varkappa_\eps}$ and that
\begin{equation}\label{initEntropyMomentEntropy}
\E\int_\T \left(\eta_0(\U_{\eps 0})+\eta_{2p}(\U_{\eps 0}) \right)dx
\end{equation}
is bounded uniformly with respect to $\eps$. Let $\U_\eps$ be the bounded solution to \eqref{stoEulereps}. Let $g$ be a subquadratic convex function (\textit{i.e.} $g$ satisfies \eqref{gsubquad}) and let $(\eta,H)$ be the entropy-entropy flux pair associated to $g$ by \eqref{entropychi}-\eqref{entropychiflux}. Let $s\geq 1$ satisfy
\begin{equation}\label{pVSs}
p\geq \frac32 s+\frac{s-1}{2\theta}.
\end{equation}
Then $(\eta(\U_\eps))$ and $(H(\U_\eps))$ are uniformly bounded in $L^s(\Omega;L^s(Q_T))$
\label{prop:momentEntropy}\end{proposition}

\begin{proof} Let $s_0=2\frac{4\theta+1}{3\theta+1}$. Note that $s_0>2$. By Lemma~\ref{lemmaentropies2}, we have, under condition \eqref{pVSs}, 
\begin{equation}\label{momentseta}
|\eta(\U)|^s,\; |H(\U)|^s\leq C_s(\eta_0(\U)+\eta_{p}(\U)),\quad 1\leq s\leq s_0,
\end{equation}
for all $\U\in\R_+\times\R$, where $C_s$ is constant depending on $\gamma$, $s$, $p$ only. By \eqref{initEntropyMomentEntropy} and the estimate \eqref{estimmomenteps2} on the moments of $\U_\eps$, we deduce that $\eta(\U_\eps)$ and $H(\U_\eps)$ are uniformly bounded in $L^s(\Omega;L^s(Q_T))$. 
\end{proof}
\subsubsection{Bound in $C([0,T];W^{-2,2}(\T))$}\label{sec:BoundHminus2}

In the following statement, $W^{-2,2}(\T)$ denotes the dual to the space $W^{2,2}(\T)$.

\begin{proposition}[Additional continuity estimate] Let  ${\U_\eps}_0\in W^{2,2}(\T)$ sa\-tis\-fy ${\rho_\eps}_0\geq c_{\eps 0}$ a.e. in $\T$, for a positive constant $c_{\eps 0}$. Let $p\in\N$ satisfy $p\geq 4+\frac{1}{2\theta}$. Assume that hypotheses \eqref{A0eps}, \eqref{Trunceps}, \eqref{Lipsigmaeps} are satisfied, that ${\U_\eps}_0\in\Lambda_{\varkappa_\eps}$ and that
\begin{equation}\label{initEntropyAddCont}
\E\int_\T \left(\eta_0(\U_{\eps 0})+\eta_{2p}(\U_{\eps 0}) \right)dx
\end{equation}
is bounded uniformly with respect to $\eps$. Let $\U_\eps$ be the bounded solution to \eqref{stoEulereps}. Let $g$ be a subquadratic convex function (\textit{i.e.} $g$ satisfies \eqref{gsubquad}) and let $(\eta,H)$ be the entropy-entropy flux pair associated to $g$ by \eqref{entropychi}-\eqref{entropychiflux}. Let $B_\eps(t)$ be the random distribution
\begin{align}
B_\eps(t)=&\eta({\U_\eps}_0)+\int_0^t\left[-\partial_x H({\U_\eps})+\eps\partial^2_{xx}\eta({\U_\eps})\right]d s\nonumber\\
&+\int_0^t \eta'({\U_\eps})\mathbf{\Psi}^{\eps}({\U_\eps})\,d W(s)
+ \frac{1}{2}\int_0^t\GG^{\eps}({\U_\eps})^2\partial^2_{qq} {\eta}({\U_\eps})ds.\label{Itoentropyepseps}
\end{align}
Then, for all $\alpha\in(0,1/4)$, the $W^{-2,2}(\T)$-valued process $(B_\eps(t))$ has a modification which has almost surely $\alpha$-H\"older trajectories and satisfies
\begin{equation}\label{timeeps}
\E\|B_\eps\|_{C^\alpha([0,T];W^{-2,2}(\T))}^2=\mathcal{O}(1),
\end{equation}
where $\mathcal{O}(1)$ depends on $\gamma$, $T$, $p$, on the constant $A_0$ in \eqref{A0eps} and on the bound on \eqref{initEntropyAddCont} only.
\label{prop:regteps}\end{proposition}

\begin{proof} Let $\varphi\in W^{2,2}(\T)$ such that $\|\varphi\|_{W^{2,2}(\T)} \leq 1$. For $0\leq s\leq t\leq T$, the increment $\<B_\eps(t)-B_\eps(s),\varphi\>_{W^{-2,2}(\T),W^{2,2}(\T)}$ is the sum of various terms, which we denote by $D_\eps^j(s,t)$, $j=1,\ldots,4$. The first term is 
$$
D_\eps^1(s,t)=\int_s^t \<H(\U_\eps(\sigma)),\partial_x\varphi\>_{L^2(\T)} d\sigma.
$$
By \eqref{momentseta} and \eqref{estimmomenteps2}, we have 
$$
\E\sup_{\sigma\in[0,T]}\|H(\U_\eps(\sigma))\|_{L^2(\T)}^2=\mathcal{O}(1).
$$
It is easy to deduce from this estimate the bound
$$
\E|D_\eps^1(s,t)|^4=\mathcal{O}(1)(t-s)^4.
$$
We obtain the same bounds for $D_\eps^j(s,t)$, $j=2,4$, where
\begin{equation*}
D_\eps^2(s,t)=\int_s^t \<\eps\eta(\U_\eps(\sigma)),\partial^2_{xx}\varphi\>_{L^2(\T)} d\sigma,\quad
D_\eps^4(s,t)=\frac{1}{2}\int_s^t\<\GG^{\eps}({\U_\eps})^2\partial^2_{qq} {\eta}({\U_\eps}),\varphi\>_{L^2(\T)}d\sigma.
\end{equation*}
To treat the term $D_\eps^4(s,t)$, we use in particular the estimates \eqref{dqqetaG} (with $m=1$), \eqref{momentseta} and \eqref{estimmomenteps2}, which give 
$$
\E\sup_{\sigma\in[0,T]}\|\GG^{\eps}({\U_\eps})^2\partial^2_{qq} {\eta}({\U_\eps})\|_{L^2(\T)}^2(\sigma)=\mathcal{O}(1).
$$
Eventually, by \eqref{dqetaG} (with $m=1$), \eqref{momentseta} and \eqref{estimmomenteps2} and the Burkholder-Davis-Gundy Inequality, we obtain
$$
\E|D_\eps^3(s,t)|^4=\mathcal{O}(1)(t-s)^2,
$$
where
$$
D_\eps^3(s,t)=\int_s^t \<\eta'({\U_\eps})\mathbf{\Psi}^{\eps}({\U_\eps}),\varphi\>_{L^2(\T)}\,d W(\sigma).
$$
We conclude by the Kolmogorov Theorem.
\end{proof}

Let us apply Proposition~\ref{prop:regteps} with $g(\xi)=1$ and $g(\xi)=\xi$, respectively. In that case $\eta(\U)=\rho$, respectively $\eta(\U)=q$. Therefore, we obtain the following statement.

\begin{corollary}[Additional continuity estimate, corollary] Let  ${\U_\eps}_0\in W^{2,2}(\T)$ sa\-tis\-fy ${\rho_\eps}_0\geq c_{\eps 0}$ a.e. in $\T$, for a positive constant $c_{\eps 0}$. Let $p\in\N$ satisfy $p\geq 4+\frac{1}{2\theta}$. Assume that hypotheses \eqref{A0eps}, \eqref{Trunceps}, \eqref{Lipsigmaeps} are satisfied, that ${\U_\eps}_0\in\Lambda_{\varkappa_\eps}$ and that
\begin{equation}\label{initEntropyAddContCor}
\E\int_\T \left(\eta_0(\U_{\eps 0})+\eta_{2p}(\U_{\eps 0}) \right)dx
\end{equation}
is bounded uniformly with respect to $\eps$. Let $\U_\eps$ be the bounded solution to \eqref{stoEulereps}. Then, for all $\alpha\in(0,1/4)$, 
\begin{equation}\label{timeepsCor}
\E\|\U_\eps\|_{C^\alpha([0,T];W^{-2,2}(\T))}^2=\mathcal{O}(1),
\end{equation}
where $\mathcal{O}(1)$ depends on $\gamma$, $T$, $p$, on the constant $A_0$ in \eqref{A0eps} and on the bound on \eqref{initEntropyAddCont} only.
\label{cor:regteps}\end{corollary}

\subsubsection{Bound on the viscosity term}\label{sec:BoundMeasure}

Let $\mathcal{M}_b(\overline{Q_T})$ denote the set of bounded Borel measures on $\overline{Q_T}$ and $\mathcal{M}_b^+(\overline{Q_T})$ denote the subset of non-negative bounded measures. 

\begin{proposition} Under the hypotheses of Proposition~\ref{prop:regteps}, the random measure $e^\eps$ on $\overline{Q_T}$ defined by
\begin{equation}\label{defeeps}
\<e^\eps,\varphi\>_{\mathcal{M}_b(\overline{Q_T}),C(\overline{Q_T})}=\iint_{Q_T} \eps\eta''(\U_{\eps})\cdot({\partial_x\U_{\eps}},{\partial_x\U_{\eps}})\varphi(x,t) dx dt
\end{equation}
is uniformly bounded in $L^2(\Omega;\mathcal{M}_b^+(\overline{Q_T}))$. 
\label{prop:bounde}\end{proposition}

\begin{proof} We apply the entropy balance equation \eqref{Itoentropyeps} with $\varphi\equiv 1$ and $t=T$. We obtain then, with the notations of Proposition~\ref{prop:regteps}, 
\begin{equation}\label{etaeeps}
\|\eta(\U_\eps)(T)\|_{L^1(\T)}+\|e^\eps\|_{\mathcal{M}_b(Q_T)}=\<B_\eps(T),\varphi\>_{W^{-2,2}(\T),W^{2,2}(\T)}.
\end{equation}
By \eqref{momentseta} and \eqref{estimmomenteps2}, we have $\E\|\eta(\U_\eps)(T)\|_{L^1(\T)}^2=\mathcal{O}(1)$. By
\eqref{timeeps}, we deduce from \eqref{etaeeps} that $\E\|e^\eps\|^2_{\mathcal{M}_b(Q_T)}=\mathcal{O}(1)$.
\end{proof}

\section{Probabilistic Young measures}\label{sec:YoungMeasures}

Let $\U_\eps$ be the solution to \eqref{stoEulereps} given in Theorem~\ref{th:existspatheps}. Our aim is to prove the convergence of $(\U_\eps)$. The standard tool for this is the notion of measure-valued solution introduced by Di Perna, \cite{Diperna83a}. In this section we give some precisions about it in our context of random solutions. More precisely, we know that, almost surely, $({\U_\eps})$ defines a Young measure $\nu_\eps$ on $\R_+\times\R$ by the formula
\begin{equation}
\<\nu^\eps_{x,t},\varphi\>:=\<\delta_{{\U_\eps}(x,t)},\varphi\>=\varphi({\U_\eps}(x,t)),\quad\forall\varphi\in C_b(\R_+\times\R).
\label{defnueps}\end{equation}
Our aim is to show that $\nu_\eps\rightharpoonup\nu$ (in a sense to be specified), where $\nu$ has some specific properties. To that purpose, we will use the probabilistic compensated compactness method developed in the Appendix of \cite{FengNualart08} and some results on the convergence of probabilistic Young measures that we introduce here. Note that the notion of random Young measure has also been introduced and developed by Brze{\'z}niak and Serrano in \cite{BrzezniakSerrano13}, compare in particular \cite[Lemma 2.18]{BrzezniakSerrano13} and Proposition~\ref{prop:compactYproba} below. Brze{\'z}niak and Serrano, \cite{BrzezniakSerrano13}, use the theory of Young measure developed by Castaing, Raynaud de Fitte, Valadier, \cite{CastaingRaynauddeFitteValadier04}.

\subsection{Young measures embedded in a space of Probability measures}\label{sec:YoungMeasuresProbas}

Let $(Q,\mathcal{A},\lambda)$ be a finite measure space. Without loss of generality, we will assume $\lambda(Q)=1$. A Young measure on $Q$ (with state space $E$) is a mea\-su\-ra\-ble map $Q\to\mathcal{P}_1(E)$, where $E$ is a topological space endowed with the $\sigma$-algebra of Borel sets, $\mathcal{P}_1(E)$ is the set of probability measures on $E$, itself endowed with the $\sigma$-algebra of Borel sets corresponding to the topology defined by the weak\footnote{actually, weak convergence {\it of probability measures}, also corresponding to the tight convergence of finite measures} convergence of measures, {\it i.e.} $\mu_n\to\mu$ in $\mathcal{P}_1(E)$ if 
\begin{equation*}
\<\mu_n,\varphi\>\to\<\mu,\varphi\>,\quad\forall\varphi\in C_b(E).
\end{equation*}
As in \eqref{defnueps}, any measurable map $w\colon Q\to E$ can be viewed as a Young measure $\nu$ defined by
\begin{equation*}
\<\nu_z,\varphi\>=\<\delta_{w(z)},\varphi\>=\varphi(w(z)),\quad\forall\varphi\in C_b(E),\quad\mbox{for }\lambda-\mbox{almost all } z\in Q.
\end{equation*}
A Young measure $\nu$ on $Q$ can itself be seen as a probability measure on $Q\times E$ defined by 
\begin{equation*}
\<\nu,\psi\>=\int_Q\int_E \psi(z,p) d\nu_z(p) d\lambda(z),\quad\forall\psi\in C_b(Q\times E).
\end{equation*}
We then have, for all $\psi\in C_b(Q)$ ($\psi$ independent on $p\in E$), $\<\nu,\psi\>=\<\lambda,\psi\>$, that is to say
\begin{equation}
\pi_*\nu=\lambda,
\end{equation}
where $\pi$ is the projection $Q\times E\to Q$ and the push-forward of $\nu$ by $\pi$ is defined by $\pi_*\nu(A)=\nu(\pi^{-1}(A))$, for all Borel subset $A$ of $Q$. Assume now that $Q$ is a compact subset of $\R^s$ and $E$ is a closed subset of $\R^m$, $m,s\in\N^*$, and, conversely, let $\mu$ is a probability measure on $Q\times E$ such that $\pi_*\mu=\lambda$. Then, by the Slicing Theorem (\textit{cf.} Attouch, Buttazzo, Michaille \cite[Theorem~4.2.4]{AttouchButtazzoMichaille06}), we have: for $\lambda$-a.e. $z\in Q$, there exists $\mu_z\in\mathcal{P}_1(E)$ such that,
\begin{equation*}
z\mapsto \<\mu_z,\varphi\>
\end{equation*}
is measurable from $Q$ to $\R$ for every $\varphi\in C_b(E)$, and 
\begin{equation*}
\<\mu,\psi\>=\int_Q\int_E \psi(z,p) d\mu_z(p) d\lambda(z),
\end{equation*}
for all $\psi\in C_b(Q\times E)$. This means precisely that $\mu$ is a Young measure on $Q$. We therefore denote by
\begin{equation*}
\mathcal{Y}=\left\{\nu\in\mathcal{P}_1(Q\times E);\pi_*\nu=\lambda\right\}
\end{equation*}
the set of Young measures on $Q$. 
\medskip

We use now the Prohorov's Theorem, \textit{cf.} Billingsley \cite[Theorem~5.1]{BillingsleyBook}, to give a compactness criteria in $\mathcal{Y}$. We assume that $Q$ is a compact subset of $\R^s$ and $E$ is a closed subset of $\R^m$. We also assume that the $\sigma$-algebra $\mathcal{A}$ of $Q$ is the $\sigma$-algebra of Borel sets of $Q$.  

\begin{proposition}[Bound against a Lyapunov functional] Let $\eta\in C(E;\R_+)$ satisfy the growth condition
\begin{equation*}
\lim_{p\in E,|p|\to+\infty} \eta(p)=+\infty.
\end{equation*} 
Let $C>0$ be a positive constant. Then the set 
\begin{equation}
K_C=\left\{\nu\in\mathcal{Y};\int_{Q\times E}\eta(p)d\nu(z,p)\leq C\right\}
\label{compactKC}\end{equation}
is a compact subset of $\mathcal{Y}$.
\label{prop:compactY}\end{proposition}

\begin{proof} The condition $\pi_*\nu=\lambda$ being stable by weak convergence, $\mathcal{Y}$ is closed in $\mathcal{P}_1(Q\times E)$. By Prohorov's Theorem,  \cite[Theorem~5.1]{BillingsleyBook}, $K_C$ is relatively compact in $\mathcal{Y}$ if, and only if it is tight.  Besides, $K_C$ is closed since
\begin{equation*}
\int_{Q\times E}\eta(p)d\nu(z,p)\leq\liminf_{n\to+\infty}\int_{Q\times E}\eta(p)d\nu_n(z,p)
\end{equation*}
if $(\nu_n)$ converges weakly to $\nu$. It is therefore sufficient to prove that $K_C$ is tight, which is classical: let $\eps>0$. For $R\geq 0$, let
\begin{equation*}
V(R)=\inf_{|p|>R}\eta(p).
\end{equation*}
Then $V(R)\to +\infty$ as $R\to+\infty$ by hypothesis and, setting $M_R=Q\times [\overline{B}(0,R)\cap E]$, we have 
\begin{equation*}
V(R)\nu(M_R^c)\leq \int_{Q\times E}\eta(p)d\nu(z,p)\leq C,
\end{equation*}
for all $\nu\in K_C$, whence $\sup_{\nu\in K_C}\nu(M_R^c)<\eps$ for $R$ large enough. 
\end{proof}

The following result is a generalization of  \cite[Proposition 2.4.1]{CastaingRaynauddeFitteValadier04} which considers  
$\eta(p)=d(p,p_0)^q$ with $d$ a distance (or semi-distance) $p_0$ a given basis point, and $q\geq 1$.

\begin{proposition}[Momentum-convergence] Let $(\nu_n)$ be a sequence of Young measures in $\mathcal{Y}$ satisfying the bound
\begin{equation}\label{BoundOrlicz}
C(\eta):=\sup_n \int_{Q\times E}\eta(p)^s d\nu_n(z,p)<+\infty,
\end{equation}
where $\eta$ is a non-negative continuous function on $E$ and $s>1$. Assume that $\nu_n\to\nu$ in $\mathcal{Y}$. Then
\begin{equation}\label{BoundOrliczLim}
\int_{Q\times E}\eta(p)^s d\nu(z,p)<+\infty,
\end{equation}
and
\begin{equation}\label{CVOrlicz}
\int_{Q\times E}\varphi(z)\eta(p)^r d\nu_n(p,z)\to \int_{Q\times E}\varphi(z)\eta(p)^r d\nu(p,z),
\end{equation}
for all $r\in[1,s)$ and all $\varphi\in L^{\frac{s}{s-r}}(Q,\lambda)$.
\label{prop:OrliczY}\end{proposition}

\begin{proof} Let $\chi\in C(\R_+)$ be a non-negative non-increasing function supported in $[0,2]$ with value $1$ on $[0,1]$ and define the truncature function 
$\chi_R(a)=\chi\left(\frac{a}{R}\right)$, $R>1$. By the monotone convergence theorem, 
$$
\int_{Q\times E}\eta(p)^s d\nu(z,p)=\lim_{R\to+\infty}\int_{Q\times E}\eta(p)^s \chi_R(\eta(p)) d\nu(z,p).
$$
Since $p\mapsto \eta(p)^s\chi_R(\eta(p))$ is continuous and bounded, we also have
$$
\int_{Q\times E}\eta(p)^s \chi_R(\eta(p)) d\nu(z,p)=\lim_{n\to+\infty}\int_{Q\times E}\eta(p)^s \chi_R(\eta(p)) d\nu_n(z,p)\leq C(\eta).
$$
In particular, 
\begin{equation}\label{BoundOrliczLim2}
\int_{Q\times E}\eta(p)^s d\nu(z,p)\leq C(\eta),
\end{equation}
which gives \eqref{BoundOrliczLim}. Similarly, we have $J_R(\nu_n,\varphi)\to J_R(\nu,\varphi)$, where
$$
J_R(\nu,\varphi)=\int_{Q\times E}\varphi(z)\eta(p)^r\chi_R(\eta(p)) d\nu(p,z),
$$
for a fixed $\varphi\in C_b(Q)$. Let us also set
$$
J(\nu,\varphi)=\int_{Q\times E}\varphi(z)\eta(p)^r d\nu(p,z).
$$
The bounds \eqref{BoundOrlicz} and \eqref{BoundOrliczLim2}, a kind of equi-integrability conditions, give the following estimate:
\begin{equation}\label{JRvsJ}
|J_R(\mu,\varphi)-J(\mu,\varphi)|\leq \|\varphi\|_{C_b(Q)}\frac{C(\eta)}{R^{s-r}},
\end{equation}
for $\mu\in\{\nu_n;n\geq 0\}\cup\{\nu\}$.
Indeed, we have 
$
1-\chi_R(a)\leq\mathbf{1}_{a>R}\leq\frac{a^{s-r}}{R^{s-r}},
$
and therefore
$$
|J_R(\mu,\varphi)-J(\mu,\varphi)|\leq\frac{1}{R^{s-r}}\int_{Q\times E}|\varphi(z)|\eta(p)^{s}d\mu(p,z)
\leq  \|\varphi\|_{C_b(Q)}\frac{C(\eta)}{R^{s-r}}.
$$
Consequently,
$$
|J(\nu,\varphi)-J(\nu_n,\varphi)|\leq 2\|\varphi\|_{C_b(Q)}\frac{C(\eta)}{R^{s-r}}+|J_R(\nu,\varphi)-J(\nu_n,\varphi)|,
$$
and thus $J(\nu_n,\varphi)\to J(\nu,\varphi)$. This gives \eqref{CVOrlicz} in the case where $\varphi$ is continuous and bounded on $Q$. Let $\mu\in\{\nu_n;n\geq 0\}\cup\{\nu\}$. Since $\varphi\mapsto J(\mu,\varphi)$ is linear and
$$
|J(\mu,\varphi)|\leq \left[\int_{Q\times E}|\varphi(z)|^{\frac{s}{s-r}}d\mu(p,z)\right]^\frac{s-r}{s}\left[\int_{Q\times E}|\eta(p)|^s d\mu(p,z)\right]^\frac{r}{s}
\leq C(\eta)^\frac{r}{s}\|\varphi\|_{L^{\frac{s}{s-r}}(Q,\lambda)},
$$
the general case follows.
\end{proof}

\subsection{A compactness criterion for probabilistic Young measures}\label{sec:YoungMeasuresProbasProbas}

As above, we assume that $Q$ is a compact subset of $\R^s$ and $E$ is a closed subset of $\R^m$. We endow $\mathcal{P}_1(Q\times E)$ (and thus $\mathcal{Y}$ also) with the Prohorov's metric $d$. Then $(\mathcal{P}_1(Q\times E),d)$ is a complete, separable metric space, weak convergence coincides with $d$-convergence, and a subset $A$ is relatively compact if, and only if it is tight, \cite[p.72]{BillingsleyBook}. 

\begin{definition} A random Young measure is a $\mathcal{Y}$-valued random variable.
\end{definition}

\begin{proposition} Let $\eta\in C(E;\R_+)$ satisfy the growth condition
\begin{equation*}
\lim_{p\in E,|p|\to+\infty} \eta(p)=+\infty.
\end{equation*} 
Let $M>0$ be a positive constant. If $(\nu_n)$ is a sequence of random Young measures on $Q$ satisfying the bound
\begin{equation*}
\E\int_{Q\times E}\eta(p)d\nu_n(z,p)\leq M,
\end{equation*}
then, up to a subsequence, $(\nu_n)$ is converging in law.
\label{prop:compactYproba}\end{proposition}

\begin{proof} Let $\mathcal{L}(\nu_n)\in\mathcal{P}_1(\mathcal{Y})$ denote the law of $\nu_n$. To prove that it is tight, we use the Prohorov's Theorem. Let $\eps>0$. For $C>0$, let $K_C$ be the compact set defined by \eqref{compactKC}. If $\nu$ is a random Young measure, then we have
\begin{equation*}
\P(\nu\notin K_C)=\P\left(1<\frac{1}{C}\int_{Q\times E}\eta(p)d\nu(z,p)\right)\leq\frac{1}{C}\E\int_{Q\times E}\eta(p)d\nu(z,p),
\end{equation*}
hence 
\begin{equation*}
\sup_{n\in\N}\mathcal{L}(\nu_n)(\mathcal{Y}\setminus K_C)=\sup_{n\in\N}\P(\nu_n\notin K_C)\leq\frac{M}{C}<\eps,
\end{equation*}
for $C$ large enough, which proves the result. 
\end{proof}

The following proposition is a version of Proposition~\ref{prop:OrliczY} for random Young measures. 

\begin{proposition}[Momentum-convergence, random Young measures] Let $(\nu_n)$ be a sequence of random Young measures satisfying the bound
\begin{equation}\label{RandomBoundOrlicz}
C(\eta):=\sup_n \E\int_{Q\times E}\eta(p)^s d\nu_n(z,p)<+\infty,
\end{equation}
where $\eta$ is a non-negative continuous function on $E$ and $s>1$. Assume that, almost-surely, $\nu_n\to\nu$ in $\mathcal{Y}$. Then
\begin{equation}\label{RandomBoundOrliczLim}
\E\int_{Q\times E}\eta(p)^s d\nu(z,p)<+\infty,
\end{equation}
and
\begin{equation}\label{RandomCVOrlicz}
\lim_{n\to+\infty}\E\left|\int_{Q\times E}\varphi(z)\eta(p)^r d\nu_n(p,z)- \int_{Q\times E}\varphi(z)\eta(p)^r d\nu(p,z)\right|=0,
\end{equation}
for all $r\in[1,s)$ and all $\varphi\in L^{\frac{s}{s-r}}(Q,\lambda)$.
\label{prop:OrliczYRandom}\end{proposition}

\begin{proof} As in the proof of Proposition~\ref{prop:OrliczY}, an argument of truncature and of monotone convergence gives 
\begin{equation}\label{RandomBoundOrliczLim2}
\E\int_{Q\times E}\eta(p)^s d\nu(z,p)\leq C(\eta),
\end{equation}
hence \eqref{RandomBoundOrliczLim}. Let $\varphi\in C_b(Q)$. Set
$$
X_n^R=\int_{Q\times E}\varphi(z)\eta(p)^r\chi_R(\eta(p)) d\nu_n(z,p),\quad X^R=\int_{Q\times E}\varphi(z)\eta(p)^r\chi_R(\eta(p)) d\nu(z,p),
$$
where $\chi_R$ is the truncature function introduced in the proof of Proposition~\ref{prop:OrliczY}. We know by hypothesis that $X_n^R\to X^R$ almost surely. We have also
\begin{equation}\label{RandomEquiIntE}
\E|X_n^R|^{s/r}\leq\|\varphi\|_{C_b(Q)}^{s/r}C(\eta),
\end{equation}
hence $(X_n^R)_n$ is equi-integrable. By the Vitali Theorem, we deduce that $X_n^R\to X^R$ in $L^1(\Omega)$. Then one can show, as in the proof of Proposition~\ref{prop:OrliczY}, that
$$
\E|X_n-X|\leq 2\|\varphi\|_{C_b(Q)}\frac{C(\eta)}{R^{s-r}}+\E|X_n^R-X^R|,
$$
where
$$
X_n=\int_{Q\times E}\varphi(z)\eta(p)^r d\nu_n(z,p),\quad X=\int_{Q\times E}\varphi(z)\eta(p)^r d\nu(z,p).
$$
The convergence \eqref{RandomCVOrlicz} then follows (again, we refer to the proof of Proposition~\ref{prop:OrliczY} for the details).
\end{proof}

\begin{remark}\label{RkRandomOrlizcY} Under the hypotheses of Proposition~\ref{prop:OrliczYRandom}, we have a result of convergence a little bit stronger than \eqref{RandomCVOrlicz}. Indeed, if $\varphi\in L^{\frac{s}{s-r}}(Q,\lambda)$, we have a bound on $X_n$ and $X$ in $L^{s/r}(\Omega)$ (\textit{cf.} \eqref{RandomEquiIntE}). Consequently
\begin{equation}\label{RandomCVOrliczgamma}
\lim_{n\to+\infty}\E\left|\int_{Q\times E}\varphi(z)\eta(p)^r d\nu_n(p,z)- \int_{Q\times E}\varphi(z)\eta(p)^r d\nu(p,z)\right|^\delta=0,
\end{equation}
for all $\delta\in\left[1,\frac{s}{r}\right)$
\end{remark}

We end this section with a result about random Young measure being almost surely Dirac masses.

\begin{definition}[Random Dirac mass] Let $r\geq 1$ and let $\nu$ be a random Young measure. We say that $\nu$ is an $L^r$-random Dirac mass if there exists $u\in L^r(\Omega\times Q;E)$ such that, almost-surely, $\nu=\delta_u\rtimes\lambda$, \textit{i.e.} (indicating by the superscript $\omega$ the dependence on $\omega$): for $\P$-almost all $\omega\in\Omega$, 
\begin{equation}\label{eqRandomDiracmass}
\int_{Q\times E}\varphi(p,z) d\nu_{z}^\omega(p)d\lambda(z)=\int_Q \varphi(u^\omega(z),z) d\lambda(z),
\end{equation}
for all $\varphi\in C_b(Q\times E)$.
\end{definition}

\begin{proposition} Let $r\geq 1$, let $\nu$ be a random Young measure on the probability space $(\Omega,\P)$ and let $\tilde\nu$ be a random Young measure on a probability space $(\tilde\Omega,\tilde\P)$ such that $\nu$ and $\tilde\nu$ have same laws. Then $\nu$ is an $L^r$-random Dirac mass if, and only if, $\tilde\nu$ is an $L^r$-random Dirac mass, \textit{i.e.} the fact that $\nu$ is an $L^r$-random Dirac mass depends on the distribution of $\nu$ uniquely.
\label{prop:LawDiracMass}\end{proposition}

\begin{proof} We denote by $\tilde\E$ the expectancy with respect to $\tilde\P$. Let $\psi\colon\R^m\to\R$ be a strictly convex function satisfying the growth condition $$
C_1 |p|^r\leq |\psi(p)|\leq C_2(1+|p|^r).
$$ 
If $\nu$ is an $L^r$-random Dirac mass, then 
\begin{equation}\label{psiDirac}
\E\int_{Q\times E}\psi(p)d\nu_{z}(p)d\lambda(z)=\E\int_Q\psi\left(\int_E p d\nu_{z}(p)\right) d\lambda(z),
\end{equation}
and both sides of the equation (equal to $\E\|\psi(u)\|_{L^1(Q)}$) are finite. Equation \eqref{psiDirac} can be rewritten 
\begin{equation}
\E\varphi(\nu)=\E\theta(\nu),
\label{varphitheta}\end{equation}
where the functions $\varphi$ and $\theta$ are defined on $\mathcal{Y}$ as the applications
$$
\varphi\colon\mu\mapsto\int_{Q\times E}\psi(p)d\mu_z(p)d\lambda(z),\quad\theta\colon\mu\mapsto\int_Q\psi\left(\int_E p d\mu_{z}(p)\right) d\lambda(z).
$$
The function $\varphi$ is continuous on $\mathcal{Y}$ and, by the Lebesgue dominated convergence theorem, $\theta$ is continuous on the subset 
$$
\mathcal{Y}_r:=\left\{\mu\in\mathcal{Y};\int_{Q\times E}|p|^r d\mu_z(p)d\lambda(z)<+\infty\right\}.
$$
If $\tilde\nu$ has same law as $\nu$, then \eqref{varphitheta} shows that $\tilde\P$-almost surely $\tilde\nu\in \mathcal{Y}_r$, that
\begin{equation}\label{psiDiracTilde}
\E\int_{Q\times E}\psi(p)d\tilde\nu_{z}(p)d\lambda(z)=\E\int_Q\psi\left(\int_E p d\tilde\nu_{z}(p)\right) d\lambda(z),
\end{equation}
and that both sides of the equation \eqref{psiDiracTilde} are finite. Note that, $\tilde\P$-almost surely, for $\lambda$-almost all $z\in Q$, 
\begin{equation}\label{Jensentildenu}
\int_E \psi(p)d\tilde\nu_{z}(p)\geq \psi\left(\int_E p d\tilde\nu_{z}(p)\right), 
\end{equation}
by the Jensen Inequality. By strict convexity of $\psi$, there is equality in \eqref{Jensentildenu} if, and only if, $\tilde\nu_z$ is the Dirac mass $\delta_{\tilde u(z)}$, where
\begin{equation}\label{nutoutilde}
\tilde u(z):=\int_E p d\tilde\nu_{z}(p).
\end{equation}
Therefore \eqref{Jensentildenu} shows that $\tilde\P$-almost surely, for $\lambda$-almost all $z\in Q$, $\nu_z=\delta_{\tilde u(z)}$. In particular, 
\eqref{eqRandomDiracmass} is satisfied by $\tilde\nu$ and $\tilde u$. By \eqref{nutoutilde}, $\tilde u$ is measurable from $\Omega\times Q$ to $E$. Since
$$
\E\int_Q \psi(u)d\lambda=\E\int_{Q\times E}\psi(p)d\tilde\nu_{z}(p)d\lambda(z)<+\infty
$$
in \eqref{psiDiracTilde}, we have $u\in L^r(\Omega\times Q;E)$.
\end{proof}

\subsection{Convergence to a random Young measure}\label{sec:cvYoungMeasures}

Let $\U_\eps$ be a bounded solution to \eqref{stoEulereps}. We will apply the results of paragraphs \ref{sec:YoungMeasuresProbas}-\ref{sec:YoungMeasuresProbasProbas} to the case $Q=Q_T$, $\lambda$ is the $2$-dimensional Lebesgue measure on $Q_T$, $E=\R_+\times\R$ and $\nu^\eps=\delta_{(\rho_\eps,u_\eps)}\rtimes\lambda$, that is to say 
\begin{equation}\label{defYoungeps}
\int_{Q_T\times\R_+\times\R}\varphi(x,t,\rho,u) d\nu^\eps_{x,t}(\rho,u) dx dt =\int_{Q_T} \varphi(x,t,\rho_\eps(x,t),u_\eps(x,t)) dx dt,
\end{equation}
for all $\varphi\in C_b(Q_T\times\R_+\times\R)$. Hence our aim is to show the convergence, in an appropriate sense, of $(\nu_\eps)$. We will also need the convergence of $(e^\eps)$, where $e^\eps$ is defined by \eqref{defeeps}. To that purpose, we introduce the path space 
$$
\mathcal{X}_e=\mathcal{M}_b^+(\overline{Q_T})
$$ 
Recall that $\mathcal{M}_b^+(\overline{Q_T})$ is the set of non-negative bounded Borel measures on $\overline{Q_T}$. Then $\mathcal{X}_e$ is a subset of $\mathcal{M}_b(\overline{Q_T})$, the set of bounded Borel measures on $\overline{Q_T}$, which is the dual to $C(\overline{Q_T})$. The topology that we consider on $\mathcal{X}_e$ is the weak-star topology induced by $C(\overline{Q_T})$. 
We also use the convergence of $(\U_\eps)$ in the path-space
$$
\mathcal{X}_\U:=C^\beta([0,T];W^{-3,2}(\T)).
$$
Here $\beta\in(0,1/4)$ is a given exponent and the negative Sobolev space $W^{-3,2}(\T)$ is the dual to $W^{3,2}(\T)$.

\begin{proposition} Let  ${\U_\eps}_0\in W^{2,2}(\T)$ satisfy ${\rho_\eps}_0\geq c_{\eps 0}$ a.e. in $\T$, for a positive constant $c_{\eps 0}$. Assume that hypotheses \eqref{A0eps}, \eqref{Trunceps}, \eqref{Lipsigmaeps} are satisfied, that ${\U_\eps}_0\in\Lambda_{\varkappa_\eps}$ and that
\begin{equation}\label{UniformInitialEnergy}
\E\int_\T \frac12\rho_{\eps 0} u^2_{\eps 0}+\frac{\kappa}{\gamma-1}\rho^\gamma_{\eps 0}\ dx
\end{equation}
is bounded uniformly with respect to $\eps$. Let $\U_\eps$ be the bounded solution to \eqref{stoEulereps} and let $\nu^\eps$ be the Random Young measure associated to $\U_\eps$ defined by \eqref{defYoungeps}. Let $(\eps_n)$ be a sequence of real numbers decreasing to zero and let $\mathcal{X}_W$ be the path space defined by \eqref{WpathSpace}. Then, up to a subsequence, there exists a probability space $(\tilde\Omega,\tilde{\mathcal{F}},\tilde\P)$, some random variables $(\tilde{\nu}^{\eps_n},\tilde W,\tilde e^{\eps_n},\tilde \U_{\eps_n})$ and $(\tilde\nu,\tilde W,\tilde e,\tilde \U)$ with values in $\mathcal{Y}\times\mathcal{X}_W\times\mathcal{X}_e\times\mathcal{X}_\U$ such that
\begin{enumerate} 
\item the law of $(\tilde{\nu}^{\eps_n},\tilde W,\tilde e^{\eps_n},\tilde \U_{\eps_n})$ under $\tilde\P$ coincide with the law of $(\nu^{\eps_n},W,e^{\eps_n},\U_{\eps_n})$,
\item $(\tilde{\nu}^{\eps_n},\tilde e^{\eps_n},\U_{\eps_n})$ converges $\tilde\P$-almost surely to $(\tilde\nu,\tilde e,\tilde\U)$ in the topology of $\mathcal{Y}\times\mathcal{X}_e\times\mathcal{X}_\U$.
\end{enumerate}
Furthermore, $\tilde\U\in C([0,T];W^{-2,2}(\T))$ $\tilde\P$-almost-surely.
\label{prop:cvrandomY}\end{proposition}

\begin{proof} let $\eta$ be the entropy (energy in that case) defined by \eqref{entropychi} with $g(\xi)=|\xi|^{2}$. Then $\eta$ is coercive by \eqref{estimetabelow}. For such an $\eta$, \eqref{UniformInitialEnergy} and the uniform estimate \eqref{estimmomenteps2} show with Proposition~\ref{prop:compactYproba} that the sequence of random Young measures $(\nu^{\eps_n})$ is tight. We deduce from \eqref{timeepsCor} (taking $\alpha\in(\beta,1/4)$) that $(\U_{\eps_n})$ is tight in $\mathcal{X}_\U$. The tightness of $(e^{\eps_n})$ in $\mathcal{X}_e$ follows from Proposition~\ref{prop:bounde}. Since the single random variable $W$ is tight on the Polish space $\mathcal{X}_W$, the $4$-uple $(\nu^{\eps_n},W,e^{\eps_n},\U_{\eps_n})$ is tight on $\mathcal{Y}\times\mathcal{X}_W\times\mathcal{X}_e\times\mathcal{X}_\U$. We can apply then the Skorohod Theorem \cite[p.~70]{BillingsleyBook} to conclude.
Since the closed balls of $C^\beta([0,T];W^{-2,2}(\T)$ are closed in $\mathcal{X}_\U$, it follows from \eqref{timeepsCor} that $\tilde\U\in C([0,T];W^{-2,2}(\T))$ $\tilde\P$-almost-surely.
\end{proof}

\begin{remark} We may take $\tilde{\Omega}=[0,1]$, with $\tilde\F$ the $\sigma$-algebra of the Borelians on $[0,1]$ and $\tilde\P$ the Lebesgue measure on $[0,1]$, see \cite{Skorohod56}.
\label{rk:Skorohod01}\end{remark}

\section{Reduction of the Young measure}\label{sec:reductionYoung}

Proposition~\ref{prop:cvrandomY} above gives the existence of a random young measure $\tilde\nu$ such that $\tilde\nu_\eps$ converges in law and almost surely in the sense of Young measures to $\tilde\nu$. We will now apply the compensated compactness method to prove that a.s., for a.e. $(x,t)\in Q_T$, either $\tilde\nu_{x,t}$ is a Dirac mass or $\tilde\nu_{x,t}$ is concentrated on the vacuum region $\{\rho=0\}$. To do this, we will use the probabilistic compensated compactness method of \cite{FengNualart08} to obtain a set of functional equations satisfied by $\tilde\nu$. Then we conclude by adapting the arguments of \cite{LionsPerthameSouganidis96}.

\subsection{Compensated compactness}\label{sec:CC}

Let $\mathcal{G}$ denote the set of functions $g\in C^2(\R)$, convex, with $g$ sub-quadratic and $g'$ sub-linear:
\begin{equation}\label{gsubquad}
|g(\xi)|\leq C(g)(1+|\xi|^2),\quad |g'(\xi)|\leq C(g)(1+|\xi|),
\end{equation}
for all $\xi\in\R$, for a given constant $C(g)\geq 0$.

\subsubsection{Preparation to Murat's Lemma}\label{sec:CCMurat}

For $p\in[1,+\infty]$, we denote by $W_0^{1,p}(Q_T)$ the set of functions $u$ in the Sobolev space $W^{1,p}(Q_T)$ such that $u=0$ on $\T\times\{0\}$ and $\T\times\{T\}$. We denote by $W^{-1,p}(Q_T)$ the dual of $W_0^{1,p'}(Q_T)$, where $p'$ is the conjugate exponent to $p$. First we prove the tightness of the sequence $(\eps\partial_{xx}^2 \eta({\U_\eps}))_{\eps>0}$.

\begin{proposition}[Case $\gamma\leq 2$] We assume $\gamma\leq 2$. Let  ${\U_\eps}_0\in W^{2,2}(\T)$ satisfy ${\rho_\eps}_0\geq c_{\eps 0}$ a.e. in $\T$, for a positive constant $c_{\eps 0}$. Assume that hypotheses \eqref{A0eps}, \eqref{Trunceps}, \eqref{Lipsigmaeps} are satisfied, that ${\U_\eps}_0\in\Lambda_{\varkappa_\eps}$ and that
\begin{equation}\label{UniformInitialEnergyless2}
\E\int_\T \frac12\rho_{\eps 0} u^2_{\eps 0}+\frac{\kappa}{\gamma-1}\rho^\gamma_{\eps 0}\ dx
\end{equation}
is bounded uniformly with respect to $\eps$. Let $\U_\eps$ be the bounded solution to \eqref{stoEulereps}. 
Let $r\in(1,2)$ and let $\eta$ be an entropy of the form \eqref{entropychi} with $g\in\mathcal{G}$ (\textit{cf.} \eqref{gsubquad}). 
Then the sequence of random variables $(\eps \partial_{xx}^2 \eta({\U_\eps}))_{\eps>0}$ is tight in $W^{-1,r}(Q_T)$.
\label{prop:Muratgammaless2}\end{proposition}

\begin{proof} We suppose first that $\gamma<2$ and we set 
$
m=\displaystyle\frac{r}{2-r}(2-\gamma).
$
We can assume that $r\in\left(\frac{2}{3-\gamma},2\right)$. Then $m>1$. We will show that $(\eps\partial_{xx}^2 \eta({\U_\eps}))$ converges to zero in probability on $W^{-1,r}(Q_T)$ by proving that 
\begin{equation}\label{eq:murat0lessP}
\ds\lim_{\eps\to 0}\eps \partial_x\eta({\U_\eps})=0\mbox{ in probability in }L^r(Q_T).
\end{equation}
To obtain \eqref{eq:murat0lessP}, it is sufficient to prove the convergence
\begin{equation}\label{eq:murat0less}
\ds\lim_{\eps\to 0}\eps \partial_x \eta({\U_\eps})=0\mbox{ in }L^r(Q_T),
\end{equation}
conditionally to the bounds
\begin{equation}\label{LmRless}
\|\rho_\eps\|_{L^m(Q_T)}^m\leq R,
\end{equation} 
and
\begin{equation}
\eps\iint_{Q_T}\Big\{\Big[{\rho}_\eps^\gamma+|{u}_\eps|^4\Big]\rho_{\eps}^{\gamma-2} |\partial_x{\rho}_\eps|^2
+\Big[{\rho}_\eps(1+{\rho}_\eps^{2\theta}+|{u}_\eps|^2)\Big] \rho_\eps|\partial_x {u}_\eps|^2\Big\}  dx dt\leq R,\label{gradRless}
\end{equation}
where $R>1$ is fixed. Indeed, by the estimates \eqref{estimmomenteps2}, \eqref{corestimgradientepsrho}, \eqref{corestimgradientepsu} and the Markov Inequality, the probabilities of the events \eqref{LmRless} and \eqref{gradRless} are arbitrary large for large $R$, uniformly with respect to $\eps$. The proof of \eqref{eq:murat0less} is similar to the analysis in \cite[pp.627-629]{LionsPerthameSouganidis96}, with the difference that we do not use $L^\infty$ estimates here.  We note first that, by \eqref{gsubquad}, we have
\begin{equation*}
|\partial_\rho\eta(\U)|\leq C\left(1+|u|^2+\rho^{2\theta}\right),
\end{equation*}
and
$$
|\partial_u\eta(\U)| \leq C\rho\left(1+|u|+\rho^\theta\right),
$$
for a given non-negative constant that we still denote by $C$. By the Young Inequality, we obtain the bounds
\begin{align}
|\partial_x \eta({\U_\eps})|^r
\leq& C\left\{\left(1+|{u}_\eps|^{2r}+|{\rho}_\eps|^{2r\theta}\right)|\partial_x{\rho}_\eps|^r+\left(1+|{u}_\eps|^r+|{\rho}_\eps|^{r\theta}\right){\rho}_\eps^r|\partial_x {u}_\eps|^r\right\}\nonumber\\
	\leq & C\left\{1+\left(1+|{u}_\eps|^{2r}\right)|\partial_x{\rho}_\eps|^r+|{\rho}_\eps|^{4\theta}|\partial_x{\rho}_\eps|^2+{\rho}_\eps\left[1+|{\rho}_\eps|^{2\theta}+|{u}_\eps|^2\right]\rho_\eps|\partial_x {u}_\eps|^2\right\},\label{muratlessr}
\end{align}
where $C$ denotes some constant possibly varying from places to places that depends only on $r$. By \eqref{LmRless}, \eqref{gradRless} therefore,
\begin{equation}\label{LmGradRless}
\eps^r\iint_{Q_T}|\partial_x \eta({\U_\eps})|^r\, dx dt\leq C_R\eps^{r-1}+C\eps^r\iint_{Q_T}\left(1+|{u}_\eps|^{2r}\right)|\partial_x{\rho}_\eps|^r\, dx dt,
\end{equation}
where the constant $C_R$ depends on $R$. Since $\gamma\leq 2$, we have furthermore
\begin{align*}
\left(1+|{u}_\eps|^{2r}\right)|\partial_x{\rho}_\eps|^r&={\rho}_\eps^{\frac{r}{2}(2-\gamma)}\left(1+|{u}_\eps|^{2r}\right){\rho}_\eps^{\frac{r}{2}(\gamma-2)}|\partial_x{\rho}_\eps|^r\\
&\leq C{\rho}_\eps^{m}+C\left(1+|{u}_\eps|^{4}\right){\rho}_\eps^{\gamma-2}|\partial_x{\rho}_\eps|^2.
\end{align*}
By \eqref{LmRless}, \eqref{gradRless} and \eqref{LmGradRless}, we conclude to
\begin{equation}\label{gammaleq2}
\eps^r\iint_{Q_T}|\partial_x \eta({\U_\eps})|^r\, dx dt\leq C_R\eps^{r-1},
\end{equation}
for all $\eps\in(0,1)$. This gives the convergence~\eqref{eq:murat0less}. If $\gamma=2$, the arguments used above remain valid, taking $r=2$.
\end{proof}

\begin{proposition}[Case $\gamma>2$] We assume $\gamma>2$. Let  ${\U_\eps}_0\in W^{2,2}(\T)$ satisfy ${\rho_\eps}_0\geq c_{\eps 0}$ a.e. in $\T$, for a positive constant $c_{\eps 0}$. Assume that hypotheses \eqref{A0eps}, \eqref{Trunceps}, \eqref{Lipsigmaeps} are satisfied, that ${\U_\eps}_0\in\Lambda_{\varkappa_\eps}$ and that
\begin{equation}\label{UniformInitialEnergyg2}
\E\int_\T \frac12\rho_{\eps 0} u^2_{\eps 0}+\frac{\kappa}{\gamma-1}\rho^\gamma_{\eps 0}\ dx
\end{equation}
is bounded uniformly with respect to $\eps$. Let $\U_\eps$ be the bounded solution to \eqref{stoEulereps}. Assume that there exists $m>4$ such that the sequence $(\eps^{\frac{1}{\gamma-2}}\|u^\eps\|_{L^m(Q_T)})$ is stochastically bounded: for all $\alpha>0$, there exists $M>0$ such that, for all $\eps\in(0,1)$,
\begin{equation}\label{uepsLmg2}
\P\left(\eps^{\frac{1}{\gamma-2}}\|u^\eps\|_{L^m(Q_T)}> M\right)<\alpha.
\end{equation}
Let $r\in(1,2)$ and let $\eta$ be an entropy of the form \eqref{entropychi} with $g\in\mathcal{G}$ (\textit{cf.} \eqref{gsubquad}). 
Then the sequence of random variables $(\eps \partial_{xx}^2 \eta({\U_\eps}))_{\eps>0}$ is tight in $W^{-1,r}(Q_T)$.
\label{prop:Murat0g2}\end{proposition}

\begin{proof} We begin as in the proof of Proposition~\ref{prop:Muratgammaless2}. Without loss of ge\-ne\-ra\-li\-ty, we assume $\displaystyle\frac{4r}{2-r}\geq m$. We will obtain \eqref{eq:murat0lessP} here by proving that, given $\eta>0$, 
\begin{equation}\label{eq:murat0g2}
\lim_{\eps\to 0}\P(A_{\eps,\eta})=0,\quad A_{\eps,\eta}:=\left\{\|\eps\partial_x \eta(\U_\eps)\|_{L^r(Q_T)}>\eta\right\}.
\end{equation}
For $R>1$, we consider the events \eqref{gradRless} and 
\begin{equation}\label{LmR}
\|{u}_\eps\|_{L^m(Q_T)}\leq R,\quad \|{u}_\eps\rho_\eps^\frac12\|_{L^2(Q_T)}\leq R.
\end{equation} 
By \eqref{estimmomenteps2},  \eqref{corestimgradientepsrho}, \eqref{corestimgradientepsu} and \eqref{uepsLmg2}, the probability of the event
\begin{equation}\label{eventBR}
B_{\eps,R}:=\left\{\eqref{gradRless}\;\&\;\eqref{LmR}\;\&\;\eps^{\frac{1}{\gamma-2}}\|u_\eps\|_{L^m(Q_T)}\leq M\right\}
\end{equation}
is arbitrarily close to $1$ for large $R$, uniformly with respect to $\eps$. To obtain \eqref{eq:murat0g2}, it is therefore sufficient to prove
\begin{equation}\label{eq:murat1g2}
\lim_{\eps\to 0}\|\eps\partial_x \eta(\U_\eps)\|_{L^r(Q_T)}=0\mbox{ a.e. on }B_{\eps,R},
\end{equation}
for every $R>1$. To get \eqref{eq:murat1g2}, we use the estimate \eqref{muratlessr}, which gives \eqref{LmGradRless}. The remaining term in the right-hand side of \eqref{LmGradRless} is estimated as follows: let $\delta>0$. First, we have $1\leq\delta^{r(2-\gamma)/2}\rho_\eps^{r(\gamma-2)/2}$ on the set $\{\rho_\eps\geq \delta\}$ and, by the H\"older Inequality and \eqref{gradRless},
\begin{align*}
&\eps^r\iint_{Q_T}\left(1+|{u}_\eps|^{2r}\right)|\partial_x{\rho}_\eps|^r\, dx dt\\
\leq\ & C\delta^{\frac{r(2-\gamma)}{2}} \left(\eps^2 \iint_{Q_T}\left(1+|{u}_\eps|^{4}\right)\rho_\eps^{\gamma-2}|\partial_x{\rho}_\eps|^2\, dx dt\right)^{r/2}+\eps^r\iint_{Q_T}\left(1+|{u}_\eps|^{2r}\right)\mathbf{1}_{\rho_\eps<\delta}|\partial_x{\rho}_\eps|^r\, dx dt\\
\leq\ & C_R\eps^{r/2}\delta^{\frac{r(2-\gamma)}{2}}
+\eps^r\iint_{Q_T}\left(1+|{u}_\eps|^{2r}\right)\mathbf{1}_{\rho_\eps<\delta}|\partial_x{\rho}_\eps|^r\, dx dt.
\end{align*}
To estimate the part corresponding to $\{{\rho}_\eps<\delta\}$, we first use the H\"older Inequality to obtain
\begin{align}
\eps^r\iint_{Q_T}\left(1+|{u}_\eps|^{2r}\right)|\partial_x{\rho}_\eps|^r\mathbf{1}_{{\rho}_\eps<\delta}
\leq\ &\eps^{r/2}\Big(\iint_{Q_T}(1+|{u}_\eps|^{2r})^{\frac{2}{2-r}}\Big)^{\frac{2-r}{2}}\Big(\eps\iint_{Q_T}|\partial_x{\rho}_\eps|^2\mathbf{1}_{{\rho}_\eps<\delta}\Big)^{\frac{r}{2}}\nonumber\\
\lesssim\ &\eps^{r/2}(1+\|u_\eps\|_{L^m(Q_T)})^{2r}\Big(\eps\iint_{Q_T}|\partial_x{\rho}_\eps|^2\mathbf{1}_{{\rho}_\eps<\delta}\Big)^{\frac{r}{2}}.\label{rhoinfdelta1}
\end{align}
Then, we multiply the first Equation of the system~\eqref{eq:stoEulereps}, \textit{i.e.} Equation
$$
\partial_t{\rho}_\eps+\partial_x({\rho}_\eps {u}_\eps)=\eps\partial_{xx}^2{\rho}_\eps,
$$ 
by $\min({\rho}_\eps,\delta)$, and then sum the result over $Q_T$. This gives, by \eqref{LmR} and for some constants varying from lines to lines
\begin{align*}
\eps\iint_{Q_T}|\partial_x{\rho}_\eps|^2\mathbf{1}_{{\rho}_\eps<\delta}&\leq C\delta+C\Big(\iint_{Q_T}{\rho}_\eps|{u}_\eps||\partial_x{\rho}_\eps|\mathbf{1}_{{\rho}_\eps<\delta}\Big)\\
&\leq C\delta+C\delta^{1/2}\Big[\iint_{Q_T}|u_\eps|^2\rho _\eps\Big]^{\frac12}\Big[\iint_{Q_T}|\partial_x{\rho}_\eps|^2\mathbf{1}_{{\rho}_\eps<\delta}\Big]^{\frac12}\\
&\leq C_R\left(\delta+\frac{\delta}{\eps}\right)+\frac{\eps}{2}\iint_{Q_T}|\partial_x{\rho}_\eps|^2\mathbf{1}_{{\rho}_\eps<\delta},
\end{align*}
from which we deduce
$$
\eps\iint_{Q_T}|\partial_x{\rho}_\eps|^2\mathbf{1}_{{\rho}_\eps<\delta}\leq C_R\left(\delta+\frac{\delta}{\eps}\right).
$$
Reporting this result in \eqref{LmGradRless} and \eqref{rhoinfdelta1}, we get
\begin{equation}\label{muratdeltaVSeps}
\eps^r\iint_{Q_T}|\partial_x \eta({\U_\eps})|^r\, dx dt
\leq C_R\big(\eps^{r-1}+\delta^{\frac{r}{2}(2-\gamma)}\eps^{r/2}+\delta^{r/2}(1+\|u_\eps\|_{L^m(Q_T)})^{2r}\big).
\end{equation}
We take $\delta=o(\eps^{\frac{1}{\gamma-2}})$. On the event $B_{\eps,R}$ (\textit{cf.} \eqref{eventBR}), \eqref{muratdeltaVSeps} reads then
$$
\eps^r\iint_{Q_T}|\partial_x \eta({\U_\eps})|^r\, dx dt=o(1).
$$
This concludes the proof of \eqref{eq:murat1g2} and of Proposition~\ref{prop:Murat0g2}.
\end{proof}

\begin{remark}[Growth of $\|u^\eps\|_{L^{4+}(Q_T)}$] Since $\Lambda_{\varkappa_\eps}$ is an invariant region for $\U_\eps$, a sufficient condition to \eqref{uepsLmg2} is that $\eps^{\frac{1}{\gamma-2}}\varkappa_\eps$ is bounded:
\begin{equation}\label{growthvarkappaeps}
\eps^{\frac{1}{\gamma-2}}\varkappa_\eps\lesssim 1.
\end{equation}
In that case we have even $\eps^{\frac{1}{\gamma-2}}\|u^\eps\|_{L^\infty(Q_T)}\lesssim 1$ almost surely.
\label{rk:growthuepsL4}\end{remark}

The next Proposition is similar to Lemma~4.20 in \cite{FengNualart08}.

\begin{proposition} Let  ${\U_\eps}_0\in W^{2,2}(\T)$ satisfy ${\rho_\eps}_0\geq c_{\eps 0}$ a.e. in $\T$, for a positive constant $c_{\eps 0}$. Let $p\in\N$ satisfy $p\geq 4+\frac{1}{2\theta}$. Assume that hypotheses \eqref{A0eps}, \eqref{Trunceps}, \eqref{Lipsigmaeps} are satisfied, that ${\U_\eps}_0\in\Lambda_{\varkappa_\eps}$ and that
\begin{equation}\label{UniformInitialEnergyMeps}
\E\int_\T \left(\eta_0(\U_{\eps 0})+\eta_{2p}(\U_{\eps 0}) \right)dx
\end{equation}
is bounded uniformly with respect to $\eps$ (recall that $\eta_m$ denotes the entropy associated by \eqref{entropychi} to the convex function $\xi\mapsto\xi^{2m}$). Let $\U_\eps$ be the bounded solution to \eqref{stoEulereps}. Let $\eta$ be an entropy of the form \eqref{entropychi} with $g\in\mathcal{G}$ (\textit{cf.} \eqref{gsubquad}). Let 
\begin{equation}\label{defMeps}
M^\eps(t)=\int_0^t \partial_q\eta({\U_\eps})(s)\Phi^\eps({\U_\eps})(s)dW(s).
\end{equation}
Then $\partial_t M^\eps$ is tight in $W^{-1,2}(Q_T)$.
\label{prop:Murat2}\end{proposition}

\begin{proof} The proof is in essential the proof of Lemma~4.19 in \cite{FengNualart08}. However, we will proceed slightly differently (instead of using Marchaud fractional derivative we work directly with fractional Sobolev spaces and an Aubin-Simon compactness lemma). We begin by giving some precisions on the sense of $\partial_t M^\eps$: this is the random element of $W^{-1,2}(Q_T)$ defined $\P$-almost surely by
$$
\<\partial_t M^\eps,z\>_{W^{-1,2}(Q_T),W^{1,2}_0(Q_T)}=-\<M^\eps,\partial_t z\>_{L^2(Q_T),L^2(Q_T)}.
$$
Let $0\leq s\leq t\leq T$. In what follows we denote by $C$ any constant, that may vary from line to line, which depends on the data only and is independent on $\eps$. By the Burkholder-Davis-Gundy Inequality, we have
\begin{equation*}
\E\|M^\eps(t)-M^\eps(s)\|_{L^4(\T)}^4\leq C\int_{\T}\E\left|\int_s^t |\partial_q\eta({\U_\eps})|^2|\GG^\eps({\U_\eps})|^2 d\sigma\right|^{2} dx,
\end{equation*}
and, using the H\"older Inequality,
\begin{equation*}
\E\|M^\eps(t)-M^\eps(s)\|_{L^4(\T)}^4\leq C|t-s|\int_s^t\E\int_{\T} \left[|\partial_q\eta({\U_\eps})|^2|\GG^\eps({\U_\eps})|^2\right]^{2}d\sigma dx.
\end{equation*}
By \eqref{gsubquad}, and \eqref{dqetaG} with $m=1$, we have 
$$
|\partial_q\eta(\U)|^2\GG^2(\U)\leq C(\eta_0(\U)+\eta_2(\U)).
$$ 
Taking the square of both sides, we obtain
\begin{equation}\label{MepsetaG}
\left[|\partial_q\eta(\U)|^2\GG^2(\U)\right]^2\leq C(\eta_0(\U)+\eta_p(\U))
\end{equation}
by Lemma~\ref{lemmaentropies2}. The uniform estimate \eqref{estimmomenteps2} and \eqref{UniformInitialEnergyMeps} give
\begin{equation}\label{deltaMts}
\E\|M^\eps(t)-M^\eps(s)\|_{L^4(\T)}^4
\leq C|t-s|^2,
\end{equation}
and, by integration with respect to $t$ and $s$,
\begin{equation}
\E\int_0^T\hskip -7pt\int_0^T\frac{\|M^\eps(t)-M^\eps(s)\|_{L^4(\T)}^4}{|t-s|^{1+2\nu}}dtds\leq C,
\label{Esobolevfrac}\end{equation}
as soon as $\nu<1/2$. The left-hand side in this inequality \eqref{Esobolevfrac} is the norm of $M^\eps$ in the space $L^4(\Omega;W^{\nu,4}(0,T;L^4(\T)))$. Since $L^4(\T)\hookrightarrow H^{-1}(\T)$, it follows that
\begin{equation*}
\E\|M^\eps\|_{W^{\nu,4}(0,T;H^{-1}(\T))}^4\leq C.
\end{equation*}
We use the continuous injection 
$$
W^{\nu,4}(0,T;H^{-1}(\T))\hookrightarrow C^{0,\mu}([0,T];H^{-1}(\T))
$$ 
for every $0<\mu<\nu-\frac{1}{4}$ to obtain
\begin{equation}
\E\|M^\eps\|_{C^{0,\mu}([0,T];H^{-1}(\T))}^4\leq C.
\label{Esobolevfrac2}\end{equation}
Besides, taking $s=0$ in \eqref{deltaMts} and summing with respect to $t\in(0,T)$ gives also
\begin{equation}
\E\|M^\eps\|_{L^4(Q_T)}^4\leq C.
\label{AubinSimonNice}\end{equation}
By the Aubin-Simon compactness Lemma, \cite{Simon87}, the set
$$
A_R:=\left\{M\in L^2(Q_T); \|M^\eps\|_{C^{0,\mu}([0,T];H^{-1}(\T))}\leq R,\;\|M\|_{L^4(Q_T)}\leq R\right\}
$$
is compact in $C([0,T];H^{-1}(\T))$, hence compact in $L^2(0,T;H^{-1}(\T))$. Consequently \eqref{Esobolevfrac2} and \eqref{AubinSimonNice} show that $(M^\eps)$ is tight as a $L^2(0,T;H^{-1}(\T))$-random variable, and we conclude that $(\partial_t M^\eps)$ is tight as a $W^{-1,2}(Q_T)$-random variable. 
\end{proof}

\subsubsection{Functional equation}\label{sec:CCcl}

\begin{proposition} Let  ${\U_\eps}_0\in W^{2,2}(\T)$ satisfy ${\rho_\eps}_0\geq c_{\eps 0}$ a.e. in $\T$, for a positive constant $c_{\eps 0}$. Let $p\in\N$ satisfy $p\geq 4+\frac{1}{2\theta}$. Assume that hypotheses \eqref{A0eps}, \eqref{Trunceps}, \eqref{Lipsigmaeps} are satisfied, that ${\U_\eps}_0\in\Lambda_{\varkappa_\eps}$ and that
\begin{equation}\label{UniformInitialEnergyMurat}
\E\int_\T \left(\eta_0(\U_{\eps 0})+\eta_{2p}(\U_{\eps 0}) \right)dx
\end{equation}
is bounded uniformly with respect to $\eps$. Let $\U_\eps$ be the bounded solution to \eqref{stoEulereps}. If $\gamma>2$, we suppose that \eqref{uepsLmg2} is satisfied. Let $(\eta,H)$ be an entropy-entropy flux of the form \eqref{entropychi}-\eqref{entropychiflux} with $g\in\mathcal{G}$ (\textit{cf.} \eqref{gsubquad}). Then the family
$$
\left\{\div_{t,x}(\eta({\U_\eps}),H({\U_\eps}));\eps\in(0,1)\right\}
$$ 
is tight in $W^{-1,2}(Q_T)$.
\label{prop:divcurl}\end{proposition}

\begin{proof} 
\textbf{Step 1.} By Proposition~\ref{prop:momentEntropy}, $\eta(\U_\eps)$ and $H(\U_\eps)$ are uniformly bounded in $L^s(\Omega;L^s(Q_T))$. As a consequence, $\div_{t,x}(\eta({\U_\eps}),H({\U_\eps}))$ is stochastically bounded in $W^{-1,s}(Q_T)$.\medskip

\textbf{Step 2.}  We consider the entropy balance equation~\eqref{Itoentropyeps}, which we rewrite as the following distributional equation on $Q_T$:
\begin{equation*}
\div_{t,x}(\eta({\U_\eps}),H({\U_\eps}))
=-\eps \eta''({\U_\eps})\cdot({\partial_x\U_\eps},{\partial_x\U_\eps})+\eps\partial^2_{xx} \eta({\U_\eps})+\partial_t M^\eps+\frac{1}{2}\GG^{\eps}({\U_\eps})^2\partial^2_{qq} {\eta}({\U_\eps}),
\end{equation*}
where $M^\eps$ is defined by \eqref{defMeps}. Let $r\in(1,2)$. By Proposition~\ref{prop:Muratgammaless2}, 
Proposition~\ref{prop:Murat0g2} and Proposition~\ref{prop:Murat2}, the families 
$\{\eps\partial^2_{xx} \eta({\U_\eps})\}_{\eps\in(0,1)}$ and $\{\partial_t M^\eps\}_{\eps\in(0,1)}$ are tight 
in $W^{-1,r}(Q_T)$ and $W^{-1,2}(Q_T)$ respectively.
The two remaining terms
$$
\eps\eta''({\U_\eps})\cdot({\partial_x\U_\eps},{\partial_x\U_\eps})\quad\mbox{and}\quad \frac12|\GG({\U_\eps})|^2\partial^2_{qq}\eta({\U_\eps})
$$
are stochastically bounded in measure on $Q_T$ by \eqref{corestimgradientepsrho}-\eqref{corestimgradientepsu} and \eqref{A0}-\eqref{estimmomenteps2} respectively (we use \eqref{dqqetaG} with $m=1$ to estimate this latter term).
\medskip

\textbf{Step 3.}  We want now to apply the stochastic version of the Murat's Lemma, Lemma~A.3 in \cite{FengNualart08}. 
If we refer strictly to the statement of Lemma~A.3 in \cite{FengNualart08}, there is an obstacle here, due 
to the fact that $\eps \partial_{xx}^2 \eta({\U_\eps})$ is neither tight in $W^{-1,2}(Q_T)$, neither stochastically 
bounded in measure on $Q_T$. However, in the proof of Lemma~A.3 in \cite{FengNualart08}, the property which is used regarding the term that is stochastically bounded in measure on $Q_T$ is only the fact that it is tight in $W^{-1,r}(Q_T)$ for $1<r<2$, due to the compact injection $W^{1,\sigma}_0(Q_T)\hookrightarrow C(\overline{Q_T})$ for $\sigma>2$. The argument of interpolation theory which combines this compactness result with the stochastic bound in $W^{-1,r}(Q_T)$ can therefore be directly applied here: we deduce that the sequence of $W^{-1,2}(Q_T)$ random variables 
$$
\div_{t,x}(\eta({\U_\eps}),H({\U_\eps}))=\partial_t\eta({\U_\eps})+\partial_x H({\U_\eps})
$$ 
is tight. 
\end{proof}

We apply now the div-curl lemma to obtain the functional equation \eqref{resultCC} below.

\begin{proposition}[Functional Equation] Let  ${\U_\eps}_0\in W^{2,2}(\T)$ satisfy ${\rho_\eps}_0\geq c_{\eps 0}$ a.e. in $\T$, for a positive constant $c_{\eps 0}$. Let $p\in\N$ satisfy $p\geq 4+\frac{1}{2\theta}$. Assume that hypotheses \eqref{A0eps}, \eqref{Trunceps}, \eqref{Lipsigmaeps} are satisfied, that ${\U_\eps}_0\in\Lambda_{\varkappa_\eps}$ and that
\begin{equation}\label{UniformInitialEnergyDivCurl}
\E\int_\T \left(\eta_0(\U_{\eps 0})+\eta_{2p}(\U_{\eps 0}) \right)dx
\end{equation}
is bounded uniformly with respect to $\eps$. Let $\U_\eps$ be the bounded solution to \eqref{stoEulereps}. 
If $\gamma>2$, we furthermore suppose that the possible growth of $\varkappa_\eps$ with $\eps$ is limited according to \eqref{growthvarkappaeps}. Let $(\eta,H)$, $(\hat\eta,\hat H)$ be some entropy-entropy flux pairs of the form \eqref{entropychi}-\eqref{entropychiflux} associated to some convex functions $g,\hat g\in\mathcal{G}$ respectively (\textit{cf.} \eqref{gsubquad}). Let $\tilde\nu$ be the random Young measure given by Proposition~\ref{prop:cvrandomY}. Then, almost surely, for a.e. $(x,t)\in Q_T$, 
\begin{equation}\label{resultCC}
\<\hat\eta,\tilde\nu_{x,t}\>\<H,\tilde\nu_{x,t}\>-\<\eta,\tilde\nu_{x,t}\>\<\hat H,\tilde\nu_{x,t}\>=\<\hat\eta H-\eta\hat H,\tilde\nu_{x,t}\>.
\end{equation}
Besides, if \eqref{resultCC} is realized, then, for all $v,v'\in\R$, 
\begin{equation}
2\lambda\Big( \langle\chi(v) u \rangle \langle\chi(v')\rangle -\langle\chi(v)\rangle
\langle\chi(v') u \rangle \Big) 
=(v-v') \Big( \langle\chi(v) \chi(v')\rangle -\langle\chi(v)\rangle\langle\chi(v') \rangle \Big),  \label{chiinter}
\end{equation} 
where $\chi(\U,v)=(v-z)_+^\lambda(w-v)_+^\lambda$, $z:=u-\rho^\theta$, $w:=u+\rho^\theta$, and
$$
\langle\chi(v)\rangle=\int \chi(\U,v) \, d\tilde\nu_{x,t}(\U).
$$ 
\label{prop:divcurlapplied}\end{proposition}

\begin{proof} Let $(\eps_n)$ be the sequence considered in Proposition~\ref{prop:cvrandomY} (to be exact, this is a subsequence of $(\eps_n)$ that we are considering). By Proposition~\ref{prop:LawDiracMass}, $\tilde\nu_{x,t}^{\eps_n}$ is an $L^r$-random Dirac mass for every $n$. In particular, it satisfies almost surely, for a.e. $(x,t)\in Q_T$, the identity
\begin{equation}\label{resultCCn}
\<\hat\eta,\tilde\nu_{x,t}^{\eps_n}\>\<H,\tilde\nu_{x,t}^{\eps_n}\>-\<\eta,\tilde\nu_{x,t}^{\eps_n}\>\<\hat H,\tilde\nu_{x,t}^{\eps_n}\>=\<\hat\eta H-\eta\hat H,\tilde\nu_{x,t}^{\eps_n}\>.
\end{equation}
Let
$$
X_n(x,t)=(\<\eta,\tilde\nu_{x,t}^{\eps_n}\>,\<H,\tilde\nu_{x,t}^{\eps_n}\>),\quad \hat X_n(x,t)=(\<\hat\eta,\tilde\nu_{x,t}^{\eps_n}\>,\<\hat H,\tilde\nu_{x,t}^{\eps_n}\>).
$$
By Proposition~\ref{prop:momentEntropy} and Proposition~\ref{prop:OrliczYRandom}, we have, for all $\varphi\in L^2(Q_T;\R^2)$,
$$
\lim_{n\to+\infty}\E\left|\<X_n-X,\varphi\>_{L^2(Q_T;\R^2)}\right|=0,
$$
where $X(x,t)=(\<\eta,\tilde\nu_{x,t}\>,\<H,\tilde\nu_{x,t}\>)$. Up to a subsequence, the event \{$\<X_n-X,\varphi\>_{L^2(Q_T;\R^2)}\to 0$ as $n\to+\infty$\} is of probability one. The subsequence may depend on $\varphi$ at this stage. Since $L^2(Q_T;\R^2)$ is separable, however, there exists a subsequence of $(\eps_n)$ such that, for all $\varphi\in L^2(Q_T;\R^2)$, $\<X_n-X,\varphi\>_{L^2(Q_T;\R^2)}\to 0$ as $n\to+\infty$ a.s., that is to say $X_n\to X$ in weak-$L^2(Q_T)$ almost surely. Similarly (by taking a further subsequence of $(\eps_n)$ if necessary), we obtain $\hat X_n\to \hat X$ in weak-$L^2(Q_T)$ almost surely, where $\hat X(x,t)=(\<\hat\eta,\tilde\nu_{x,t}\>,\<\hat H,\tilde\nu_{x,t}\>)$. Let 
$$
\hat X_n^\bot=(-\<\hat H,\tilde\nu_{x,t}^{\eps_n}\>,\<\hat\eta,\tilde\nu_{x,t}^{\eps_n}\>)
$$
and let $\alpha>0$. Note that 
$$ 
\curl_{t,x}\hat X_n^\bot=\div_{t,x}\hat X_n.
$$
By Proposition~\ref{prop:divcurl} (we use Remark~\ref{rk:growthuepsL4} to ensure that \eqref{uepsLmg2} is satisfied if $\gamma>2$), there exists a compact subset $K_\alpha$ of $W^{-1,2}(Q_T)$ such that the event
\begin{equation}\label{divcurlcompact}
\div_{t,x}X_n\in K_\alpha\;\&\; \curl_{t,x}\hat X_n^\bot\in K_\alpha
\end{equation}
has probability greater than $1-\alpha$. If \eqref{divcurlcompact} is realized, then the div-curl lemma, \cite{Murat78} ensures that the product $X_n\cdot\hat X_n^\bot$ is converging in weak-$L^1(Q_T)$ to the product $X\cdot\hat X^\bot$. The product $X_n\cdot\hat X_n^\bot$ is the left-hand side of \eqref{resultCCn}. Therefore, we can pass to the limit in \eqref{resultCCn} to obtain \eqref{resultCC} almost surely conditionally to \eqref{divcurlcompact}, for a.e. $(x,t)\in Q_T$, that is to say for almost all $(\omega,x,t)\in A_\alpha$ with $\tilde\P\times\mathcal{L}^2(A_\alpha)\geq (1-\alpha)\mathcal{L}^2(Q_T)$ (we denote by $\mathcal{L}^2$ the Lebesgue measure on $Q_T$). We consider a sequence $(\alpha_n)$ converging to $0$. We can choose the sets $K_{\alpha_n}$ as an increasing sequence, in which case $(A_{\alpha_n})$ is also increasing. We set
$$
A=\bigcup_{n\in\N}A_{\alpha_n}.
$$
Then $A$ is of full $\tilde\P\times\mathcal{L}^2$-measure and \eqref{resultCC} is satisfied on $A$. The identity \eqref{chiinter} follows from the formulas \eqref{entropychi}, \eqref{entropychiflux} and \eqref{resultCC}.
\end{proof}

\subsection{Reduction of the Young measure}\label{sec:reduction}
We now follow \cite{LionsPerthameSouganidis96} to conclude. We switch from the variables $(\rho,u)$ or $(\rho,q)$ to $(w,z)$, where $w$ and $z$ are the Riemann invariants
$$
z=u-\rho^\theta,\quad w=u+\rho^\theta.
$$
We write then $\chi(w,z,v)$ for $\chi(\U,v)$. Let us fix $(\omega,x,t)$ such that \eqref{chiinter} is satisfied. Set 
$$
\mathcal{C}=\{v \in \R \, ; \, \langle\chi(v)\rangle >0\}=\bigcup_{(w,z) \in \textrm{supp} \tilde\nu_{x,t}} \{ v \, ; \, z < v < w\}.
$$
Let 
$$
V=\{(\rho,u)\in\R_+\times\R|\rho=0\}=\{(w,z)\in\R^2| w=z\}
$$
denote the vacuum region. If $\mathcal{C}$ is empty, then $\tilde\nu_{x,t}$ is concentrated on $V$. Assume $\mathcal{C}$ not empty. By Lemma I.2 in \cite{LionsPerthameSouganidis96} then, $\mathcal{C}$ is an open interval in $\R$, say $\mathcal{C}=]a,b[$, where $-\infty\leq a<b\leq+\infty$ (we use here the french notation for open intervals to avoid the confusion with the point $(a,b)$ of $\R^2$). Furthermore all the computations of \cite{LionsPerthameSouganidis96} apply here, and thus, as in Section I.6 of \cite{LionsPerthameSouganidis96}, we obtain 
\begin{equation}\label{endLPS96}
\langle\rho^{2\lambda\theta} \langle\chi\circ\pi_i\rangle\phi\circ\pi_i  \rangle=0,
\end{equation}
for any continuous function $\phi$ with compact support in $\mathcal{C}$, where $\pi_i\colon\R^2\to\R$ denote the projection on the first coordinate $w$ if $i=1$, and the projection on the second coordinate $z$ if $i=2$. \smallskip

Note that, if $\mathrm{supp}(\tilde\nu_{x,t})\setminus V$ is reduced to a single point $\{Q\}$, then $\pi_i(Q)\in \overline{\mathcal{C}}\setminus\mathcal{C}$ for $i=1$ and $i=2$. Assume by contradiction that there exists $Q\in\R^2$ satisfying 
\begin{equation}\label{hypQ}
Q\in\mathrm{supp}(\tilde\nu_{x,t})\setminus V,\quad \pi_i(Q)\in\mathcal{C},
\end{equation}
for a $i$ in $\{1,2\}$. Then there exists a neighbourhood $K$ of $Q$ such that $K\cap V=\emptyset$, $\nu_{x,t}(K)>0$, $\pi_i(K)\subset\mathcal{C}$. But then $\<\chi\circ\pi_i\>>0$ on $K$, $\rho>0$ on $K$ and, choosing a continuous function $\phi$ compactly supported in $\mathcal{C}$ such that $\phi>0$ on $K$ we obtain a contradiction to \eqref{endLPS96}. Consequently \eqref{hypQ} cannot be satisfied. This shows that there cannot exists {\it two distinct} points $P,Q$ in $\mathrm{supp}(\tilde\nu_{x,t})\setminus V$. Indeed, if two such points exists, then either $\pi_1(Q)<\pi_1(P)$, and then $Q$ satisfies \eqref{hypQ} with $i=1$, or $\pi_1(Q)=\pi_1(P)$ and, say, $\pi_2(P)<\pi_2(Q)$ and then $Q$ also satisfies \eqref{hypQ}. The other cases are similar by symmetry of $P$ and $Q$.

Therefore if  $\mathcal{C}  \neq \emptyset$, then the support of the restriction of $\tilde\nu_{x,t}$ to $\mathcal{C}$ is reduced to a point.
In particular, $a$ and $b$ are finite. Then, by Lemma I.2 in \cite{LionsPerthameSouganidis96}, $P:=(a,b)\in\mathrm{supp}(\nu_{x,t})$ 
and $\tilde\nu_{x,t}=\tilde\mu_{x,t}+\alpha\delta_{(\tilde\rho(x,t),\tilde u(x,t))}$, where $\tilde\mu_{x,t}=\tilde\nu_{x,t}|_V$. Using \eqref{chiinter}, we obtain
$$
0=(v-v')\chi(b,a,v)\chi(b,a,v')(\alpha-\alpha^2),
$$
for all $v,v'\in(a,b)$, and thus $\alpha=0$ or $1$. We have therefore proved the following result.

\begin{proposition}[Reduction of the Young measure] Let  ${\U_\eps}_0\in W^{2,2}(\T)$ satisfy ${\rho_\eps}_0\geq c_{\eps 0}$ a.e. in $\T$, 
for a positive constant $c_{\eps 0}$. Let $p\in\N$ satisfy $p\geq 4+\frac{1}{2\theta}$. 
Assume that hypotheses \eqref{A0eps}, \eqref{Trunceps}, \eqref{Lipsigmaeps} are satisfied, that ${\U_\eps}_0\in\Lambda_{\varkappa_\eps}$ and that
$$
\E\int_\T \left(\eta_0(\U_{\eps 0})+\eta_{2p}(\U_{\eps 0}) \right)dx
$$
is bounded uniformly with respect to $\eps$. Let $\U_\eps$ be the bounded solution to \eqref{stoEulereps}. 
If $\gamma>2$, we furthermore suppose that the possible growth of $\varkappa_\eps$ with $\eps$ is limited according to \eqref{growthvarkappaeps}. Let $\tilde\nu$ be the random Young measure given by Proposition~\ref{prop:cvrandomY}. Then, almost surely, for a.e. $(x,t)\in Q_T$, either $\tilde\nu_{x,t}$ is concentrated on the vacuum region $V$, or $\tilde\nu_{x,t}$ is reduced to a Dirac mass $\delta_{(\tilde\rho(x,t),\tilde u(x,t))}$.
\label{prop:rednu}\end{proposition}

\begin{remark} The notation $\tilde\rho$ is already used in Proposition~\ref{prop:cvrandomY} since $\tilde U$ denotes the limit of $\tilde\U_{\eps_n}$ in $\mathcal{X}_\U=C^\beta([0,T];W^{-3,2}(\T)$. Our notation is consistent however. Indeed, let us fix $\omega\in\tilde\Omega$ such that the convergence of $\tilde\nu^{\eps_n}$ and $\tilde\U_{\eps_n}$ are satisfied. Let us set
\begin{equation}\label{defUstar}
\U_*(x,t)=\int_E \begin{pmatrix}\eta(p)\\ H(p)\end{pmatrix} d\tilde\nu_{(x,t)}(p),\quad \begin{pmatrix}\eta(p)\\ H(p)\end{pmatrix}=\begin{pmatrix}\rho\\ q\end{pmatrix},
\end{equation}
where $(\eta(p),H(p))=(\rho,q)$ is the entropy-entropy flux pair obtained when taking $g(\xi)=1$ in \eqref{entropychi}-\eqref{entropychiflux}. Let $\varphi\in C^3(\overline{Q_T};\R^2)$. By passing to the limit in the sense of Young measure, respectively in the sense of the convergence in $\mathcal{X}_\U$ in the term
$$
\iint_{Q_T}\U_{\eps_n}(x,t)\cdot\varphi(x,t) dx dt,
$$
we obtain 
$$
\iint_{Q_T}\U_*(x,t)\cdot\varphi(x,t) dx dt=\iint_{Q_T}\tilde\U(x,t)\cdot\varphi(x,t) dx dt.
$$
Consequently, $\U_*=\tilde\U$ a.e. since $\varphi$ is arbitrary.
\label{tildeUUstar}\end{remark}

\subsection{Martingale solution}\label{sec:martsolfinal}

In this section we will prove Theorem~\ref{th:martingalesol}.


\subsubsection{Convergence of non-linear functionals of $(\U_\eps)$}\label{subsec:convergence}

Let $E=\R_+\times\R$. By Proposition~\ref{prop:rednu}, we have: almost surely, for every continuous and \textit{bounded} function $S$ on $E$ and every $\varphi\in C_b(Q_T)$,
\begin{equation}\label{eq:cvSeps}
\iint_{Q_T} S(\tilde{\U}_{\eps_n}(x,t)) \varphi(x,t) dxdt\to \iint_{Q_T}\int_E S(p)\varphi(x,t)d\tilde\nu_{x,t}(p)  dxdt,
\end{equation}
and we know that 
$$
\mathrm{supp}(\nu_{x,t})\cap V=\emptyset\Longrightarrow\int_E S(p)d\tilde\nu_{x,t}(p)=S(\tilde\rho(x,t),\tilde u(x,t)).
$$

\begin{proposition}[Limit in the vacuum] Let $g\in\mathcal{G}$ (\textit{cf.} \eqref{gsubquad}) and let $(\eta,H)$ be the entropy-entropy flux pair defined by \eqref{entropychi}-\eqref{entropychiflux}. Under the hypotheses of Proposition~\ref{prop:rednu}, the convergence \eqref{eq:cvSeps} holds true, in probability, for every $\varphi\in L^\infty(Q_T)$ and $S\in\{\eta,H\}$. Besides, the limit is trivial in the vacuum region: almost surely, for a.e. $(x,t)\in Q_T$, for $S\in\{\eta,H\}$,
\begin{equation}\label{limVacuum}
\mathrm{supp}(\tilde\nu_{x,t})\subset V\Longrightarrow\int_E S(p)d\tilde\nu_{x,t}(p)=0.
\end{equation}
\label{prop:nlinUeps}\end{proposition}

\begin{proof} By Proposition~\ref{prop:momentEntropy} and Proposition~\ref{prop:OrliczYRandom}, the convergence \eqref{eq:cvSeps} holds true, in $L^1(\tilde\Omega)$ for all $\varphi\in L^{\frac{s}{s-1}}(Q_T)$. In particular, it is satisfied for every $\varphi\in L^\infty(Q_T)$, with a convergence in probability.\medskip

To prove \eqref{limVacuum}, we use the two last estimates in Lemma~\ref{lemmaentropies2} with $m=1$ and $s=1$. This gives the equi-integrability estimates
\begin{equation*}\label{equi1epsu}
\E\iint_{Q_T} \left(|\eta(\U_\eps)|+|H(\U_\eps)|\right)|u_\eps| dx dt\leq C.
\end{equation*}
where $C$ is a constant independent on $\eps$. 
By Proposition~\ref{prop:OrliczYRandom} and \eqref{RandomBoundOrliczLim} follows
\begin{equation*}\label{equi1limu}
\tilde\E\iint_{Q_T}\int_E \left(|\eta(p)|+|H(p)|\right)|u|\, d\tilde\nu_{x,t} dxdt\leq C.
\end{equation*}
In particular, for $\tilde\P$-almost all $\omega$, there exists $C(\omega)<+\infty$ such that
$$
\iint_{Q_T}\int_E \left(|\eta(p)|+|H(p)|\right)|u|\, d\tilde\nu^\omega_{x,t} dxdt\leq C(\omega).
$$
Consequently, by making the distinction between the ranges $\{|u|>R\}$ and $\{|u|\leq R\}$, we see that
$$
\iint_{Q_T}\int_E \left(|\eta(p)|+|H(p)|\right) d\tilde\nu^\omega_{x,t} dxdt\leq \iint_{Q_T}\int_E \left(|\eta(p)|+|H(p)|\right)\mathbf{1}_{|u|\leq R}\, d\tilde\nu^\omega_{x,t} dxdt
+\frac{C(\omega)}{R}.
$$
If $\mathrm{supp}(\tilde\nu_{x,t}^\omega)\subset V$ then
$$
\iint_{Q_T}\int_E \left(|\eta(p)|+|H(p)|\right)\mathbf{1}_{|u|\leq R}\, d\tilde\nu^\omega_{x,t} dxdt=0
$$
and
$$
\iint_{Q_T}\int_E \left(|\eta(p)|+|H(p)|\right) d\tilde\nu^\omega_{x,t} dxdt\leq 
\frac{C(\omega)}{R}.
$$
Letting $R\to+\infty$ in this last estimate, we obtain \eqref{limVacuum}.
\end{proof}

\begin{remark} In the case where a priori $L^\infty$ bounds on $(\rho_\eps,u_\eps)$ are known, Proposition~\ref{prop:nlinUeps} is almost automatic. In the absence of such $L^\infty$ bounds it requires some additional estimates to be established. In our context, we have some estimates on moments of arbitrary orders (see \eqref{estimmomenteps2}). In some situations, like the isentropic Euler system with geometric effects, it is quite difficult to obtain enough equi-integrability to conclude. See in particular \cite{LeFlochWestdickenberg07} where such estimates are proved for the isentropic Euler system with geometric effects.
\end{remark}

Recall the definition~\eqref{defUstar} of $\U_*=\begin{pmatrix}\rho_*\\  q_*\end{pmatrix}$. Outside the vacuum, $\tilde\nu_{x,t}$ is the Dirac mass 
$$
\tilde\nu_{x,t}=\delta_{(\rho_*(x,t),u_*(x,t))},\quad u_*(x,t):=\frac{q_*(x,t)}{\rho_*(x,t)}.
$$ 
By Proposition~\ref{prop:nlinUeps}, this is still true in the vacuum, by setting \begin{equation}\label{0Vacuum}
u_*(x,t)=0.
\end{equation} 
In particular, we have then 
$$
\int_E S(p) d\tilde\nu_{(x,t)}(p)=S(\U_*(x,t)),
$$
for almost all $(\omega,x,t)\in\Omega\times Q_T$ if $S=\eta$ or $S=H$, where $(\eta,H)$ is associated to a subquadratic function $g$ (\textit{i.e.} $g\in\mathcal{G}$, as defined in \eqref{gsubquad}).
We have proved in Remark~\ref{tildeUUstar} that $\tilde\U=\U_\star$ a.e. Consequently, we have
\begin{equation}\label{Diracrac}
\int_E S(p) d\tilde\nu_{(x,t)}(p)=S(\tilde\U(x,t)),
\end{equation}
for almost all $(\omega,x,t)\in\Omega\times Q_T$.\medskip 

We can prove now the following strong convergence result.

\begin{proposition}[Strong convergence] Let $g\in\mathcal{G}$ (\textit{cf.} \eqref{gsubquad}) and let $(\eta,H)$ be the entropy-entropy flux pair defined by \eqref{entropychi}-\eqref{entropychiflux}. Under the hypotheses of Proposition~\ref{prop:rednu}, we have, up to subsequence, 
\begin{equation}\label{strongUepsU}
\eta(\tilde\U_{\eps_n})\to\eta(\tilde\U),\quad H(\tilde\U_{\eps_n})\to H(\tilde\U)
\end{equation}
in $L^2(\tilde\Omega\times Q_T)$-strong and a.s. as $L^2(Q_T)$-valued random variables.
\label{prop:nlinUepsStrong}\end{proposition}

\begin{proof} Let $S\in\{\eta,H\}$. Let us set 
$$
\Psi_n(x,t)=S(\tilde{\U}_{\eps_n}(x,t)),\quad \Psi(x,t)=\int_E S(p)\varphi(x,t)d\tilde\nu_{x,t}(p).
$$ 
By \eqref{Diracrac}, we have $\Psi(x,t)=S(\tilde U(x,t))$. Then we use Proposition~\ref{prop:momentEntropy}, Proposition~\ref{prop:OrliczYRandom}, in particular \eqref{RandomCVOrliczgamma} with $r=1$, $\delta=2$ to show that, for all $\varphi\in L^2(Q_T)$, 
\begin{equation}\label{StrongYoungWeak}
\tilde\E\left|\<\Psi_n-\Psi,\varphi\>_{L^2(Q_T)}\right|^2\to 0,
\end{equation}
and \eqref{RandomCVOrlicz} with $r=2$ to obtain 
\begin{equation}
\tilde\E\left|\|\Psi_n\|_{L^2(Q_T)}^2-\|\Psi\|_{L^2(Q_T)}^2\right|=0,
\end{equation}
when $n\to+\infty$. Using the separability of $L^2(Q_T)$, we deduce that there is a subsequence of $(\Psi_n)$, which we still denote by $(\Psi_n)$, such that, $\tilde\P$-almost surely, $\Psi_n\to\Psi$ in weak-$L^2(Q_T)$ and $\|\Psi_n\|_{L^2(Q_T)}\to\|\Psi\|_{L^2(Q_T)}$. This implies $\Psi_n\to\Psi$ in $L^2(Q_T)$-strong. In particular, up to a subsequence, we have $\Psi_n\to\Psi$ for almost all $(\tilde\omega,x,t)$. We use the estimate $\tilde\E\|\Psi_n\|_{L^s(Q_T)}^s<+\infty$, $s>2$, to conclude by the Vitali Theorem that $\Psi_n\to\Psi$ in $L^2(\tilde\Omega\times Q_T)$-strong.
\end{proof}
\subsubsection{Martingale solution}\label{subsec:martingalesol}

Let us apply Proposition~\ref{prop:nlinUepsStrong} to the entropy-entropy flux pair associated to the affine function $g\colon\xi\mapsto\alpha\xi+\beta$. Then $\eta(\U)=\alpha q+\beta\rho$. We deduce that
\begin{equation}\label{cvUepsn}
\tilde{\U}_{\eps_n}\to \tilde\U
\end{equation} 
in $L^2(\tilde\Omega\times Q_T)$ strong  and a.s. as $L^2(Q_T)$-valued random variables.\medskip

For the moment we have only supposed that $\U_{\eps 0}\in W^{2,2}(\T)$ with some uniform bounds. Assume furthermore
\begin{equation}\label{cvCI}
\lim_{\eps\to0}\U_{\eps 0}=\U_0\quad \mbox{ in }L^2(\T)
\end{equation}
and a.e. Since $\U_{\eps 0}$ avoids the vacuum ($\rho_{\eps 0}\geq c_{\eps 0}>0$ a.e.), the velocity $u_{\eps 0}=\frac{q_{\eps 0}}{\rho_{\eps 0}}$ is well defined. We assume also the convergence
\begin{equation}\label{cvCIu}
\lim_{\eps\to0}u_{\eps 0}=u_0\quad \mbox{ in }L^2(\T)
\end{equation}
and a.e. This means in particular that, for a.e. $x$ in the set $\{\rho_0=0\}$, $q_0(x)=0$. Let $g\in C^2(\R)$ be a convex subquadratic function (\textit{i.e.} $g\in\mathcal{G}$, as defined in \eqref{gsubquad}). If \eqref{initEntropyAddCont} is uniformly bounded, then we can apply the dominated convergence Theorem to obtain 
\begin{equation}\label{cvCIS}
\lim_{\eps\to0}\eta(\U_{\eps 0})=\eta(\U_0)\quad \mbox{ in }L^2(\T),
\end{equation}
for any $\eta$ defined by \eqref{entropychi}.\medskip 

Recall that $(\tilde\Omega,\tilde\P,\tilde\F,\tilde{W})$ is given by Proposition~\ref{prop:cvrandomY}. Let $(\tilde{\mathcal{F}}_t)$ be the $\tilde{\P}$-augmented canonical filtration of the process $(\tilde{\U},\tilde{W})$, \textit{i.e.}
$$
\tilde{\mathcal{F}}_t=\sigma\big(\sigma\big(\varrho_t\tilde{\U},\varrho_t\tilde{W}\big)\cup\big\{N\in\tilde{\mathcal{F}};\;\tilde{\P}(N)=0\big\}\big),\quad t\in[0,T],
$$
where the restriction operator $\varrho_t$ is defined in \eqref{restr}. We will show that the sextuplet 
$$
\big(\tilde{\Omega},\tilde{\mathcal{F}},(\tilde{\mathcal{F}}_t),\tilde{\P},\tilde{W},\tilde{\U}\big)
$$
is a weak martingale solution to \eqref{stoEuler}.\medskip

Our aim is to pass to the limit in the balance entropy equation~\eqref{Itoentropyeps}. Actually, given \eqref{strongUepsU}, it would be more natural to pass to the limit in the weak-in-time formulation of \eqref{Itoentropyeps}, which is the following one: almost surely, for all $\varphi\in C^2(\overline{Q_T})$ such that $\varphi\equiv 0$ on $\T\times\{t=T\}$,
\begin{align}
&\iint_{Q_T} \left[\eta({\U_\eps})\partial_t\varphi+H({\U_\eps})\partial_x\varphi+\eps\eta(\U_\eps)\partial_{xx}^2\varphi \right] dx dt
+\int_{\T}\eta(\U_{\eps0})\varphi(0)dx\nonumber\\
+&\int_0^T\int_{\T} \eta'({\U_\eps})\mathbf{\Psi}^{\eps}({\U_\eps})\varphi\, dx d W(t)
+ \frac{1}{2}\iint_{Q_T}\GG^{\eps}({\U_\eps})^2\partial^2_{qq} {\eta}({\U_\eps})\varphi dx dt\nonumber\\
=&\iint_{Q_T} \eps \eta''({\U_\eps})\cdot(\partial_x\U_\eps,\partial_x\U_\eps)\varphi dx dt.\label{Itoentropyepsweak}
\end{align}
However, we need to work on the processes to pass to the limit in the stochastic integral with the martingale formulation of \eqref{Itoentropyeps}. Therefore, let $\varphi_0\in C^2(\T)$ be fixed. Since
$$
t\mapsto \big\langle \eta({\tilde\U_{\eps_n}}(t)),\varphi_0\big\rangle
$$
converges to $t\mapsto \big\langle \eta(\tilde\U(t)),\varphi_0\big\rangle$ in $L^1(\tilde\Omega\times(0,T))$, we can assume, up to a subsequence (and using the Fubini Theorem), that for a.e. $t\in[0,T]$, almost surely, $\big\langle \eta({\U_{\eps_n}}(t)),\varphi_0\big\rangle$ converges to $\big\langle \eta(\tilde\U(t)),\varphi_0\big\rangle$. 
Therefore there is a Borel subset
$\mathcal{D}$ of $[0,T]$ of full measure such that, for every $t\in\mathcal{D}$, almost surely, we have the convergence
\begin{multline*}
\big\langle \eta(\tilde{\U}_{\eps_n})(t),\varphi_0\big\rangle-\big\langle \eta({\U_{\eps_n}}_0),\varphi_0\big\rangle-
\int_0^t \big\langle H(\tilde{\U}_{\eps_n}),\partial_x \varphi_0\big\rangle+{\eps_n}\big\langle\eta(\tilde{\U}_{\eps_n}),\partial^2_{xx} \varphi_0\big\rangle\,d s\\
\to 
\big\langle \eta(\tilde{\U})(t),\varphi_0\big\rangle-\big\langle \eta(\U_0),\varphi_0\big\rangle-\int_0^t\big\langle H(\tilde{\U}),\partial_x \varphi_0\big\rangle\,d s
\end{multline*}
Note that, by \eqref{cvCIS}, we have $0\in\mathcal{D}$. Furthermore, by Proposition~\ref{prop:cvrandomY}, we have, for every $\varphi\in C_b(\overline{Q_T})$, almost-surely,
\begin{equation}\label{cveepsnoY}
\<\tilde e^{\eps_n},\varphi\>_{\mathcal{M}_b(\overline{Q_T}),C_b(\overline{Q_T})}\to \<\tilde e,\varphi\>_{\mathcal{M}_b(\overline{Q_T}),C_b(\overline{Q_T})}.
\end{equation}
By the the bound on $(\tilde e^{\eps_n})$ in $L^2(\tilde\Omega;\mathcal{M}_b(\overline{Q_T}))$ (see Proposition~\ref{prop:bounde}) and the Vitali Theorem, we have the convergence \eqref{cveepsnoY} in $L^\delta(\tilde\Omega)$ for every $\delta\in[1,2)$. It follows that, for every $Y\in L^2(\tilde\Omega)$, 
for every $\varphi\in C_b(\overline{Q_T})$,
$$
\tilde\E(\<\tilde e^{\eps_n},\varphi\>_{\mathcal{M}_b(\overline{Q_T}),C_b(\overline{Q_T})}Y)\to \tilde\E(\<\tilde e,\varphi\>_{\mathcal{M}_b(\overline{Q_T}),C_b(\overline{Q_T})}Y).
$$
Let $\mathfrak{A}$ denote the countable set of the atoms of the non-negative measure $\E\tilde e$. Let $\mathfrak{A}^*=\mathfrak{A}\setminus\{0\}$. Replace 
$\mathcal{D}$ by $\mathcal{D}\setminus\mathfrak{A}^*$. Then $\mathcal{D}$ remains a set of full measure in $[0,T]$ containing $t=0$ and, for every $t\in\mathcal{D}$, for every $\varphi\in C(\T)$, we have
\begin{equation}\label{cvTeepse}
\tilde\E\left(\iint_{\overline{Q_T}}\mathbf{1}_{[0,t)}\varphi d\tilde e^{\eps_n}\ Y\right)
\to \tilde\E\left(\iint_{\overline{Q_T}}\mathbf{1}_{[0,t)}\varphi d\tilde e\ Y\right).
\end{equation}

Let
\begin{align*}
\tilde M^\eps(t)=&\big\langle \eta(\tilde{\U}_{\eps})(t),\varphi_0\big\rangle-\big\langle \eta({\U_{\eps}}_0),\varphi_0\big\rangle-\int_0^t \big\langle H(\tilde{\U}_{\eps}),\partial_x \varphi_0\big\rangle dx\\
&-\int_0^t {\eps_n}\big\langle\eta(\tilde{\U}_{\eps}),\partial^2_{xx} \varphi_0\big\rangle\,d s
+\iint_{\overline{Q_T}}\mathbf{1}_{[0,t)}\varphi_0 d\tilde e_{\eps},
\end{align*}
and
\begin{equation*}
\tilde M(t)=\big\langle \eta(\tilde{\U})(t),\varphi_0\big\rangle-\big\langle \eta(\U_0),\varphi_0\big\rangle-\int_0^t\big\langle H(\tilde{\U}),\partial_x \varphi_0\big\rangle\,d s
+\iint_{\overline{Q_T}}\mathbf{1}_{[0,t)}\varphi_0 d\tilde e.
\end{equation*}
For every $t\in\mathcal{D}$, for every $Y\in L^2(\tilde\Omega)$, we have 
\begin{equation}\label{cvMartingaleD}
\tilde\E\left(\tilde M^{\eps_n}(t) Y\right)\to\tilde\E\left(\tilde M(t) Y\right).
\end{equation}
With the result of convergence \eqref{cvMartingaleD} at hand, we will prove now that $\tilde M(t)$ is a stochastic integral with respect to $\tilde W$. The argumentation is very similar to the argumentation in Section~\ref{subsec:identif}. First, there exists some independent $(\tilde{\mathcal{F}}_t)$-adapted Wiener processes $(\tilde\beta_k(t))$ such that 
$$
\tilde W=\sum_{k\geq 1}\tilde\beta_k(t) e_k
$$
almost surely in $\mathcal{X}_W$: the proof is analogous to the proof of Lemma~\ref{lem:tildeW}. In analogy with Lemma~\ref{lem:tildeM} then, we can show that  
the processes 
\begin{equation}\label{tildeMD}
\tilde{M},\;\tilde{M}^2-\sum_{k\geq1}\int_0^\tec\big\langle \sigma_k(\tilde{\U})\partial_q\eta(\tilde{\U}),\varphi\big\rangle^2\,d r,\;
\tilde{M}\tilde{\beta}_k-\int_0^\tec\big\langle\sigma_k(\tilde{\U})\partial_q\eta(\tilde{\U}),\varphi\big\rangle\,d r
\end{equation}
are $(\tilde{\mathcal{F}}_t)$-martingales. There is however a notable difference between the result of Lemma~\ref{lem:tildeM} and the result \eqref{tildeMD} here, in the fact that the martingales in \eqref{tildeMD} are indexed by $\mathcal{D}\subset[0,T]$ since we have used the convergence \eqref{cvMartingaleD}. This means that
$$
\tilde\E(\tilde M(t)-\tilde M(s)|\tilde{\mathcal{F}}_s)=0
$$
is satisfied only for $s\leq t$ and $s,t\in\mathcal{D}$, and similarly for the other martingales in \eqref{tildeMD}. If all the processes in \eqref{tildeMD} were continuous martingales indexed by $[0,T]$, we would infer, as in the proof of Proposition \ref{prop:martsoleps}, that 
\begin{align}
\big\langle \eta(\tilde{\U})(t),\varphi_0\big\rangle&-\big\langle \eta(\U_0),\varphi_0\big\rangle-\int_0^t\big\langle H(\tilde{\U}),\partial_x \varphi_0\big\rangle\,d s\nonumber\\
&=-\iint_{\overline{Q_T}}\mathbf{1}_{[0,t)}\varphi_0 d\tilde e+\sum_{k\geq 1}\int_0^t \big\langle\sigma_k(\tilde{\U})\partial_q\eta(\tilde{\U}),\varphi_0\big\rangle\,d\tilde{\beta}_k(s),\label{preEntropy}
\end{align}
for all $t\in[0,T]$, $\tilde{\P}$-almost surely. Nevertheless, $\mathcal{D}$ contains $0$ and is dense in $[0,T]$ since it is of full measure, and it turns out, by the Proposition~A.1 in \cite{Hofmanova13b} on densely defined martingales,
that this is sufficient\footnote{indeed, it is possible to prove the equivalent equations to \eqref{EqA3Martina}-\eqref{EqA3Martina2} for all $s,t\in\mathcal{D}$} to obtain \eqref{preEntropy} for all $t\in\mathcal{D}$, $\tilde{\P}$-almost surely. Then we conclude as in the \textit{proof of Theorem~4.13} of \cite{Hofmanova13b}: let $N(t)$ denote the continuous semi-martingale defined by
\begin{equation*}
N(t)=\int_0^t\big\langle H(\tilde{\U}),\partial_x \varphi_0\big\rangle\,d s+\sum_{k\geq 1}\int_0^t \big\langle\sigma_k(\tilde{\U})\partial_q\eta(\tilde{\U}),\varphi_0\big\rangle\,d\tilde{\beta}_k(s).
\end{equation*}
Let $t\in(0,T]$ be fixed and let $\alpha\in C^1_c([0,t))$. By the It\={o} Formula we compute the stochastic differential of $N(s)\alpha(s)$ to get
\begin{align}
0=\int_0^t N(s)\alpha'(s) ds+\int_0^t &\big\langle H(\tilde{\U}),\partial_x \varphi_0\big\rangle\alpha(s)\,d s\nonumber\\
&+\sum_{k\geq 1}\int_0^t \big\langle\sigma_k(\tilde{\U})\partial_q\eta(\tilde{\U}),\varphi_0\big\rangle\alpha(s)\,d\tilde{\beta}_k(s).\label{preEntropy2}
\end{align}
By \eqref{preEntropy}, we have
$$
N(t)=\big\langle \eta(\tilde{\U})(t),\varphi_0\big\rangle-\big\langle \eta(\U_0),\varphi_0\big\rangle+\iint_{\overline{Q_T}}\mathbf{1}_{[0,t)}\varphi_0 d\tilde e,
$$
for all $t\in\mathcal{D}$, $\tilde{\P}$-almost surely. In particular, by the Fubini Theorem,
\begin{align}
\int_0^t N(s)\alpha'(s) ds=\int_0^t \big\langle &\eta(\tilde{\U})(s),\varphi_0\big\rangle\alpha'(s)\,ds\nonumber\\
&+\big\langle \eta(\U_0),\varphi_0\big\rangle\alpha(0)
-\int_{[0,t]}\alpha(\sigma)d\tilde{\rho}(\sigma),\label{preEntropy3}
\end{align}
$\tilde{\P}$-almost surely, where we have defined the non-negative measure $\tilde{\rho}$ by 
$$
\tilde{\rho}(B)=\iint_{\overline{Q_T}}\mathbf{1}_{B}\varphi_0 d\tilde e,
$$ 
for $B$ a Borel subset of $[0,T]$. If $\alpha,\varphi_0\geq 0$, then
$$
\int_{[0,t]}\alpha(\sigma)d\tilde{\rho}(\sigma)\geq 0,\quad\tilde{\P}-\mbox{almost surely},
$$
and we deduce \eqref{Entropy} from \eqref{preEntropy2}, \eqref{preEntropy3}. This concludes the proof of Theorem~\ref{th:martingalesol}.

\section{Conclusion}\label{sec:conclusion}

We want to discuss in this concluding section some open questions related to the long-time behaviour of solutions to \eqref{stoEuler}. It is known that for \textit{scalar} stochastic conservation laws with additive noise, and for non-degenerate fluxes, there is a unique ergodic invariant measure, \textit{cf.} \cite{EKMS00,DebusscheVovelle14}. Since both fields of \eqref{stoEuler} are genuinely non-linear, a form of non-degeneracy condition is clearly satisfied in \eqref{stoEuler}. Actually, in the deterministic case $\Phi\equiv0$, the solution converges to the constant state determined by the conservation of the two invariants
\begin{equation}\label{invariantDET}
\int_0^1 \rho(x)dx,\quad \int_0^1 q(x) dx.
\end{equation}
see \cite[Theorem 5.4]{ChenFrid99}. This indicates that some kind of dissipation effects (\textit{via} interaction of waves, \textit{cf.} also \cite{GlimmLax70}) occur in the Euler system for isentropic gas dynamics. However, in a system there is in a way  more room for waves to evolve than in a scalar conservation law, and the long-time behaviour in \eqref{stoEuler} may be different from the one described in \cite{EKMS00,DebusscheVovelle14}. 
\medskip

Specifically, consider the case $\gamma=2$. For such a value the system of Euler equations for isentropic  gas dynamics is equivalent to the following Shallow water system:
\begin{subequations}\label{SaintVenant}
\begin{align}
&h_t+ \partial_x (h u)dt=0, &\mbox{ in }Q_T,\label{height}\\
&(h u)_t+\partial_x (h u^2+g\frac{h^2}{2})+gh  \partial_x Z=0,&\mbox{ in }Q_T,\label{charge}
\end{align}
\end{subequations} 
with $Z(x,t)=\Phi^*(x)\frac{dW}{dt}$ and $Q_T=\T\times(0,T)$. For example, we may take
\begin{equation}\label{exampleZ}
dZ(x,t)=\sum_{k\in\N}\sigma_k\left[\cos(2\pi k x)d\beta_k^\flat(t)+\sin(2\pi k x)d\beta_k^\sharp(t)\right],
\end{equation}
with $\sigma\in l^2(\N)$ and $\beta_k^\flat(t)$, $\beta_k^\sharp(t)$ some independent Brownian motions on $\R$ (\eqref{exampleZ} is an example of space-homogeneous noise).\medskip

When $Z=Z(x)$, \eqref{SaintVenant} is a model for the one-dimensional flow of a fluid of height $h$ and speed $u$ over a ground described by the curve $z=Z(x)$ ($u(x)$ is the speed of the column of water over the abscissa $x$)\footnote{the fact that $u$ is independent on the altitude $z$ is admissible as long as $h$ is small compared to the longitudinal length $L$ of the channel, $L=1$ here, \textit{cf.} \cite{GerbeauPerthame01}}. For a random $Z$ as in \eqref{charge}, the system \eqref{SaintVenant} describes the evolution of the fluid in terms of $(h,u)$ when its behaviour is forced by the moving topography. Note that, for \textit{smooth} solutions to \eqref{SaintVenant}, with a noise given by \eqref{exampleZ}, the balance of Energy writes
\begin{equation}\label{ESaintVenant}
\frac{d\;}{dt}\E\int_\T \eta_E(\U(x,t))dx=\frac12\|\sigma\|_{l^2(\N)}^2\E\int_\T h(x,t) dx,\quad \eta_E(\U):=h\frac{u^2}{2}+g\frac{h^2}{2}.
\end{equation}
Since the total height $\int_\T h(x,t) dx$ is conserved in the evolution, the input of energy by the noise is done at \textit{constant} rate:
\begin{equation}\label{ESaintVenant2}
\frac{d\;}{dt}\E\int_\T \eta_E(\U(x,t))dx=\mathrm{Cst}=\frac12\|\sigma\|_{l^2(\N)}^2\E\int_\T h_0(x) dx.
\end{equation}

Of course, the equality is not satisfied \eqref{ESaintVenant}. We have  
\begin{equation}\label{inESaintVenant}
\frac{d\;}{dt}\E\int_\T \eta_E(\U(x,t))dx\leq\frac12\|\sigma\|_{l^2(\N)}^2\E\int_\T h_0(x) dx,
\end{equation}
as a consequence of entropy inequalities. In particular dissipation of energy occurs in shocks. Therefore, the question is to determine if an equilibrium in law (and which kind of equilibrium) for such a random process as the solution to \eqref{SaintVenant} can be reached when time goes to $+\infty$ as a result of the balance between production of energy in the stochastic source term and dissipation of energy in shocks. An hint for the existence of a unique, ergodic, invariant measure is the ``loss of memory in the system" given by the ergodic theorem: if $f$ is a bounded, continuous functional of the solution $\U(t)$, then
\begin{equation}\label{ergodicT}
\lim_{T\to+\infty}\frac1T\int_0^T f(\U(t)) dt\to \<f,\mu\> \mbox{ a.s.}
\end{equation}
where $\mu$ is the invariant measure. Before testing the ergodic convergence \eqref{ergodicT}, one has first to restrict the evolution to the right manifold. Indeed, in the scalar case \cite{EKMS00,DebusscheVovelle14}, say for the equation
$$
dv+\partial_x (A(v))=\partial_x\phi(x)dW(t),\quad x\in\T,t>0,
$$
there is a unique invariant measure $\mu_\lambda$ indexed by the constant parameter
$$
\lambda=\int_{\T} v(x)dx\in\R.
$$
For \eqref{SaintVenant}, the entropy solution is evolving on the manifold
$$
\int_{\T} h(x)dx=\mathrm{cst}.
$$
Since $\E\int_0^t h(s) d\beta_k^\flat(s)=\E\int_0^t h(s) d\beta_k^\sharp(s)=0$ for all $k$ (this is the expectancy of a stochastic integral), we have a second equation of conservation by \eqref{charge}:
$$
\E\int_{\T} q(x)dx=0.
$$
It seems therefore that the final equilibrium and the invariant measure, if they exist, should be determined uniquely by the initial value of the parameters \eqref{invariantDET}. This is what we illustrate by numerical approximations on Figure~\ref{fig:testnum}.
\begin{figure}[h]
 \centering
 \includegraphics[scale=0.6]{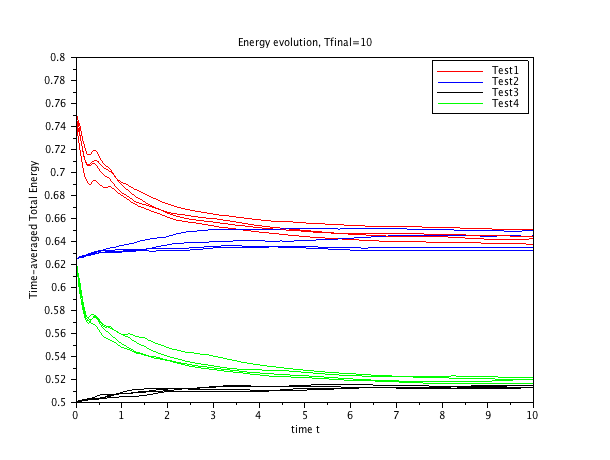}
\label{fig:testnum}
\end{figure}
On Figure~\ref{fig:testnum}, time is the abscissa coordinate, the averaged energy
$$
\frac{1}{t}\int_0^t\int_\T \eta_E(\U(x,s))ds
$$
is the ordinate coordinate. There are four different tests corresponding to four different initial conditions. The simulation on the time interval $[0,T]$, $T=10$, has been done several times, for several realizations of the noise therefore. The numerical values corresponding to each test are the following ones: first, we have taken $g=2$, $Z$ as in \eqref{exampleZ} with $\sigma_k=\mathbf{1}_{1\leq k\leq 5}$ and $h_0(x)\equiv 1$ in each four tests. The value of the initial velocity is then
$$
u_0(x)=\mathbf{1}_{0<x<1/2}\mbox{ \textcolor{red}{[Test 1]}},\quad u_0(x)=\frac12\mbox{ \textcolor{blue}{[Test 2]}},\quad u_0(x)=0\mbox{ \textcolor{green}{[Test 3]}},
$$
and
$$
u_0(x)=-\frac12\mathbf{1}_{0<x<1/2}+\frac12\mathbf{1}_{1/2<x<1}\mbox{ [Test 4]}.
$$

For the four test cases considered, the quantity $\int_\T h dx$ is the same of course and $\int_\T q dx$ has a common value in Tests 1-2 and 3-4 respectively. Observe indeed the common convergence in Tests 1-2 and 3-4. The proof of the existence of an invariant measure will be addressed in a future work.

\appendix
\section{A bound from below}\label{app:boundfrombelow}

\begin{definition} Let $\tau>0$. Let $\mathbf{1}_\mathrm{det}$ be the step function defined by \eqref{defhtau}. Let $u\in L^3(Q_T)$ and $\rho_0\in L^2(\T)$. A function $\rho\in C([0,T];L^2(\T))$ is said to be a generalized solution of the problem
\begin{equation}\label{appeqrho}
\frac12\partial_t\rho+\mathbf{1}_\mathrm{det}\left[\partial_x(\rho u)-\partial^2_x\rho\right]=0\mbox{ in }Q_T,
\end{equation}
with initial condition 
\begin{equation}\label{appICrho}
\rho(x,0)=\rho_0(x),\quad x\in\T,
\end{equation}
if 
\begin{equation}\label{spaceV}
\rho\in C([0,T];L^2(\T)),\quad \mathbf{1}_\mathrm{det}\rho\in L^2(0,T;H^1(\T)),
\end{equation}
and, for all $\varphi\in L^2(\T;H^1(0,T))$ with $\varphi(\cdot,T)=0$ such that 
$$
\mathbf{1}_\mathrm{det}\varphi\in L^2(0,T;H^1(\T)),
$$ 
one has
\begin{equation}\label{appweaksol}
\iint_{Q_T} \frac12\rho \partial_t \varphi+\mathbf{1}_\mathrm{det}\left[\rho u-\partial_x\rho\right]\partial_x\varphi\ dx dt+\frac12\int_\T \rho_0(x)\varphi(x,0)dx=0.
\end{equation}
\label{def:appweaksol}\end{definition}

This definition of solution to \eqref{appeqrho}-\eqref{appICrho} corresponds to the definition of generalized solutions 
in \cite[Eq.~(1.16), Chap III]{LadyzenskajaSolonnikovUralceva68}. The term 
$$
\iint_{Q_T} \mathbf{1}_\mathrm{det}\rho u \partial_x\varphi\ dx dt
$$
in \eqref{appweaksol} is well defined as we can see by using the H\"older inequality, which gives
$$
\iint_{Q_T} \left|\mathbf{1}_\mathrm{det}\rho u \partial_x \varphi\right|dxdt\leq \| \partial_x \varphi\|_{L^2(Q_T)}\|u\|_{L^3(Q_T)}\|\mathbf{1}_\mathrm{det}\rho\|_{L^6(Q_T)},
$$
and then using the following estimate: for all function $z$ satisfying the same condition \eqref{spaceV} as $\rho$,
\begin{equation}\label{appLE}
\|\mathbf{1}_\mathrm{det}z\|_{L^6(Q_T)}\leq C\left(\sup_{t\in[0,T]}\|z(t)\|_{L^2(\T)}\right)^{2/3}\|\mathbf{1}_\mathrm{det}\partial_x z\|_{L^2(Q_T)}^{1/3}+C\sup_{t\in[0,T]}\|z(t)\|_{L^2(\T)}.
\end{equation}
Let us recall the proof of \eqref{appLE}. We apply the Gagliardo-Nirenberg inequality, \cite[Formula (2.2)]{Nirenberg59} with $j=0$, $m=n=1$, $a=\frac{1}{3}$, $r=q=2$, using also \cite[Remark 5. p. 126]{Nirenberg59} with $\tilde q=2$, to $z(t)$, $t\in[0,T]$. This gives
$$
\|z(t)\|_{L^6(\T)}\leq C\|z(t)\|_{L^2(\T)}^{2/3}\|\partial_x z(t)\|_{L^2(\T)}^{1/3}+C\|z(t)\|_{L^2(\T)}.
$$
Then we multiply the result by $\mathbf{1}_\mathrm{det}(t)$ and we sum over $t\in[0,T]$. 
\medskip

Note also that, if $0\leq t_{2n}<T-\tau$ (where $t_k=k\tau$) and if $\varphi$ vanishes outside $(t_{2n},t_{2n+1})$, then, by \eqref{appweaksol}, we have
\begin{equation*}
\iint_{Q_{t_{2n},t_{2n+1}}} \frac12\rho \partial_t \varphi+\left[\rho u-\partial_x\rho\right]\partial_x\varphi\ dx dt=0.
\end{equation*}
Let $\rho_{2n}(x)=\rho(x,t_{2n})$. Taking $\varphi(x,t)=\psi(x,t)\min\left(\frac{t-t_{2n}}{h},1\right)$ where $h\in(0,\tau)$ and where $\psi\in H^1(Q_{t_{2n},t_{2n+1}})$ vanishes at $t=t_{2n+1}$, then letting $h\to 0$ (this is possible since $\rho$ is continuous at $t=t_{2n}$ with values in $L^2(\T)$), we obtain
\begin{equation*}
\iint_{Q_{t_{2n},t_{2n+1}}} \frac12\rho \partial_t\psi+\left[\rho u-\partial_x\rho\right]\partial_x\psi\ dx dt +\frac12 \int_\T \rho_{2n}(x)\psi(x,t_{2n}) dx=0.
\end{equation*}
This means that, in restriction to $Q_{t_{2n},t_{2n+1}}$, $\rho$ is a generalized solution to the problem
\begin{equation}\label{appeqrho2n}
\frac12\partial_t\rho+\left[\partial_x(\rho u)-\partial^2_x\rho\right]=0\mbox{ in }Q_{t_{2n},t_{2n+1}},
\end{equation}
with initial condition 
\begin{equation}\label{appICrho2n}
\rho(x,t_{2n})=\rho_{2n}(x),\quad x\in\T.
\end{equation}
Similarly, we show that $\rho(x,t)=\rho(x,t_{t_{2n+1}})$ for a.e. $x\in\T$, for all $t\in[t_{2n+1},t_{2n+2}]$. In particular, Problem~\eqref{appeqrho}-\eqref{appICrho} has a unique solution. Indeed, by \cite[Theorem~2.1, Chap III]{LadyzenskajaSolonnikovUralceva68}, we have
$$
\|\rho(t_{2n+1})\|_{L^2(\T)}\leq \sup_{t\in[t_{2n},t_{2n+1}]}\|\rho(t)\|_{L^2(\T)}+\|\partial_x\rho\|_{L^2(Q_{t_{2n},t_{2n+1}})}\leq c\|\rho(t_{2n})\|_{L^2(\T)},
$$
where $c$ depends on $\|u\|_{L^{3/2}(Q_T)}$ only. Since $\rho(t)$ is constant on intervals of the form $[t_{2n},t_{2n+1}]$, it follows that 
$$
\sup_{t\in[0,T]}\|\rho(t)\|_{L^2(\T)}\leq c^K\|\rho_0\|_{L^2(\T)},
$$
where $K$ is such that $T\leq K\tau$. In particular, $\rho=0$ if $\rho_0=0$. Introduce the notation
$$
t_\sharp:=\min(2t-t_{2n},t_{2n+2}),\quad t_\flat:=\frac{t+t_{2n}}{2},\quad t_{2n}\leq t< t_{2n+2}.
$$
Note that $(t_\sharp)_\flat=t$ if $t_{2n}< t< t_{2n+1}$ and that $(t_\flat)_\sharp=t$ for all $t$.  Set $u_\flat(x,t)=u(x,t_\flat)$. By uniqueness, we have 
\begin{equation}\label{apprhozeta}
\rho(x,t)=\zeta(x,t_\sharp)\mbox{ in }Q_T,
\end{equation}
where $\zeta\in C([0,T_\flat];L^2(\T))$ is the generalized solution of the problem
\begin{equation}\label{appeqzeta}
\partial_t\zeta+\partial_x(\zeta u_\flat)-\partial^2_x\zeta=0\mbox{ in }Q_{T_\flat},
\end{equation}
with initial condition 
\begin{equation}\label{appICzeta}
\zeta(x,0)=\rho_0(x),\quad x\in\T.
\end{equation}
Indeed, we start from
\begin{equation}\label{appweaksolzeta}
\iint_{Q_{T_\flat}} \zeta  \partial_t\psi+\left[\zeta u_\flat-\partial_x\zeta\right]\partial_x\psi\ dx dt +\int_\T \rho_0(x)\psi(x,0)dx =0,
\end{equation}
for all $\psi\in H^1(Q_{T_\flat})$ with $\psi(T_\flat)=0$. Let $\varphi\in L^2(\T;H^1(0,T))$ with $\varphi(\cdot,T)=0$ be such that 
$$
\mathbf{1}_\mathrm{det}\varphi\in L^2(0,T;H^1(\T)).
$$ 
Set $\psi(x,t):=\varphi(x,t_\flat)$. Then $\psi\in L^2(0,T_\flat;H^1(\T))$ and $\psi$ vanishes at $t=T_\flat$. We do not have $\psi\in L^1(\T;H^1(0,T_\flat))$ since $\psi$ has jumps at every points $t=t_{2n}$. However, an argument of approximation of the discontinuous function $s\mapsto s_\flat$ allows us to deduce from \eqref{appweaksolzeta} that
\begin{multline*}
\sum_{n}\int_\T \zeta(x,t_{2n+2})(\varphi(x,t_{2n+2})-\varphi(x,t_{2n+1})) dx\\
+\iint_{Q_{T_\flat}} \frac12\zeta(x,t) \partial_t\varphi(x,t_\flat)+\left[\zeta u_\flat-\partial_x\zeta\right](x,t)\partial_x\varphi(x,t_\flat)\ dx dt\\
+\int_\T \rho_0(x)\varphi(x,0)dx=0.
\end{multline*}
By a change of variable $t_\flat \mapsto t$ on every $(t_{2n},t_{2n+2})$, we obtain
\begin{multline*}
\sum_{n}\int_\T \zeta(x,t_{2n+2})(\varphi(x,t_{2n+2})-\varphi(x,t_{2n+1})) dx\\
+\sum_{n}\iint_{Q_{t_{2n},t_{2n+1}}} \zeta(x,t_\sharp) \partial_t\varphi(x,t)+2 \left[\zeta u_\flat-\partial_x\zeta\right](x,t_\sharp)\partial_x\varphi(x,t)\ dx dt\\
+\int_\T \rho_0(x)\varphi(x,0)dx=0.
\end{multline*}
Rewriting 
$$
\varphi(x,t_{2n})-\varphi(x,t_{2n-1})=\int_{t_{2n-1}}^{t_{2n}}\partial_t\varphi dt,
$$
we have
$$
\sum_{n}\int_\T \zeta(x,t_{2n+2})(\varphi(x,t_{2n+2})-\varphi(x,t_{2n+1})) dx
= \sum_{n}\iint_{Q_{t_{2n-1},t_{2n}}} \zeta(x,t_\sharp) \partial_t\varphi(x,t) dx dt.
$$
Furthermore, we use
$\ds \sum_{n}\iint_{Q_{t_{2n},t_{2n+1}}}  Z\ dx dt
= \sum_{n}\iint_{Q_{t_{2n},t_{2n+2}}} \mathbf{1}_\mathrm{det}Z\ dx dt$
with the function $Z(t,x)= \left[\zeta u_\flat-\partial_x\zeta\right](x,t_\sharp)\partial_x\varphi(x,t)$.
Using $u_\flat(t_\sharp)=u(t)$ shows that $(x,t)\mapsto\zeta(x,t_\sharp)$ satisfies \eqref{appweaksol}. Consequently, $\rho(x,t)=\zeta(x,t_\sharp)$ as asserted.

\begin{theorem}[Positivity] Let $\tau>0$. Let $\mathbf{1}_\mathrm{det}$ be the step function defined by \eqref{defhtau}. Let $u\in L^3(Q_T)$ and $\rho_0\in L^2(\T)$. Let $\rho\in C([0,T];L^2(\T))$ be the generalized solution of the problem \eqref{appeqrho}-\eqref{appICrho}. Assume $\rho_0\geq c_0$ a.e. in $\T$ where $c_0$ is a positive constant and let $m>6$. Then there exists a constant $c>0$ depending continuously on $c_0$, $T$, $m$ and
\begin{equation}\label{energiestopos}
\iint_{Q_T}\rho|\partial_x u|^2 dx dt\quad\mbox{and}\quad \|u\|_{L^m(Q_T)},
\end{equation}
such that 
\begin{equation}\label{rhopospos}
\rho\geq c
\end{equation} 
a.e. in $Q_T$.
\label{th:uniformpositive}\end{theorem}

\begin{proof} By \eqref{apprhozeta}, it is sufficient to consider the equation \eqref{appeqzeta} satisfied by $\zeta$. Note that $\zeta\in L^6(Q_{T_\flat})$ by \eqref{appLE} since $\zeta\in C([0,T_\flat];L^2(\T))$ and $\zeta\in L^2(0,T;H^1(\T))$. Since $u_\flat\in L^3(Q_{T_\flat})$, we have $\zeta u_\flat\in L^2(Q_{T_\flat})$. It follows from \eqref{appeqzeta} that $\zeta_t\in L^2_tH^{-1}_x$. Let $h\colon\R_+\to(0,+\infty)$ be a function of class $W^{2,\infty}$ and let $w=h(|\zeta|)$. Actually $\zeta$ is non-negative (see \eqref{zetapos} below), so $w=h(\zeta)$. We will use the function
\begin{equation}\label{functionh}
h(\zeta)=-\frac{\zeta}{\max(\zeta,r)^2}+\frac{2}{\max(\zeta,r)},
\end{equation}
where $r$ is a positive parameter, $r\in(0,1)$. We will prove an $L^\infty$-estimate on $w$ that is uniform in $r>0$. By passing to the limit $r\to 0$, this will give a bound from below on $\zeta$ and on $\rho$. By a chain-rule formula (\textit{cf.} Lemma 1.4 in Carrillo, Wittbold~\cite{CarrilloWittbold99} and Lemma~1.1 in Stampacchia~\cite{Stampacchia65} for example) we derive the fol\-lo\-wing equation for $w$:
\begin{equation}\label{appeqw}
\partial_t w-\partial^2_{xx} w=-\frac{2}{w}\mathbf{1}_{rw\leq 1}|\partial_x w|^2-\zeta h'(\zeta)\partial_x u_\flat-u_\flat\partial_x w.
\end{equation} 
Similarly, we have, for $p\geq 2$,
\begin{equation}\label{appeqwp}
\partial_t \frac{z^2}{p}-\partial^2_{xx} \frac{z^2}{p}
=-\frac{4}{p^2}\left(2\mathbf{1}_{rw\leq 1}+p-1\right)|\partial_x z|^2-\frac{\zeta h'(\zeta)}{w}z^2\partial_x u_\flat-\frac{u_\flat}{p}\partial_x z^2,
\end{equation} 
where $z:=w^{p/2}$. We will use \eqref{appeqwp} and an energy estimate to prove the bound
\begin{equation}\label{estimwp}
\sup_{t\in[0,T_\flat]}\|w(t)\|_{L^p(\T)}^p+\|\partial_x w^{p/2}\|_{L^2(Q_{T_\flat})}^2\leq C\|w(0)\|_{L^p(\T)}^p,
\end{equation}
where $C$ is a constant depending on $p$, $m$, $T$, $\|u\|_{L^m(Q_T)}$. Let us sum \eqref{appeqwp} on $\T$ (\textit{i.e.} we consider a test-function independent on $x$): we obtain
\begin{equation*}
\frac{d\;}{dt}\int_\T z^2 dx+\frac{4(p-1)}{p}\int_\T|\partial_x z|^2 dx
\leq-2\int_\T u_\flat z\partial_x z dx-\int_\T G(z)\partial_x u_\flat dx,
\end{equation*}
where we have introduced the function $G(z)$ defined by the implicit identity 
$$
G(z)=p\frac{\zeta h'(\zeta)}{w}z^2.
$$ 
By integration by parts, \cite[Lemma 1.1]{Stampacchia65}, we get
\begin{equation*}
\frac{d\;}{dt}\int_\T z^2 dx+\frac{4(p-1)}{p}\int_\T|\partial_x z|^2 dx
\leq\int_\T |u_\flat||\partial_x z|\left[2|z|+|G'(z)|\right] dx.
\end{equation*}
It is easy to check that $|G'(z)|\leq (2p+2)|z|$. Consequently, we have
$$
\frac{d\;}{dt}\int_\T z^2 dx+\frac{2(p-1)}{p}\int_\T|\partial_x z|^2 dx \leq \frac{2p(p+2)^2}{(p-1)}\int_\T u_\flat^2 z^2 dx.
$$
Integrating then over $t\in[0,\sigma]$ where $\sigma\leq T_\flat$, we obtain
$$
U_\sigma\leq\frac{2p(p+2)^2}{(p-1)} \iint_{Q_\sigma}u_\flat^2 z^2 dx dt+\|z(0)\|_{L^2(\T)}^2,
$$
where
$$
U_\sigma:=\sup_{t\in[0,\sigma]}\|z(t)\|_{L^2(\T)}^2+\frac{2(p-1)}{p}\|\partial_x z\|_{L^2(Q_\sigma)}^2.
$$
By the H\"older Inequality, it follows that
\begin{equation}
U_\sigma\leq \frac{2p(p+2)^2}{(p-1)} \|u_\flat\|_{L^3(Q_\sigma)}^2\|z\|_{L^6(Q_\sigma)}^2+\|z(0)\|_{L^2(\T)}^2.\label{estimUT}
\end{equation}
To obtain an estimate on the right hand-side of \eqref{estimUT}, we apply \eqref{appLE} (without $\mathbf{1}_\mathrm{det}$). This gives 
$$
U_\sigma\leq C^2\ C(p) \|u_\flat\|_{L^3(Q_\sigma)}^2\ U_\sigma+\|z(0)\|_{L^2(\T)}^2,
$$
with $C(p)=\frac{4p(p+2)^2}{(p-1)} \left(\frac{2p}{(p-1)}\right)^{1/3}$,
and then, since $m>3$,
$$
U_\sigma\leq C^2\ C(p) \sigma^e \|u\|_{L^m(Q_T)}^2 U_\sigma+\|z(0)\|_{L^2(\T)}^2,\quad e:=\frac{2}{3}-\frac{2}{m}.
$$
Let $\sigma_0>0$ be defined by
$$
C^2\ C(p)\sigma_0^e \|u\|_{L^m(Q_T)}^2 =\frac12.
$$
For $\sigma<\sigma_0$, we obtain $U_\sigma\leq 2\|w^{p/2}(0)\|_{L^2(\T)}^2$,
where $\sigma_0>0$ is some constant depending only on $\|u\|_{L^m(Q_T)}$ and $p$. 
Since an estimate on $U_\sigma$ gives in turn an estimate  on $\|z(\sigma)\|_{L^2(\T)}^2=\|w^{p/2}(\sigma)\|_{L^2(\T)}^2$, we can 
iterate our procedure from $[0,\sigma_0]$ to $[\sigma_0,2\sigma_0]$ and so on to deduce the  bound \eqref{estimwp}
with $\ds C=2^N$ with $N$ an integer such that $N\sigma_0 \geq T_\flat$.\medskip

Note that the energy estimates argument we used can be adapted to show that 
\begin{equation}\label{zetapos}
\zeta\geq 0\quad\mbox{a.e. in }Q_T
\end{equation}
(simply work on the equation satisfied by $\zeta_-$, the negative part of $\zeta$).\medskip

In the second step of the proof, we will derive the following $L^\infty$ estimate on $w$:
\begin{equation}\label{Linftyw}
\|w\|_{L^\infty(Q_{T_\flat})}\leq C\left(T,c_0,\|u\|_{L^m(Q_T)},\|\rho^{1/2}\partial_x u\|_{L^2(Q_T)}\right).
\end{equation}
To prove \eqref{Linftyw}, we use the equation \eqref{appeqw}. It is classical \cite{Ball77} that the weak solution $w$ is also a mild solution to \eqref{appeqw}:
$$
w(t)=S(t)w(0)+\int_0^{t}S(t-s)f(s)ds,
$$
where $f$ is the right hand-side of \eqref{appeqw}. Since
$$
f\leq|\zeta h'(\zeta)||\partial_x u_\flat| -u_\flat\partial_x w,
$$
we obtain
$$
0\leq w(t)\leq S(t)w(0)+W_1(t)+W_2(t),
$$
with
\begin{align*}
W_1(t)&=\int_0^{t}S(t-s)(|\zeta h'(\zeta)||\partial_x u_\flat|)(s)ds,\\
W_2(t)&=-\int_0^{t}S(t-s)(u_\flat\partial_x w)(s)ds.
\end{align*}
Let us set $g=\zeta^{1/2}|\partial_x u_\flat|$. We check on \eqref{functionh} that 
\begin{equation}\label{hw32}
\zeta^{1/2} |h'(\zeta)|\leq h(\zeta)^{3/2}=w^{3/2}
\end{equation}
for all $\zeta\geq 0$. Indeed, $\zeta^{1/2} |h'(\zeta)|= h(\zeta)^{3/2}$ if $\zeta\geq r$. If $0\leq\zeta\leq r$, then $\zeta^{1/2} |h'(\zeta)|=\frac{\zeta^{1/2}}{r^2}\leq \frac{1}{r^{3/2}}$ and $h(\zeta)\geq \frac{1}{r}$. By \eqref{hw32}, we have
$$
W_1(t)\leq W_3(t):=\int_0^{t}S(t-s)(w^{3/2}g)(s)ds.
$$
Let $q\in[1,+\infty)$, $\nu\in[1,2)$ be given. By \eqref{dxz} with $j=0$, we have
\begin{align}
\|W_3\|_{L^\infty(Q_{T_\flat})}&\leq C\|w^{3/2}g\|_{L^\nu(Q_{T_\flat})},\quad \frac{1}{\nu}<\frac{2}{3},\label{pkpkplus1}\\
&\leq C \|w\|_{L^{q}(Q_{T_\flat})}^{3/2}\|g\|_{L^2(Q_{T_\flat})},\label{resultpkpkplus1}
\end{align}
provided $q$ and $\nu$ satisfy the relation
\begin{equation}\label{pkrk}
\frac12+\frac{3}{2 q}=\frac{1}{\nu}.
\end{equation}
Using the energy estimate \eqref{estimwp} with $p=2$ and \eqref{dxz}, we have also
\begin{align}
\|W_2\|_{L^\infty(Q_{T_\flat})}&\leq C\|u_\flat\partial_x w\|_{L^\nu(Q_{T_\flat})},\quad \frac{1}{\nu}<\frac{2}{3},\label{W21}\\
&\leq C \|w(0)\|_{L^{2}(\T)}\|u\|_{L^m(Q_T)}\nonumber\\
&\leq \frac{C}{c_0}\|u\|_{L^m(Q_T)},\label{W22}
\end{align}
provided 
\begin{equation}\label{pkrk2}
\frac12+\frac{1}{m}=\frac{1}{\nu}.
\end{equation}
Finally, we have
\begin{equation}\label{W0}
\|S(\cdot)w(0)\|_{L^\infty(Q_{T_\flat})}\leq\frac{1}{c_0}.
\end{equation}
If $m>6$, then there exists $\nu>1$ satisfying \eqref{W21}-\eqref{pkrk2}. If $q>\frac{3}{2}$, then there exists $r>1$ satisfying \eqref{pkpkplus1}-\eqref{pkrk}. Thus we use \eqref{estimwp} with $p>\frac32$. It follows from \eqref{resultpkpkplus1} that $\|W_3\|_{L^\infty(Q_{T_\flat})}\leq \frac{C}{c_0^{3/2}}\|g\|_{L^2(Q_{T_\flat})}$. By \eqref{W22}-\eqref{W0}, we conclude that $\|w\|_{L^\infty(Q_{T_\flat})}$ is bounded by a quantity depending only on $c_0$ and on the norms in \eqref{energiestopos}. This concludes the proof of Theorem~\ref{th:uniformpositive}.\end{proof}

\begin{remark}\label{rk:positivityfordet} Note also that it is possible to give some precisions on the bound from below (I.57) in \cite{LionsPerthameSouganidis96}, regarding the positivity of the density $\rho$ in the deterministic parabolic approximation of the isentropic Euler system. Since, for such a system, the terms in \eqref{energiestopos} are bounded, respectively, by the initial entropy
$$
\int_\T \eta_E(\U_0(x))dx\leq C(\|\rho_0\|_{L^\infty(\T)},\|u_0\|_{L^\infty(\T)})
$$
and the $L^\infty$ norm
$$
\|u\|_{L^\infty(Q_T)}\leq T C(\|\rho_0\|_{L^\infty(\T)},\|u_0\|_{L^\infty(\T)}),
$$
where here $C$ is a continuous function of its arguments, we obtain $\rho\geq c_1$ a.e. in $Q_T$, where $c_1$ depends continuously on 
$T$, $\|\rho_0\|_{L^\infty(\T)}$, $\|u_0\|_{L^\infty(\T)}$, $c_0$, where $\ds c_0=\inf_{x\in\T}\rho_0(x)$.
\end{remark}

\section{Regularizing effects of the one-dimensional heat equation}\label{app:regparabolic}

In this section, we collect some results on the regularizing effects of the one-dimensional periodic heat equation: all the estimates below are very classical, but given for the heat equation on the whole line $\R$ in general. Since the proofs for the case of the circle are simple and short, we give them in full detail, see \cite{CazenaveHaraux98} for further results.\medskip 

\subsection{Heat semi-group}\label{app:Heatsemigroup}

Let us denote by $(S(t))$ the one-dimensional Heat semi-group associated to the Heat Equation
\begin{equation}\label{PerHeat}
(\partial_t-\partial^2_{xx})z=f,
\end{equation}
on $Q_T=\T\times(0,T)$. For some given data $z_0$ and $f$ (\textit{e.g.} integrable on $\T$ and $Q_T$ respectively), the mild solution in $C([0,T];L^1(\T))$ to \eqref{PerHeat} satisfying $z(0)=z_0$ is given by the formula
\begin{equation}\label{milddivpara}
z(t)=S(t)z_0+\int_0^t S(t-s)f(s)ds.
\end{equation}
Using either a spectral decomposition or working on $\R$ with periodic functions, we obtain 
$$
S(t)u(x)=K_t*u(x)=\int_\T K_t(y) u(x-y)dy,
$$
where the Kernel $K_t(x)$ is defined by
\begin{equation}\label{perHeatKernel}
K_t(x)=\sum_{n\in\Z}e^{-4\pi^2n^2 t}e_n(x)=\sum_{n\in\Z}G_t(x+n).
\end{equation}
Here $e_n$ is the $n$-th Fourier basis element on $\T$ and $G_t$ the heat kernel on $\R$:
$$
e_n(x):=e^{2\pi i nx},\quad G_t(x):=\frac{1}{(4\pi t)^{1/2}}e^{-\frac{|x|^2}{4t}}.
$$
By the second identity in \eqref{perHeatKernel}, we easily obtain for $p=1$ or $p=+\infty$ the estimate
\begin{equation}\label{partialKtp}
\|\partial_t^k\partial_x^j K_t\|_{L^p(\T)}\leq C(k,j,p) t^{-k-j/2-1/(2p')},
\end{equation}
for all $k,j\in\N$, $t>0$, where $p'$ is the conjugate exponent to $p$ and $C(k,j,p)$ is a constant depending on $k,j,p$ only. By interpolation between the cases $p=1$ and $p=+\infty$, we obtain \eqref{partialKtp} for all $p\in[1,+\infty]$. By the Young Inequality, we have, for $1\leq p\leq q$,
$$
\|S(t)\|_{L^p_x\to L^q_x}\leq \|K_t\|_{L^m(\T)},
$$
where $\frac1p+\frac1m=1+\frac1q$. It follows from \eqref{partialKtp} that
\begin{equation}\label{regSt1}
\|S(t)\|_{L^p_x\to L^q_x}\leq C(p,q)t^{-\frac{1}{2}\left(\frac{1}{p}-\frac1q\right)},
\end{equation}
for $1\leq p\leq q\leq +\infty$, for a given constant $C(p,q)$ and, more generally, 
\begin{equation}\label{regSr2}
\|\partial_t^k\partial_x^j S(t)\|_{L^p_x\to L^q_x}\leq C(p,q,k,j)t^{-\frac{1}{2}\left(\frac{1}{p}-\frac1q\right)-\frac{j}{2}-k},
\end{equation}
for $k,j\in\N$, $1\leq p\leq q\leq +\infty$.
We deduce from \eqref{regSr2} the following result.
\begin{lemma} Let $1\leq p\leq +\infty$, $j\in\N$, $f\in L^p(Q_T)$, $z_0\in L^p(\T)$ then
\begin{subequations}\label{dxHeat}
\begin{align}
\bigg\|\partial_x^j\int_0^t S(t-s)f(s)ds \bigg\|_{L^q(Q_T)}&\leq C \|f\|_{L^p(Q_T)}\quad\mbox{if}\quad \frac1q\leq\frac1p<\frac1q+\frac{2-j}{3},\label{dxz}\\
\big\|\partial_x^j S(t)z_0 \big\|_{L^q(Q_T)}&\leq C \|z_0\|_{L^p(\T)}\quad\mbox{if}\quad  \frac1q\leq \frac1p<\frac3q-j,\label{dxz0}
\end{align}\end{subequations}
where the constant $C$ depends on $p,q,j,T$.
\end{lemma}
\begin{proof} we have 
$$
\big\|\partial_x^j S(t)z_0 \big\|_{L^q(\T)}\leq C(p,q,j)\|z_0\|_{L^p(\T)} t^{-\mu},\quad\mu:=\frac{1}{2}\left(\frac{1}{p}-\frac1q\right)+\frac{j}{2},
$$
by \eqref{regSr2} if $p\leq q$. The right-hand side is in $L^q_t$ if $\mu<\frac1q$, \textit{i.e.} $\frac1p<\frac3q-j$. Similarly, 
$$
\big\|\partial_x^j S(t-s)f(s) \big\|_{L^q(\T)}\leq C(p,q,j)\|f(s)\|_{L^p(\T)} (t-s)^{-\mu},
$$
if $p\leq q$. Let $g(t)=t^{-\mu}\mathbf{1}_{t\in(0,T)}$, $h(t)=\|f(t)\|_{L^p(\T)}\mathbf{1}_{t\in(0,T)}$. By the Young Inequality for the convolution of functions, we have
$$
\bigg\|\partial_x^j\int_0^t S(t-s)f(s)ds \bigg\|_{L^q(Q_T)}\leq C(p,q,j)\|h\|_{L^p(0,T)}\|g\|_{L^m(0,T)},\quad \frac1p+\frac1m=1+\frac1q,
$$
and $\|g\|_{L^m(0,T)}<+\infty$ if, and only if, $m\mu<1$. This last condition is equivalent to 
$$
\frac{1}{2}\left(\frac{1}{p}-\frac1q\right)+\frac{j}{2}<1+\frac1q-\frac1p,
$$
\textit{i.e.} $\frac1p<\frac1q+\frac{2-j}{3}$.
\end{proof}

Let $J=(\mathrm{Id}-\partial_x^2)^{-1/2}$ and $s\in(0,1)$. Using the first identity in \eqref{perHeatKernel} (spectral decomposition), we have
$$
\|J^s S(t) u\|_{L^2(\T)}^2=\sum_{n\in\Z}(1+4\pi^2 |n|^2)^{s/2} e^{-8\pi^2 |n|^2 t} |\<u,e_n\>|^2,
$$
which gives
\begin{equation}\label{regSJ}
\|J^s S(t)\|_{L^2_x\to L^2_x}\leq Ct^{-\frac{s}{2}},
\end{equation}
where $C$ is a constant depending on $s$ only.\medskip

We finish this part with the proof of the following estimate~\eqref{toHolderU0}, that we will need in \eqref{HolderU0}. Let $u\in H^1(\T)$. Using the first identity in \eqref{perHeatKernel} (spectral decomposition), we have
\begin{align}
\|S(t)u-u\|_{L^2(\T)}^2=&\sum_{n\in\Z}|\<u,e_n\>|^2\left|1-e^{-4\pi n^2 t}\right|^2\nonumber\\
\leq &2 \sum_{n\in\Z}|\<u,e_n\>|^2\left|1-e^{-4\pi n^2 t}\right|\nonumber\\
\leq &8\pi \sum_{n\in\Z}|\<u,e_n\>|^2 n^2 t=\frac{2}{\pi}\|\nabla u\|_{L^2(\T)}^2 t.\label{toHolderU0}
\end{align}

\subsection{Fractional Sobolev space}\label{app:fracSobolev}

For $0< s<1$, $1<p<+\infty$, $1\leq q\leq+\infty$ we recall that we can define the Besov Space $B^s_{p,q}(\T)$ as a space of functions $u\in L^p(\T)$ such that 
\begin{equation}\label{DefSobolev}
[u]_{B^s_{p,q}(\T)}=\left(\int_{-1/2}^{1/2}\|\Delta_h u\|_{L^p(\T)}^q |h|^{-1-sp} dh\right)^{1/q}<+\infty,\quad \Delta_h u(x):=u(x+h)-u(x).
\end{equation}
Then we set $\|u\|_{B^s_{p,q}(\T)}=\|u\|_{L^p(\T)}+[u]_{B^s_{p,q}(\T)}$, see  Theorem~1.2.5 in Triebel, \cite{Triebel92}. \medskip

It is easy to show the algebra property
\begin{equation}\label{SobolevAlgebra}
\|uv\|_{B^s_{p,q}(\T)}\lesssim \|u\|_{L^\infty(\T)}\|v\|_{B^s_{p,q}(\T)}+\|v\|_{L^\infty(\T)}\|u\|_{B^s_{p,q}(\T)},
\end{equation}
for $u,v\in B^s_{p,q}(\T)\cap L^\infty(\T)$. Similarly, any $F\colon\R\to\R$ locally Lipschitz satisfying $F(0)=0$ operates on $B^s_{p,q}(\T)\cap L^\infty(\T)$:
\begin{equation}\label{Foperates}
\|F(u)\|_{B^s_{p,q}(\T)}\leq \mathrm{Lip}_R(F) \|u\|_{B^s_{p,q}(\T)},\quad R:=\|u\|_{L^\infty(\T)},
\end{equation}
where $\mathrm{Lip}_R(F)$ is the Lipschitz constant of $F$ in restriction to $[-R,R]$.\medskip

For $0< s<1$ and $1\leq p<+\infty$, we denote by $H^s_p(\T)$ the Bessel potential space of functions $u\in L^p(\T)$ such that $J^s u\in L^p(\T)$, where
$J=(\mathrm{Id}-\partial_x^2)^{-1/2}$, with the norm
$$
\|u\|_{H^s_p(\T)}=\|u\|_{L^p(\T)}+\|J^s u\|_{L^p(\T)}.
$$
We then have 
\begin{equation}\label{BesovBessel}
B^s_{22}(\T)=H^s_2(\T),
\end{equation}
see Equation (7) in Theorem~1.3.2 of \cite{Triebel92}. Actually the references we give in \cite{Triebel92} are for spaces of functions on $\R$, but the results are valid on $\T$, see Remark~4, paragraph~1.5.4 of \cite{Triebel92}. We denote by $W^{s,2}(\T)$ the space in \eqref{BesovBessel}, used in Proposition~\ref{prop:regboundedeps} for example.\medskip


\newcommand{\etalchar}[1]{$^{#1}$}
\def\ocirc#1{\ifmmode\setbox0=\hbox{$#1$}\dimen0=\ht0 \advance\dimen0
  by1pt\rlap{\hbox to\wd0{\hss\raise\dimen0
  \hbox{\hskip.2em$\scriptscriptstyle\circ$}\hss}}#1\else {\accent"17 #1}\fi}
  \def\cprime{$'$}
\providecommand{\bysame}{\leavevmode\hbox to3em{\hrulefill}\thinspace}
\providecommand{\MR}{\relax\ifhmode\unskip\space\fi MR }
\providecommand{\MRhref}[2]{%
  \href{http://www.ams.org/mathscinet-getitem?mr=#1}{#2}
}
\providecommand{\href}[2]{#2}

\end{document}